\documentclass[aos]{imsart}

\RequirePackage{amsthm,amsmath,amsfonts,amssymb}
\RequirePackage[authoryear]{natbib}%% uncomment this for author-year citations
\RequirePackage[colorlinks,citecolor=blue,urlcolor=blue]{hyperref}
\RequirePackage{graphicx}

\usepackage{comment}
\usepackage{bm}
\usepackage{subcaption}
\usepackage{enumitem}

\startlocaldefs
\theoremstyle{plain}

\newtheorem{Theorem}{Theorem}[section]
\newtheorem{Lemma}[Theorem]{Lemma}
\newtheorem{Corollary}[Theorem]{Corollary}
\theoremstyle{definition}

\newtheorem*{Example}{Example}
\newtheorem*{Remark}{Remark}
\newtheorem{Assumption}{Assumption}

\newcommand{\CL}{\mathcal{L}}
\newcommand{\CH}{\mathcal{H}}

\newcommand{\T}{\mathbb{T}}

\newcommand{\normH}[2]{\left\| #1 \right\|_{#2}}

\newcommand{\borelH}[1]{\mathcal{B}\left( #1 \right)}
\newcommand{\measH}[1]{\left( #1,\borelH{#1} \right)}

\newcommand{\defset}[1]{\left\{ #1 \right\}}
\newcommand{\innerprod}[3]{\left\langle #1,#2 \right\rangle_{#3}}

\newcommand{\GPmeasure}[2]{\mathcal{N}\left(#1,#2\right)}

\newcommand{\Aref}[1]{Assumption \ref{#1}}

\renewcommand{\theenumi}{\roman{enumi}}

\newcounter{proofstep}
\renewcommand{\theproofstep}{\arabic{proofstep}}

\newcommand{\proofstep}[1]{%
  \refstepcounter{proofstep}% カウンタを 1 増やす
  \par\medskip\noindent
  \textbf{Step \theproofstep\ (#1).}\quad%
}


\endlocaldefs

\begin{document}
\title{Exponential mixing properties of nonlinear functional autoregressive models}

\begin{aug}
%%%%%%%%%%%%%%%%%%%%%%%%%%%%%%%%%%%%%%%%%%%%%%%
%% Additional information (such as           %%
%% indicating the corresponding author) can  %%
%% be included in the Acknowledgments        %%
%% section if necessary.                     %%
\author[A]{\fnms{Shuntarou}~\snm{Suzuki}\ead[label=e1]{s-suzuki@sigmath.es.osaka-u.ac.jp}}
\author[A,B]{\fnms{Yoshikazu}~\snm{Terada}\ead[label=e2]{terada.yoshikazu.es@osaka-u.ac.jp}}
%%%%%%%%%%%%%%%%%%%%%%%%%%%%%%%%%%%%%%%%%%%%%%
%% Addresses                                %%
%%%%%%%%%%%%%%%%%%%%%%%%%%%%%%%%%%%%%%%%%%%%%%
\address[A]{Graduate School of Engineering Science, The University of Osaka, Osaka, Japan\printead[presep={,\ }]{e1,e2}}

\address[B]{Center for Advanced Integrated Intelligence Research, RIKEN}
\end{aug}

\begin{abstract}
The importance of functional data analysis has increased substantially in recent years.
In machine learning, nonlinear function regression based on deep neural networks is referred to as operator learning, 
and many of its applications involve functional time series data.
However, the theoretical understanding of nonlinear models in functional time series analysis remains limited, 
as most existing works focus on linear models.
In this paper, we derive basic properties for analyzing adaptive learning in nonlinear functional autoregressive (NFAR) models.
Specifically, we derive sufficient conditions for NFAR models to be exponentially mixing.
We provide an example with a Hammerstein operator under which these conditions are satisfied.
As an application of exponential mixing, we consider operator learning for NFAR models with Urysohn operators
and derive convergence rates for adaptive estimators based on deep neural networks.
\end{abstract}



\maketitle
\section{Introduction}\label{s1}

In recent years, advances in measurement technology have enabled the collection of data in which each observational unit is densely observed over a continuum, as in spectrometric or longitudinal studies.
Functional data analysis (FDA), introduced by \citet{ramsay1982when}, 
provides a systematic framework for analyzing such data.
In FDA, a functional time series is a time series whose observations are functions, 
arising in a wide range of applications.
For example, a year-long record of daily electricity demand yields a curve for each day, viewed as a function of time,
while a three-month record of daily temperature fields yields a surface for each day, viewed as a function of latitude and longitude.
The analysis of functional time series has attracted considerable attention
(\citealp{bosq2000linear}; \citealp{horvath2012inference}; \citealp{KokoszkaReimherr2017}).
% Representative contributions include \citet{bosq1991modelization}, 
% which proposed the functional autoregressive model 
% and provided an estimation method for the linear operator appearing in the model, 
% together with its consistency. 
% \citet{hormann2010weakly} introduced the $L^{p}$-$m$ approximation condition, 
% which quantifies temporal dependence in functional time series and 
% provides a theoretical basis for studying the asymptotic behavior of statistics constructed from functional time series.
% In addition, 
% \citet{aston2012detecting} studied the detection of change points in the mean function of functional time series. 
% Moreover, \citet{hormann2015dynamic} proposed dynamic FPCA, 
% a dimension reduction method specifically designed for functional time series. 
% For comprehensive references, 
% see \citet{bosq2000linear}, \citet{horvath2012inference}, and \citet{KokoszkaReimherr2017}.

One of the most important models for functional time series is the functional autoregressive (FAR) model.
Let $\mathcal{H} := \Bigl\{ f:[0,1]^{d}\to\mathbb{R}\ \Bigm|\ \int_{[0,1]^{d}} \lvert f(\bm{u}) \rvert^{2} d\bm{u} < \infty \Bigr\}$ and let $\{X_{t}\}_{t\in\mathbb{N}}$ be an $\mathcal{H}$-valued time series. 
The FAR model is defined by
\begin{align}
  X_{t+1} &= \Psi_{0}(X_{t}) + \xi_{t}, \quad t \in \mathbb{N}, \label{t_far}
\end{align}
where $\Psi_{0}$ is an operator from $\mathcal{H}$ to $\mathcal{H}$, 
and $\{\xi_{t}\}_{t\in\mathbb{N}}$ is a sequence of i.i.d.\ random elements in $\mathcal{H}$ with mean zero.
Most existing studies consider the case where $\Psi_{0}$ is a Hilbert-Schmidt operator.
% %When $\Psi_{0}$ is a Hilbert-Schmidt operator, a well-developed theory of statistical inference is available.
% For the estimation of $\Psi_{0}$, 
% \citet{bosq1991modelization} proposed a method combining the Yule-Walker approach with functional principal component analysis, 
% \citet{AntoniadisSapatinas2003} developed a wavelet-based estimator, 
% and \citet{KarginOnatski2008} proposed a method based on low-rank approximation via singular value decomposition. 
% Other contributions to statistical inference beyond estimation include \citet{horvath2010testing} and \citet{kokoszka2013determining}. 
Studies on the estimation of $\Psi_{0}$ include \citet{bosq1991modelization},
\citet{AntoniadisSapatinas2003}, and \citet{KarginOnatski2008}.
Statistical inference other than the estimation of $\Psi_{0}$, such as stability
testing and order determination, has also been actively studied; see, for example,
\citet{horvath2010testing} and \citet{kokoszka2013determining}.
However, classical FAR models are restricted to linear operators 
and thus have limited expressive power for temporal dynamics.
We consider a nonlinear functional autoregressive (NFAR) model in which $\Psi_{0}$ is a nonlinear operator.

NFAR models are closely related to operator learning, which has recently attracted considerable attention in machine learning. 
Operator learning aims to approximate (possibly nonlinear) operators via deep neural networks (DNNs).
For further details, we refer the reader to \citet{subedi2025operator}.
One practical advantage of these methods is that, after training, 
they can provide fast approximations to computationally expensive input-output maps between function spaces.
This feature is useful in the simulation and prediction of physical systems. 
A representative problem is to learn an operator that describes the time evolution of spatiotemporal data.
For example, \citet{Kissas2022Learning} proposed models based on Deep Operator Networks (DeepONets; \citealp{lu2021deeponet}) 
and Fourier neural operators (FNOs; \citealp{Li2021FNO}) for predicting flow past a cylinder, 
while \citet{pathak2022fourcastnet} applied an FNO-based architecture to global weather forecasting; see also \citet{choi2024applications}. 
Because these applications involve learning nonlinear maps from current or past spatial fields to future spatial fields, 
they can naturally be interpreted as learning nonlinear transition operators for functional time series. 
Accordingly, NFAR models are closely connected to practical applications of operator 
learning and are an important subject of study in both functional time series analysis and operator learning.
%However, in this case, the FAR model captures only relatively simple temporal dynamics. 
%We consider a functional autoregressive model with a nonlinear operator $\Psi_{0}$.
%In this work, we refer to the case where $\Psi_{0}$ is nonlinear as a nonlinear functional autoregressive (NFAR) model. 

% NFAR models are closely related to operator learning, 
% which has recently attracted considerable attention in machine learning.
% Operator learning aims to approximate (possibly nonlinear) operators via deep neural networks.
% Representative architectures include the Deep Operator Network (DeepONet) 
% proposed by \citet{lu2021deeponet} and the Fourier Neural Operator (FNO) proposed by \citet{Li2021FNO}. 
% Their advantages include the ability to learn mappings from input functions to output functions and fast inference.
% Owing to these practical advantages, operator learnings have been widely applied to the simulation and prediction of physical phenomena.
However, the theoretical foundations for estimating and learning $\Psi_{0}$ in NFAR models are still underdeveloped.
In the context of operator learning, 
recent studies have treated the problem of learning an unknown target operator from input--output pairs of functions 
as nonparametric regression between function spaces, 
and have derived generalization error bounds for the resulting estimators.
The available theory is largely restricted to independent and identically distributed observations,
 and temporally dependent data have received comparatively little attention.
For example, \citet{Reinhardt2024Operator} consider nonparametric regression between separable Hilbert spaces, 
a setting closely related to function-on-function regression in FDA.
They derive concentration inequalities for least-squares empirical risk minimizers over general operator classes.
Related studies have also developed generalization-error bounds for specific DNN-based operator estimators; 
see, for example, \citet{liu2024deep}.
These studies are also based on i.i.d.\ sampling schemes.
In the functional time series literature, only a few studies consider the estimation of $\Psi_{0}$ in NFAR models.
One representative study is \citet{ZhuPolitis2017}, 
which proposed a kernel estimator for $\Psi_{0}$ and established the asymptotic normality of the resulting estimator.
Their analysis, however, relies on the assumption that the mixing coefficients 
of the functional time series decay at a sufficiently fast rate.
As discussed below, verifying such mixing conditions for NFAR models is nontrivial and has not been systematically addressed. A detailed verification of this assumption was outside the scope of \citet{ZhuPolitis2017}.

% Mixing conditions refer to measures of the degree of dependence among observations in models with dependence, 
% such as time series models or spatial models. 
Mixing conditions are assumptions that characterize the strength of dependence among observations in dependent data settings, such as time series or spatial models.
For further details on mixing conditions, see \citet{rio2013inequalities}. 
In statistics, mixing conditions are applied to show asymptotic properties of nonparametric estimators, 
including deep learning-based methods, for samples with dependence
(\citealp{kreiss1998regression}; \citealp{tjostheim1994nonparametric}; \citealp{kurisu2024adaptive}).
Accordingly, the verification of mixing conditions has been studied extensively.
% For instance, \citet{chen2000geometric} gave sufficient conditions for exponential $\beta$-mixing properties of nonlinear AR models 
% in the finite dimensional case.
For instance, \citet{chen2000geometric} gave sufficient conditions under which nonlinear AR models are exponentially $\beta$-mixing in the finite-dimensional case.
Moreover, \citet{meyntweedie2009} and \citet{hairer2011harris} provided sufficient conditions for exponential $\beta$-mixing of a broader class of Markov chains.
However, it is nontrivial for stochastic processes taking values in infinite dimensional spaces to satisfy mixing conditions.  
This difficulty stems from the fact that probability measures on infinite dimensional spaces can be mutually singular.
For example, whether a stochastic partial differential equation (i.e., an infinite dimensional stochastic process) exhibits exponential mixing has typically been studied on a case-by-case basis (\citealp{Hairer2002ExponentialMixing,GoldysMaslowski2005ExponentialErgodicity}).  
Although broadly applicable conditions for guaranteeing exponential mixing in large classes of stochastic partial differential equations remain limited, similar difficulties also arise for NFAR models. 
While establishing sufficient conditions for NFAR models to be exponentially mixing is crucial for deriving the asymptotic behavior of estimators of $\Psi_{0}$ and the generalization error of learning methods, this problem has not been systematically addressed.

In this paper, we first establish sufficient conditions under which the FAR model \eqref{t_far} is exponentially $\beta$-mixing, 
including the case where $\Psi_{0}$ is nonlinear. 
As an application of this result, we derive an oracle inequality for operator learning methods targeting $\Psi_{0}$.
Furthermore, we consider 
the case in which $\Psi_{0}$ has the following form:
\begin{align}
     \Psi_{0}(x)(\bm{u})&= \int_{[0,1]^{d}}\psi_{0}(\bm{u},\bm{v},x(\bm{v}))d\bm{v},\quad \bm{u} \in [0,1]^{d},\,x \in \mathcal{H}. \nonumber
\end{align}
In this setting, $\Psi_{0}$ is referred to as a Urysohn operator and $\psi_{0}$ as a Urysohn kernel.
Urysohn operators are closely related to recent applications of operator learning to physical simulation. 
Indeed, solution operators for partial differential equations, such as the Burgers and Navier--Stokes equations, 
can often be represented or approximated, under suitable conditions, through integral equations or nonlocal integral operators. 
Such nonlinear integral operators can be naturally described within the Urysohn-operator framework. 
Thus, estimating the Urysohn kernel $\psi_0$, and hence the induced operator $\Psi_0$, 
is naturally connected to operator learning problems arising in physical simulation. 
This connection is also reflected in the neural integral equation framework of \citet{Zappala2024NeuralIntegralOperators}, 
where spatiotemporal data are formulated as solutions to integral equations and the integral terms are parameterized 
by neural networks and learned from data. Motivated by this connection, 
we propose an operator learning method for estimating $\Psi_0$ 
by approximating the Urysohn kernel $\psi_0$ with a DNN. We then use the derived oracle inequality to establish an upper bound on the generalization error of the proposed method.

In summary, the main contributions of this work are as follows:
\begin{itemize}
    \item We derive sufficient conditions under which NFAR models are exponentially $\beta$-mixing,
    including cases where $\Psi_{0}$ is nonlinear.
    \item As an application, we derive an oracle inequality for an estimator of $\Psi_{0}$ in \eqref{t_far} under the assumption that the NFAR model is exponentially $\beta$-mixing.
    \item We propose a novel operator learning model for NFAR models in which $\Psi_{0}$ is a Urysohn operator, and establish an upper bound on its generalization error. 
\end{itemize}
The rest of the paper is organized as follows. 
Section \ref{s2} introduces notation and assumptions for the nonlinear functional autoregressive model.
Section \ref{s3} provides sufficient conditions ensuring that NFAR models are exponentially $\beta$-mixing.
As an application, Section \ref{s4} develops an oracle inequality for estimators of $\Psi_{0}$ under suitable regularity and derives the corresponding convergence rates when the true operator is restricted to Urysohn operators 
and its Urysohn kernel is learned via ReLU  DNNs.
Moreover, Section \ref{s5} empirically investigates the generalization error via numerical experiments for the setting considered in Section \ref{s4}.
The appendix presents the proofs of the main results.

\section{Preliminaries}\label{s2}
\subsection{Notation}\label{s21}
% Let $\mathcal{H}$ be a separable Hilbert space.
% For any subset $A \subset \mathcal{H}$, we write $\overline{A}$ for its closure. 
% For an operator $\Phi : \mathcal{H} \to \mathcal{H}$, we write $\mathrm{Im}(\Phi)$ for its image.
% Let $\CL$ denote the space of all bounded linear operators from $\CH$ to $\CH$. 
% For $\Phi \in \mathcal{L}$, we denote its kernel and operator norm by 
% $\mathcal{K}(\Phi) := \{x \in \mathcal{H} \mid \Phi (x) = 0\}$ and 
% $\|\Phi\|_{\mathcal{L}} := \sup_{\left\| x \right\|_\CH = 1}\left\| \Phi (x) \right\|_{\mathcal{H}}$, respectively. 
% For any $f, g \in \mathcal{H}$, we define the function $f \otimes g : [0,1]^{d} \times [0,1]^{d} \to \mathbb{R}$ by
% $
% (f \otimes g)(u,v) = f(u)g(v), \,u,v \in [0,1]^{d},
% $
% and the operator $f \otimes g : \mathcal{H} \to \mathcal{H}$ by
% $
% (f \otimes g)(h) = f \langle g,h \rangle_{\mathcal{H}},\,h \in \mathcal{H}.
% $
For any subset $A \subset \mathcal{H}$, let $\overline{A}$ denote its closure.
For an operator $\Phi : \mathcal{H} \to \mathcal{H}$, let $\mathrm{Im}(\Phi)$ denote its image.
Let $\CL$ denote the space of all bounded linear operators from $\CH$ to $\CH$.
For $\Phi \in \mathcal{L}$, let
$\mathrm{Ker}(\Phi) := \{x \in \mathcal{H} \mid \Phi(x) = 0\}$ and
$\|\Phi\|_{\mathcal{L}} := \sup_{\|x\|_\CH = 1}\|\Phi(x)\|_{\mathcal{H}}$
denote its null space and operator norm, respectively.
% For any $f, g \in \mathcal{H}$, define the function
% $f \otimes g : [0,1]^{d} \times [0,1]^{d} \to \mathbb{R}$ by
% $
% (f \otimes g)(\bm{u},\bm{v}) = f(\bm{u})g(\bm{v}),\,\bm{u},\bm{v} \in [0,1]^{d},
% $ 
% and the operator $f \otimes g : \mathcal{H} \to \mathcal{H}$ by 
% $
% (f \otimes g)(h) = f \langle g,h \rangle_{\mathcal{H}},\,h \in \mathcal{H}
% $.
Let $(e_j)_{j\in \mathbb{N}}$ be an orthonormal basis of $\CH$ and let $\mathbb{T}$ denote the space of trace-class operators on $\mathcal{H}$:
    \begin{align}
        \T &:= \defset{
                T \in \CL \mid \sum_{j = 1}^{\infty}
                \innerprod{(TT^{*})^{\frac{1}{2}}e_{j}}{e_{j}}{\CH} < \infty 
        }, \nonumber
    \end{align}
    where $T^{*}:\mathcal{H} \to \mathcal{H}$ denotes the adjoint operator of $T$. 
For $T \in \mathbb{T}$, its trace norm is defined by 
    \begin{align}
        \normH{T}{\T}&:= \sum_{j = 1}^{\infty} \Bigl\langle (TT^{*})^{\frac{1}{2}}e_{j}, e_{j}\Bigr\rangle_{\mathcal{H}}. \nonumber
    \end{align}
For $R>0$, let $V_R:=\{x\in\mathcal{H}\mid \|x\|_{\mathcal{H}}\le R\}$ and define
\begin{align}
\mathcal{C}_b(\mathcal{H})
:=\Big\{\phi:\mathcal{H}\to\mathbb{R}\ \text{measurable}\ \Big|\ \sup_{x\in\mathcal{H}}|\phi(x)|<\infty\Big\},
\end{align}
equipped with the norm $\|\phi\|_{\mathcal{C}_b}:=\sup_{x\in\mathcal{H}}|\phi(x)|$.
Let $\borelH{\CH}$ be the Borel $\sigma$-algebra of $\CH$.
% For $\mu \in \mathcal{H}$ and $\mathcal{Q} \in \mathbb{T}$, we will denote by $\GPmeasure{\mu}{\mathcal{Q}}$ 
For $\mu \in \mathcal{H}$ and a self-adjoint, positive-definite $Q \in \mathbb{T}$,
we denote by $\GPmeasure{\mu}{Q}$ the Gaussian measure on $\measH{\CH}$ 
with mean $\mu$ and covariance operator $Q$.
% For $v \in \mathbb{R}^{r}$ and $i = 1,\dots,r$, we denote the $i$-th component by $v_{i}$ and define
% $
% \lvert v \rvert_{\infty} := \max_{1 \le i \le r} \lvert v_{i} \rvert
% $. 
For any $x \in \mathbb{R}$, let $\lfloor x \rfloor$ denote the integer satisfying $\lfloor x \rfloor \le x < \lfloor x \rfloor + 1$.
For $\bm{v} \in \mathbb R^r$, let $v_i$ denote its $i$-th component, $i=1,\dots,r$.
We denote by $\lvert \bm{v} \rvert_{0}$ the number of nonzero components of $\bm{v}$.
% For \(\bm{v}\in\mathbb{R}^r\), let \(\bm{v}_i\) denote its \(i\)-th component, \(i=1,\ldots,r\). 
We define $
\lvert \bm{v}\rvert_{\infty}
:=
\max_{1\le i\le r}\lvert v_i\rvert$.
% For $\bm{v} \in \mathbb{R}^{r}$, we denote by $\lvert \bm{v} \rvert_{0}$ the number of nonzero components of $\bm{v}$.
Let $U$ be a subset of $\mathbb{R}^{r}$.
% For functions $f,g: U \to \mathbb{R}$, we define the inner product by
% $
% \langle f,g\rangle_{L^{2}(U)} := \int_{U} f(\bm{u})g(\bm{u})\, d\bm{u},
% $
% and the $L^{2}(U)$ norm by
% $
% \left\| f \right\|_{L^{2}(U)} := \sqrt{\langle f,f\rangle_{L^{2}(U)}}
% $.
For a continuous function $f : U \to \mathbb{R}$, we define
$
\| f \|_{\infty} := \sup_{\bm{u} \in U} \lvert f(\bm{u})\rvert
$. 
For a continuous function $f : U \to \mathbb{R}^{d} ; \bm{u} \mapsto \left(f_{1}(\bm{u}),\dots,f_{d}(\bm{u})\right)^{\top}$, we define
$
\| f \|_{\infty} := \max_{1 \le i \le d} \| f_{i} \|_{\infty}
$.
Moreover, for an operator $\Psi : \mathcal{H} \to \mathcal{H}$, we define
$
\| \Psi \|_{\infty} := \sup_{x \in \mathcal{H}} \| \Psi(x) \|_{\mathcal{H}}
$. 
Let $M$ be any fixed positive constant. For any operator $\Psi : \mathcal{H} \to \mathcal{H}$, 
we define the operator $\bar{\Psi}^{(M)}$ by
\begin{align}
\bar{\Psi}^{(M)}(x) := \Psi(x) \bm{1}_{\{\left\| x \right\|_{\infty} \leq M \}},\quad x \in \mathcal{H}, \label{truncate}    
\end{align}
and for any function $\psi:[0,1]^{d} \times [0,1]^{d} \times \mathbb{R} \to \mathbb{R}$, we define $\bar{\psi}^{(M)}$ by
\[
\bar{\psi}^{(M)}(\bm{u},\bm{v},x):=\psi(\bm{u},\bm{v},x)\bm{1}_{[-M,M]}(x),
\qquad \bm{u},\bm{v}\in [0,1]^{d}, x\in\mathbb{R}.
\]
% In addition, for any continuous function $\psi : [0,1]^{d} \times [0,1]^{d} \times [-M,M] \to \mathbb{R}$, we will define its sup norm by $\left\| \psi \right\|_{\infty} := \sup_{\substack{\bm{u},\bm{v} \in D \\ x \in [-M,M]}}\lvert \psi(\bm{u}, \bm{v}, x) \rvert$.
In addition, for two sequences $(a_{n})$ and $(b_{n})$, if there exists a constant $C$ such that for all $n \in \mathbb{N}$, $a_{n} \leq C b_{n}$, we write $a_{n} \lesssim b_{n}$. 
Furthermore, if $a_{n} \lesssim b_{n}$ and $b_{n} \lesssim a_{n}$, we write $a_{n} \asymp b_{n}$.

\subsection{Nonlinear functional autoregressive model}\label{s22}

Let $(\Omega,\mathcal{I},\{\mathcal{F}_t\},\mathbb P)$ be a stochastic basis. 
We consider the following NFAR process $\{X_{t}\}_{t \in \mathbb{N}}$ taking values in $\mathcal{H}$:
\begin{align}
X_{t+1}(\bm{u})& = \Psi_{0}\left(X_{t}\right)(\bm{u}) + \xi_{t}(\bm{u}),\quad \bm{u} \in [0,1]^{d},\quad t \in \mathbb{N}, \label{N_far}
\end{align}
where $\Psi_{0}$ is a nonlinear operator from $\mathcal{H}$ to $\mathcal{H}$,
and $\{\xi_t\}_{t\in\mathbb N}$ is an i.i.d.\ sequence of $\mathcal{H}$-valued random elements with mean zero such that,
for each $t\in\mathbb N$, $\xi_t$ is independent of $\mathcal{F}_t$.
%Moreover, for each $t \in \mathbb{N}$, $\xi_{t}$ is independent of $X_{t},X_{t-1},\dots$. 
For notational simplicity, write $D := [0,1]^d$.
The covariance function $k: D \times D \to \mathbb{R}$ is defined by $k(\bm{u},\bm{v}) := \mathbb{E}\left[\xi_1(\bm{u})\xi_1(\bm{v})\right]$ for $\bm{u},\bm{v} \in D$.
We assume that $\int_{D}k(\bm{u},\bm{u})d\bm{u} < \infty$.
The covariance operator $Q : \mathcal{H} \to \mathcal{H}$ is defined by
\[
(Q f)(\bm{u}) := \int_{D} k(\bm{u}, \bm{v}) f(\bm{v}) \, d\bm{v}, \quad f \in \mathcal{H},\ \bm{u} \in D.
\]

Throughout this paper, we impose a Gaussianity assumption on the noise sequence $\{\xi_t\}_{t\in\mathbb N}$.
\begin{Assumption} \label{a1}
For each $t\in \mathbb{N}$, $\xi_t$ is distributed as $\mathcal{N}(\bm{0},Q)$.
% where $\mathcal{N}(\bm{0},Q)$ denotes a Gaussian measure on $\mathcal{H}$ with mean zero and covariance operator $Q$.
\end{Assumption}
\begin{Remark}\label{gauss_noise_reason}
The Gaussianity assumption plays a crucial role in the proofs of our main results. 
% To establish exponential mixing for the NFAR model, 
To establish that NFAR models are exponentially $\beta$-mixing, we adopt an approach analogous to that used for Markov chains in finite dimensional settings, which relies on the existence of the density ratio $dP_h/dP_0$ between the distributions of $\xi_t$ and $\xi_t + h$ for $h \in \mathcal{H}$. 
In infinite dimensional spaces, however, probability measures are often mutually singular, and hence such a density ratio need not exist.
On the other hand, when $\xi_t$ is Gaussian, Lemma~A.1 ensures the existence of $dP_h/dP_0$ under a suitable condition. 
Therefore, to ensure that the density ratio is well defined in the infinite dimensional setting, we impose \Aref{a1}.
\end{Remark}
% \begin{Assumption}\label{a2}
% $\overline{\mathrm{Image}(\mathcal{Q})} = \mathcal{H}$.  
% \end{Assumption} 

% \begin{Remark}
% Regarding \Aref{a2}, the condition $\overline{\mathrm{Image}(\mathcal{Q})} = \mathcal{H}$ 
% is equivalent to $\mathrm{Ker}(\mathcal{Q}) = \{0\}$. 
% A sufficient condition for this is that all eigenvalues of $\mathcal{Q}$ are strictly positive. 
% Moreover, it suffices that the covariance function \(k\) is continuous on \(D \times D\) to ensure \( \int_{[0,1]^d} k(\mathbf{u},\mathbf{u})\, d\mathbf{u} < \infty \).
% \end{Remark}
% 第一章で述べた通り本研究の主要な結果として関数型自己回帰モデルの exponential mixing を与える．
% これにより第 4章で与えられるような観測で明示的に記述することが難しい nonparametric な推定量
% の汎化誤差や Nadaraya Watson 型推定量などの代表的な推定手法の漸近分布を導出するための
% 確率論的な基盤を与えることができる．

\section{Main Results}\label{s3}

\subsection{Importance of the mixing property}\label{s30}

We first discuss the role of mixing properties in the theoretical analysis of nonparametric estimators for NFAR models.
To derive the generalization error of a nonparametric estimator of the true operator $\Psi_{0}$ constructed from a given time series $\boldsymbol{X}_{T+1} = \{X_{t}\}_{t = 1}^{T+1}$, 
it is necessary to bound the expectation of
\begin{align}
&\frac{1}{T}\sum_{t = 1}^{T}\left\| \Psi(X_{t}) - \Psi_{0}(X_{t}) \right\|_{\mathcal{H}}^{2}, \label{gen}
\end{align}
for $\Psi \in \mathcal{F}$, where $\mathcal{F}$ is a given hypothesis set.
In our setting, the sequence $\boldsymbol{X}_{T+1}$ exhibits temporal dependence, 
and we allow both $\Psi_{0}$ and $\Psi \in \mathcal{F}$ to be nonlinear. 
As a result, standard concentration inequalities are not directly applicable, 
and it is therefore difficult to bound the expectation in \eqref{gen}.
In such cases, a common approach is to introduce a simpler coupling sequence $\boldsymbol{\tilde{X}}_{T+1}=\{\tilde{X}_t\}_{t=1}^{T+1}$
that approximates the original sequence in a suitable sense.
The coupling sequence $\boldsymbol{\tilde{X}}_{T+1}$ has the same marginal distribution as $\boldsymbol{X}_{T+1}$ 
but exhibits substantially weaker dependence, thereby enabling the use of classical concentration inequalities. By introducing $\boldsymbol{\tilde{X}}_{T+1}$, we decompose \eqref{gen} as
\begin{align}
\sum_{t = 1}^{T}\left\| \Psi(X_{t}) - \Psi_{0}(X_{t}) \right\|_{\mathcal{H}}^{2}
=& \sum_{t = 1}^{T}\left\| \Psi(\tilde{X}_{t}) - \Psi_{0}(\tilde{X}_{t}) \right\|_{\mathcal{H}}^{2}\nonumber\\
&+ \sum_{t = 1}^{T}\left\{\left\| \Psi(X_{t}) - \Psi_{0}(X_{t}) \right\|_{\mathcal{H}}^{2} -
\left\| \Psi(\tilde{X}_{t}) - \Psi_{0}(\tilde{X}_{t}) \right\|_{\mathcal{H}}^{2} \right\}.
\end{align}
We then apply concentration inequalities to the first term, while bounding the approximation error
\begin{align}
&\left\| \Psi(X_{t}) - \Psi_{0}(X_{t}) \right\|_{\mathcal{H}}^{2} - \left\| \Psi(\tilde{X}_{t}) - \Psi_{0}(\tilde{X}_{t}) \right\|_{\mathcal{H}}^{2}. \label{approx_error}
\end{align}
The evaluation of the approximation error in \eqref{approx_error} depends on the specific construction of the coupling sequence $\boldsymbol{\tilde{X}}_{T+1}$.

Two representative frameworks for constructing coupling sequences are $L^{p}$-$m$ approximation \citep{hormann2010weakly} and mixing.
The $L^{p}$-$m$ approximation is convenient to use, since it applies to 
NFAR models under relatively mild conditions, such as the integrability of the noise $\{\xi_{t}\}_{t \in \mathbb{N}}$ and Lipschitz continuity of $\Psi_{0}$ with the Lipschitz constant strictly less than one.
% A limitation of the $L^{p}$-$m$ approximation is that it controls moments of $X_{t}-\tilde{X}_{t}$, 
However, it controls moments of $X_{t}-\tilde{X}_{t}$, 
but does not necessarily yield corresponding bounds for $\Psi(X_{t})-\Psi(\tilde{X}_{t})$. 
This issue is particularly pronounced when $\Psi$ lacks smoothness.
Such a situation arises, for example, when $\Psi_{0}$ is learned using methods based on DNNs, such as operator learning models.
The reason is that, in these settings, the approximation error between the DNN-based model and the true operator $\Psi_{0}$ is derived by considering the following hypothesis set, consisting of operators whose input domain is restricted to a bounded set:
\begin{align}
\mathcal{F}
= \left\{\Psi\bm{1}_{A} \mid \Psi \in \tilde{\mathcal{F}}\right\},
\nonumber
\end{align}
where $A$ is a bounded subset of $\mathcal{H}$ and $\tilde{\mathcal{F}}$ denotes the class of operators to which the estimator belongs.
This type of restriction is common in the analysis of generalization error for deep learning models.
For example, \citet{kurisu2024adaptive} derived generalization error bounds for deep learning methods applied to prediction in finite dimensional nonlinear AR models under the restriction that the input of the regression function is confined to a bounded set.
In this case, each $\Psi \in \mathcal{F}$ is discontinuous on $\mathcal{H}$.

Consequently, under $L^{p}$-$m$ approximation one typically needs to restrict attention to settings where these difficulties do not arise.
In contrast, under a mixing framework, for any bounded measurable $\Psi$,
we can derive a bound on the error $\Psi(X_{t}) - \Psi(\tilde{X}_{t})$ in terms of the error $X_{t} - \tilde{X}_{t}$.
Motivated by this advantage, we show in this section that NFAR models are exponentially $\beta$-mixing.
\subsection{Exponential $\beta$-mixing}\label{s31}
We show that NFAR models are exponentially $\beta$-mixing under suitable conditions.
%We establish these results using Markov chain theory.
First, we show exponential ergodicity of NFAR models and then derive the desired mixing result.
% An example of NFAR models satisfying exponential $\beta$-mixing is also provided.
We also provide an example satisfying these conditions.

We begin by introducing notation for Markov chains. 
Let $\{X_{t}\}_{t \in \mathbb{N}}$ be a time-homogeneous Markov chain on  the stochastic basis $(\Omega,\mathcal{I},\{\mathcal{F}_{t}\},\mathbb{P})$. 
For any $t \in \mathbb{N}$, define
    \begin{align}
        \mathcal{P}^{t}(x,A) := \mathbb{P}(X_{t+1} \in A \mid X_{1} = x),\quad A \in \mathcal{B}\left(\mathcal{H}\right), x \in \mathcal{H}, \nonumber
    \end{align}
    and 
    \begin{align}
     \left(\mathcal{P}^{t}\phi\right)(x)&:= \int_{\mathcal{H}} \phi(z) \mathcal{P}^{t}(x,dz),\quad \phi \in \mathcal{C}_{b}\left(\mathcal{H}\right), x \in \mathcal{H}. \nonumber   
    \end{align}
Moreover, define the Markov kernel
    \begin{align}
        \mathcal{P}\left(x,A\right)&:= \mathbb{P}\left(X_{2} \in A \mid X_{1} = x\right),\quad A \in \mathcal{B}\left(\mathcal{H}\right), x \in \mathcal{H}, \nonumber
    \end{align}
and  
    \begin{align}
        \left(\mathcal{P}\phi\right)(x) := \int_{\mathcal{H}} \phi(y)\mathcal{P}(x,dy),\quad \phi \in \mathcal{C}_{b}(\mathcal{H}), x \in \mathcal{H}. \nonumber
    \end{align}

We also define the $\beta$-mixing coefficient and exponential $\beta$-mixing. 
For $j \in \mathbb{N}$, the $\beta$-mixing coefficient $\beta(j)$
is defined by
\begin{align}
    \beta(j) := \sup_{n \in \mathbb{N}}\mathbb{E}\biggl[\sup_{B \in \mathcal{F}_{j+n}^{\infty}}
        \big| \mathbb{P}(B) - \mathbb{P}(B \mid X_{n},\dots,X_{1})\big|\biggr], \nonumber
\end{align}
where $\mathcal{F}_{j+n}^{\infty} := \sigma\!\left(\bigcup_{l = j+n}^{\infty}\sigma(X_{l})\right)$ for $n \in \mathbb{N}$. The Markov chain $\{X_{t}\}_{t \in \mathbb{N}}$ is said to be exponentially $\beta$-mixing if there exist positive constants $C_{1}$ and $C_{2}$ such that for any $j \in \mathbb{N}$,
\begin{align}
    &\beta(j) \leq C_{1}\exp(-C_{2}j).\nonumber
\end{align}

We impose the following assumptions on the covariance operator of the noise sequence $\{\xi_{t}\}_{t \in \mathbb{N}}$ and the operator $\Psi_{0}$. In the finite dimensional case, \Aref{a2} and \Aref{m1} are automatically satisfied, while the Gaussianity assumption in \Aref{a1} can be weakened to allow for a broader class of noise distributions; see \citet{chen2000geometric}.

\begin{Assumption}\label{a2}
$\overline{\mathrm{Im}(Q)} = \mathcal{H}$.  
\end{Assumption} 

\begin{Assumption}\label{m1}
    There exists a nonlinear operator $m_{0}:\mathcal{H}\to\mathcal{H}$ such that $\Psi_{0}=Q\circ m_{0}$.
\end{Assumption}

\begin{Assumption}\label{m2}
    Under \Aref{m1}, there exist constants $c_{1},c_{2}>0$ such that
\[
c_1\|Q\|_{\mathcal{L}}<1\;\text{ and }\;
\forall x\in\mathcal{H};\;\|m_{0}(x)\|_{\mathcal{H}} \le c_{1}\|x\|_{\mathcal{H}} + c_{2}.
\]
\end{Assumption}

\begin{Remark}
Regarding \Aref{a2}, the condition $\overline{\mathrm{Im}(Q)} = \mathcal{H}$ 
is equivalent to $\mathrm{Ker}(Q) = \{0\}$. 
A sufficient condition for this is that all eigenvalues of $Q$ are strictly positive. 
% Moreover, it suffices that the covariance function \(k\) is continuous on \(D \times D\) to ensure \( \int_{[0,1]^d} k(\mathbf{u},\mathbf{u})\, d\mathbf{u} < \infty \).
\end{Remark}

% We establish exponential ergodicity of the process \eqref{N_far} under \Aref{m1} and \Aref{m2}.
% Based on Theorem 1.2 of \citet{hairer2011harris}, 
% we establish exponential ergodicity for NFAR models.
Applying Lemma 2.4 in the supplementary material, a restatement of Theorem 1.2 of \citet{hairer2011harris}, we establish exponential ergodicity for NFAR models.

\begin{Theorem}[Exponential ergodicity]\label{exponential_ergodic}
Consider the time series $\{X_t\}_{t\in\mathbb N}$ defined by \eqref{N_far}.
Suppose that \Aref{a1}, \Aref{a2}--\Aref{m2} hold.
Then, $\{X_{t}\}_{t \in \mathbb{N}}$ admits a unique invariant distribution $\pi$. 
Moreover, there exist constants $\gamma \in (0,1)$ and $C>0$ such that, 
for any $t \in \mathbb{N}$
\begin{align}
        &\sup_{\substack{\phi \in \mathcal{C}_{b}(\mathcal{H})\\ 
        \left\| \phi \right\|_{\mathcal{C}_{b}(\mathcal{H})} \leq 1} }
        \left\lvert \left(\mathcal{P}^{t}\phi\right)(x) - \int_{\mathcal{H}}\phi(z)\pi(dz)\right\rvert \leq C\gamma^{t}(1 + \left\| x \right\|_{\mathcal{H}}), \quad x \in \mathcal{H}. \nonumber
\end{align}
\end{Theorem}

Based on Theorem \ref{exponential_ergodic}, it follows that the time series $\{X_{t}\}_{t \in \mathbb{N}}$ defined by \eqref{N_far} is exponentially $\beta$-mixing
under \Aref{a1} and \Aref{a2}--\Aref{m2}:
\begin{Corollary}[Exponential $\beta$-mixing]\label{beta_mixing_for_far_chapter3}
% Suppose that \Aref{a1}, \Aref{a2} -- \Aref{m2} hold.
% Then the functional time series $\{X_t\}_{t\in\mathbb{N}}$ defined by \eqref{N_far}
% is exponentially $\beta$-mixing.
Suppose that \Aref{a1} and \Aref{a2}--\Aref{m2} hold. 
Assume further that the time series
$\{X_t\}_{t \in \mathbb{N}}$ defined by \eqref{N_far} is stationary. 
Then, $\{X_t\}_{t \in \mathbb{N}}$ is exponentially $\beta$-mixing.
\end{Corollary}
% \begin{Remark}
% In many cases, the existence of a stationary distribution follows from \Aref{m2} alone. 
% To establish exponential ergodicity in total variation, however, an additional condition such as \Aref{m1} is required.
% This is because we consider the infinite-dimensional state space $\mathcal{H}$, where probability measures are often mutually singular.
% In particular, for two points $x_{1},x_{2}\in\mathcal{H}$, the transition kernels $\mathcal{P}(x_{1},\cdot)$ and $\mathcal{P}(x_{2},\cdot)$ may be mutually singular.
% Such singularity precludes exponential ergodicity, which requires that the total variation distance between
% $\mathcal{P}(x_{1}, \cdot)$ and $\mathcal{P}(x_{2}, \cdot)$ contracts at a uniform rate for any $x_{1},x_{2}\in\mathcal{H}$.
% \Aref{m1} ensures the overlap between $\mathcal{P}(x_{1}, \cdot)$ and $\mathcal{P}(x_{2}, \cdot)$
% for any $x_{1},x_{2} \in \mathcal{H}$ and hence yields uniform contraction.
% \end{Remark}
\begin{Remark}
In many cases, the existence of a stationary distribution follows from \Aref{m2} alone. 
To establish exponential ergodicity in total variation, however, additional conditions such as 
\Aref{a1}, \Aref{a2} and \Aref{m1} are required. 
This is because we consider the infinite dimensional state space $\mathcal{H}$, 
where probability measures are often mutually singular 
as stated in Remark \ref{gauss_noise_reason}.
In particular, for two points $x_{1},x_{2}\in\mathcal{H}$, 
the transition kernels $\mathcal{P}(x_{1},\cdot)$ and $\mathcal{P}(x_{2},\cdot)$ may be mutually singular.
Such singularity precludes exponential ergodicity, which requires that the total variation distance between
$\mathcal{P}(x_{1}, \cdot)$ and $\mathcal{P}(x_{2}, \cdot)$ contracts at a uniform rate for any $x_{1},x_{2}\in\mathcal{H}$.
\Aref{a1}, \Aref{a2} and 
\Aref{m1} ensure sufficient overlap between $\mathcal{P}(x_{1}, \cdot)$ and $\mathcal{P}(x_{2}, \cdot)$
for any $x_{1},x_{2} \in \mathcal{H}$ and hence yield uniform contraction.
\end{Remark}

\begin{Example}\label{example_mixing}
We now present a class of examples for which $\Psi_{0}$ satisfies \Aref{m1}.
In particular, we focus on the case where $\Psi_{0}$ is of the form
\begin{align}
    \Psi_{0}(x)(\bm{u}) := \int_{D} f(\bm{u}, \bm{v})\tau\left(\bm{v}, x(\bm{v})\right)d\bm{v}, 
    \quad \bm{u} \in D, x \in \mathcal{H}, \nonumber
\end{align}
where $f: D \times D \to \mathbb{R}$ and $\tau: D \times \mathbb{R} \to \mathbb{R}$. 
This class is a special case of the class of operators considered in Section~\ref{s42}. 
We investigate conditions on $f$ and $\tau$ under which \Aref{m1} holds. 

To this end, we introduce some notation.
Let $\{\lambda_{i}\}_{i = 1}^{\infty}$ be the sequence of eigenvalues of the covariance operator $Q$ of the noise, satisfying $\sum_{i = 1}^{\infty}\lambda_{i} < \infty$, and let $\{e_{i}\}_{i = 1}^{\infty} \subset \mathcal{H}$ be the corresponding eigenfunctions.
We define the sequence $\{f_{ij}\}_{i,j \in \mathbb{N}}$ by
$f_{ij} := \int_{D \times D} f(\bm{u},\bm{v})e_{i}(\bm{u})e_{j}(\bm{v})d\bm{u}d\bm{v}$.
We also define, for each $i \in \mathbb{N}$, an operator $\tau_{i}:\mathcal{H} \to \mathbb{R}$ by
$\tau_{i}(x) := \langle \tau\left(\cdot, x(\cdot)\right), e_{i} \rangle_{\mathcal{H}}, \, x \in \mathcal{H}.$
We further define a function $\tilde{k}: D \times D \to \mathbb{R}$ by
\begin{align}
    &\tilde{k}(\bm{u},\bm{v}) = \sum_{i = 1}^{\infty}\lambda_{i}^{2}e_{i}(\bm{u})e_{i}(\bm{v}
    ), \quad \bm{u},\bm{v} \in D. \nonumber
\end{align}
Let $\mathcal{H}_{\tilde{k}}$ denote the reproducing kernel Hilbert space associated with $\tilde{k}$. We then consider the tensor product space $\mathcal{H}_{\tilde{k}} \otimes \mathcal{H}_{\tilde{k}}$, which is given by 
\begin{align}
    \mathcal{H}_{\tilde{k}} \otimes \mathcal{H}_{\tilde{k}}
    &:=\left\{
    f: D \times D \to \mathbb{R} \;\middle|\;
    \sum_{i,j = 1}^{\infty}\frac{f_{ij}^{2}}{\lambda_{i}^{2}\lambda_{j}^{2}} < \infty 
    \right\}. \nonumber
\end{align}
With this notation, we now identify conditions on 
$f$ and $\tau$.
Noting that the operator $\Psi_{0}$ can be written as
\begin{align}
    \Psi_{0}(x)(\bm{u}) & = \sum_{i = 1}^{\infty}\left(\sum_{j = 1}^{\infty}f_{ij}\tau_{j}(x)\right)e_{i}(\bm{u}), \quad \bm{u} \in D, x \in \mathcal{H}, \nonumber
\end{align}
we define, for each $i$, an operator $g_{i}:\mathcal{H} \to \mathbb{R}$ by
$
g_{i}(x) := \sum_{j = 1}^{\infty} f_{ij}\tau_{j}(x)\; (x \in \mathcal{H}).
$
A sufficient condition for \Aref{m1} to hold is that, for all $x \in \mathcal{H}$,
\begin{align}
    \sum_{i = 1}^{\infty}\frac{g_{i}^{2}(x)}{\lambda_{i}^{2}} < \infty. \label{temp_g}
\end{align}
Indeed, if \eqref{temp_g} holds for all $x \in \mathcal{H}$, then the operator $m_{0} : \mathcal{H} \to \mathcal{H}$ is well defined by
$m_{0}(x) := \sum_{i = 1}^{\infty}\left(g_{i}(x)/\lambda_{i}\right)e_{i},\,x \in \mathcal{H}$ 
and with $m_{0}$ one has $\Psi_{0} = Q \circ m_{0}$, from which \Aref{m1} follows.
We therefore seek conditions on $f$ and $\tau$ sufficient for \eqref{temp_g} to hold.
Note that, for any $x \in \mathcal{H}$,
\begin{align}
    \sum_{i = 1}^{\infty}\frac{g_{i}^{2}(x)}{\lambda_{i}^{2}}
    \leq \sum_{i,j = 1}^{\infty}\frac{f_{ij}^{2}}{\lambda_{i}^{2}}\tau_{j}^{2}(x).\nonumber
\end{align}
Therefore, \Aref{m1} holds if
\begin{align}
    \sum_{i,j = 1}^{\infty}\frac{f_{ij}^{2}}{\lambda_{i}^{2}} < \infty, \quad
    \sup_{x \in \mathcal{H}}\left(\sum_{i = 1}^{\infty}\tau_{i}^{2}(x)\right) < \infty. \label{f_tau_condition}
\end{align}
For instance, \eqref{f_tau_condition} holds if
$
f \in \mathcal{H}_{\tilde{k}} \otimes \mathcal{H}_{\tilde{k}}$ and $\sup_{\bm{u} \in D, x \in \mathbb{R}} \lvert \tau(
\bm{u},x)\rvert < \infty$.
The condition
$
f \in \mathcal{H}_{\tilde{k}} \otimes \mathcal{H}_{\tilde{k}}
$ 
can be interpreted as a smoothness assumption on $f$.
\end{Example}

\section{Applications of exponential mixing}\label{s4}
% In this section, as an application of exponential mixing for NFAR models, 
% we study a regression problem with infinite-dimensional inputs and outputs, 
% as introduced in \eqref{far_intro} in Section \ref{s1}. 
In this section, we use the exponential $\beta$-mixing property of NFAR models to derive an oracle inequality for learning methods of the true operator $\Psi_{0}$. 
% In Section \ref{s40}, we describe the setting for \eqref{far_intro}.
% In Section \ref{s41}, we focus on the estimation of $\Psi_{0}$.
% In particular, we derive an upper bound on the generalization gap of an estimator taking values in a given hypothesis set.
In Section \ref{s41}, we derive an upper bound on the generalization error of an estimator taking values in a given hypothesis set.
In Section \ref{s42}, we address more practical settings. 
We restrict attention to the case where the true operator $\Psi_{0}$ is a Urysohn operator and estimate $\Psi_{0}$ by learning its Urysohn kernel using DNNs.

As noted at the beginning of Section \ref{s30}, we derive an upper bound on the generalization error of estimators of $\Psi_{0}$ by adopting a hypothesis set whose input domain is restricted to a bounded subset of $\mathcal{H}$. 
This restriction is necessary to derive approximation error bounds for deep learning models, as such models are unable to represent complex functions in regions far from the origin. 
Indeed, many results on the approximation and generalization error of deep learning impose boundedness assumptions on the input domain (\citealp{DeVore2021NeuralNetworkApproximation};
\citealp{Liu2025DeepNeuralNetworksAdaptive}; \citealp{schmidt2019nonparametric}). 
Accordingly, in Section \ref{s4} we impose boundedness on the input domain of the hypothesis set. 
As discussed in Section \ref{s3}, this restriction necessitates the mixing assumption.
\subsection{Nonparametric estimator for a nonlinear functional autoregressive model}\label{s41} 
In this section, we derive an oracle inequality for estimators of the true operator $\Psi_{0}$
 in \eqref{N_far}. We first introduce the statistical setting and additional assumptions. Throughout Section 4, we consider the time series $\{X_{t}\}_{t = 1}^{T + 1}$ defined by \eqref{N_far} and impose the following assumptions:
\begin{Assumption}\label{a3}
    For each $t \in \mathbb{N}$, the map $\bm{u} \mapsto X_{t}(\bm{u})$ is continuous on $D$.
\end{Assumption}

\begin{Assumption}\label{a4}
    The time series $\{X_{t}\}_{t \in \mathbb{N}}$ is exponentially $\beta$-mixing; that is, its $\beta$-mixing coefficient $\beta(j)$ satisfies
\begin{align}
    \beta(j)& \leq C_{1, \beta}\exp(-C_{2,\beta}j) \nonumber
\end{align}
for some positive constants $C_{1,\beta}$ and $C_{2, \beta}$ and for all $j \in \mathbb{N}$.
\end{Assumption}

\begin{Remark}
Corollary \ref{beta_mixing_for_far_chapter3} in Section \ref{s31} provides a sufficient condition for \Aref{a4}. Moreover, a concrete example of an operator $\Psi_{0}$ satisfying \Aref{m1} is given in Example \ref{example_mixing}.
\end{Remark}

We also introduce a hypothesis set and related assumptions.
Let $\mathcal{F}$ be a class of measurable mappings from $\mathcal{H}$ to $\mathcal{H}$, which serves as a hypothesis set.
We impose the following assumptions on $\mathcal{F}$. 
\begin{Assumption}\label{a6} There exists a positive constant $M > 0$ such that
for any $\Psi \in \mathcal{F}$ and $x \in \mathcal{H}$ with $\|x\|_{\infty} > M$, we have $\Psi(x) = 0$.
\end{Assumption}
\begin{Assumption}\label{a7} 
The metric space $(\mathcal{F}, \|\cdot\|_{\infty})$ is compact.
\end{Assumption}

For any estimator $\widehat{\Psi}_{T}$ taking values in $\mathcal{F}$, we also introduce its generalization error.  
% As noted at the beginning of Section~\ref{s4}, our analysis is carried out on a bounded subset of $\mathcal{H}$.
% In particular, throughout this section we work with sets that are bounded with respect to the sup norm
% $\|\cdot\|_{\infty}$.
% Accordingly, we also assume that the true operator $\Psi_{0}$ is supported on the same bounded set.
As noted at the beginning of Section 4, our analysis is carried out on a subset of $\mathcal H$ that is bounded with respect to the sup norm $\|\cdot\|_\infty$. Accordingly, we compare $\hat{\Psi}_T$ with the truncated operator $\bar{\Psi}_0^{(M)}$ on the same bounded domain.
We therefore define the generalization error of $\widehat{\Psi}_{T}$ by
\begin{align}
R\left(\widehat{\Psi}_{T},\bar{\Psi}_{0}^{(M)}\right)&:= 
  \mathbb{E}\left[\frac{1}{T}\sum_{t = 1}^{T}\left\| \widehat{\Psi}_{T}\left(X_{t}^{\prime}\right) - \bar{\Psi}_{0}^{(M)}\left(X_{t}^{\prime}\right) \right\|_{\mathcal{H}}^{2}\right], \label{generalization_gap_def}
\end{align}
where $M$ is the constant appearing in \Aref{a6}, and $\{X_{t}^{\prime}\}_{t = 1}^{T+1}$ is an independent copy of $\{X_{t}\}_{t = 1}^{T+1}$ with the same distribution. If $\{X_{t}^{\prime}\}_{t = 1}^{T+1}$ and $\{X_{t}\}_{t = 1}^{T+1}$ are stationary, we can write
 \begin{align}
R\left(\widehat{\Psi}_{T},\bar{\Psi}_{0}^{(M)}\right)& = \mathbb{E}\left[\left\| \widehat{\Psi}_{T}\left(X_{1}^{\prime}\right) - \bar{\Psi}_{0}^{(M)}\left(X_{1}^{\prime}\right) \right\|_{\mathcal{H}}^{2}\right].\nonumber 
 \end{align}
In addition, for any $\Psi \in \mathcal{F}$, define the empirical risk by 
\begin{align}
    \hat{\ell}_{T}(\Psi)& = \frac{1}{T}\sum_{t = 1}^{T}\left\| X_{t+1} - \Psi(X_{t}) \right\|_{\mathcal{H}}^{2}. \nonumber
\end{align}
For any $\Psi_{1}, \Psi_{2} \in \mathcal{F}$, define
\begin{align}
 \Delta\left(\Psi_{1},\Psi_{2}\right)& := 
 \mathbb{E}\left[ 
  \hat{\ell}_{T}\left(\Psi_{1}\right)
  - \hat{\ell}_{T}\left(\Psi_{2}\right)\right], \nonumber
\end{align}
Furthermore, for any $\Psi \in \mathcal{F}$ define
\begin{align}
  \Delta(\Psi)&:= \mathbb{E}\left[\hat{\ell}_{T}(\Psi) - \min_{\Phi \in \mathcal{F}}\hat{\ell}_{T}(\Phi)\right]. \nonumber
\end{align}
Note that the existence of a minimizer is guaranteed by \Aref{a7}.

\begin{Remark}
% In this study, we restrict the domain of operators $\Psi \in \mathcal{F}$
% with respect to the sup norm $\|\cdot\|_{\infty}$ rather than the
% $\mathcal{H}$-norm $\|\cdot\|_{\mathcal{H}}$.
% As a concrete setting of \eqref{N_far},
% we consider the case studied in Section \ref{s42},
% where the true operator $\Psi_{0}$ is a Urysohn operator.
% The estimator is obtained by learning its Urysohn kernel,
% which is a real-valued function defined on a Euclidean space.
% In this setting, bounding the covering number of $\mathcal{F}$
% reduces to evaluating the complexity of models consisting of
% real-valued functions defined on a Euclidean space.
% Consequently, boundedness must be imposed pointwise over $D$, rather than in terms of the $L^{2}$-type 
% average defined by integration over $D$. For this reason, we restrict the domain of operators in $\mathcal{F}$
% to be bounded with respect to the sup norm $\|\cdot\|_{\infty}$.
In this study, we work with a sup-norm bounded domain for operators $\Psi \in \mathcal{F}$, 
rather than an $\mathcal{H}$-norm bounded one. 
This choice is natural for the setting in Section~\ref{s42}, 
where the true operator $\Psi_{0}$ is given by a Urysohn operator. 
The estimator is obtained by learning its Urysohn kernel, which is a real-valued function on a Euclidean space. 
Therefore, evaluating the approximation error of the operator 
learning model reduces to evaluating the approximation error of deep learning models for finite dimensional functions. 
Consequently, boundedness must be imposed pointwise on $D$, rather than through an $L^{2}$-type average over $D$. 
\end{Remark}
We introduce the following notation. For $\delta>0$, a finite set $\tilde{\mathcal{F}}\subset \mathcal{F}$ is called a $\delta$-covering 
of $\mathcal{F}$ with respect to $\|\cdot\|_{\infty}$ if, for every $\Psi \in \mathcal{F}$, 
there exists $\Phi \in \tilde{\mathcal{F}}$ such that $\|\Psi - \Phi\|_{\infty} \leq \delta$.
The minimal cardinality of such a $\delta$-covering 
is the covering number of $\mathcal{F}$ with respect to $\|\cdot\|_{\infty}$, 
denoted $N(\delta,\mathcal{F},\|\cdot\|_{\infty})$.
In addition, under \Aref{a1}, there exists a constant $K'>0$ such that
\begin{align}
    \mathbb{E}\left[\exp\left\{\|\xi\|_{\mathcal{H}}^{2}/(K')^{2}\right\}\right] \leq 2, \label{Kdash}
\end{align}
where $\xi$ is distributed as $\xi_1$. Since $\xi$ is a Gaussian process, the existence of such a constant $K'$ follows from Lemma 2.8 in the supplementary material. 

The following theorem provides an upper bound for the generalization error of 
the estimator. 

\begin{Theorem}[Oracle inequality]\label{thm1}
Consider the NFAR model \eqref{N_far}.
Suppose that \Aref{a1}, \Aref{a3}, and \Aref{a4} hold, and that $\mathcal{F}$ satisfies \Aref{a6}--\Aref{a7}.
Let $\delta > 0$ and assume that
$
N(\delta,\mathcal{F},\|\cdot\|_{\infty}) < \infty
$. Let $\mathcal{N}_{\delta}$ be an integer such that $\mathcal{N}_{\delta} \geq N(\delta,\mathcal{F},\|\cdot\|_{\infty}) \lor \exp(10)$. Let $\widehat{\Psi}_{T}$ be any estimator taking values in $\mathcal{F}$. 
Let $M$ be the constant in \Aref{a6}, and suppose that
$\|\bar{\Psi}_{0}^{(M)}\|_{\infty} \leq F$ for some $F\geq 1$.
Let $C_{1,\beta}$ and $C_{2,\beta}$ be the constants appearing in \Aref{a4}. Then, for any $\varepsilon\in(0,1)$, there exists a constant $C_{\varepsilon}>0$,
depending only on $(\varepsilon,K,C_{1,\beta},C_{2,\beta})$, where
$
K := \max\!\left\{
\sqrt{\|Q\|_{\mathcal{L}}},
\sqrt{\|Q\|_{\mathbb{T}}},
K',1
\right\},
$
such that for any $T\in\mathbb{N}$ with $\log T\geq 2$ and
$C_{2,\beta}\geq 2\log T/T$,
\begin{align}
R\left(\widehat{\Psi}_{T},\bar{\Psi}_{0}^{(M)}\right)
\lesssim
\frac{1+\varepsilon}{1-\varepsilon}
\left(
\Delta(\widehat{\Psi}_{T})
+
\inf_{\Psi\in\mathcal{F}}
R\left(\Psi,\bar{\Psi}_{0}^{(M)}\right)
\right)
+
C_{\varepsilon}F^{2}
\frac{(\log \mathcal{N}_{\delta})(\log T)}{T}.
\label{variance}
\end{align}
\end{Theorem}
The above theorem is an infinite dimensional extension of Theorem~3.1 in \citet{kurisu2024adaptive}. 
The proof is deferred to the supplementary material.

\subsection{Estimation via deep neural networks}\label{s42}
In this section, we address more practical settings. We will restrict attention to the case where the true operator $\Psi_{0}$ in \eqref{N_far} takes the form
\begin{align}
 \Psi_{0}(x)(\bm{u})&:= \int_{D}\psi_{0}(\bm{u},\bm{v},x(\bm{v}))\,d\bm{v},\quad \bm{u} \in D, x \in \mathcal{H}, \label{Uryshorn}
\end{align}
where $\psi_{0}:D \times D \times \mathbb{R} \to \mathbb{R}$ is a continuous function.
We estimate $\Psi_{0}$ by learning $\psi_{0}$ using DNNs.
We first introduce the notation used in this section. Let $M > 0$. 
% For a function $\psi: D \times D \times [-M,M] \to \mathbb{R}$, we define the operator $\Psi_{\psi}: \mathcal{H} \to \mathcal{H}$ by
% \begin{align}
%     \Psi_{\psi}\!\left(x\right)(\bm{u})
%     := \int_{D}\psi\!\left( \bm{u}, \bm{v}, x\left(\bm{v}\right) \right)d\bm{v}\bm{1}_{\{\left\|x\right\|_{\infty} \leq M\}},
%     \qquad \bm{u} \in D, x \in \mathcal{H}. \label{PSI_psi}
% \end{align}
For a function $\psi : D \times D \times [-M,M] \to \mathbb{R}$,
we define the operator $\Psi_\psi : \mathcal{H} \to \mathcal{H}$ by
\begin{align}
\Psi_\psi(x)(\mathbf{u})
:=\left\{
\begin{array}{cc}
\int_D \psi(\mathbf{u},\mathbf{v},x(\mathbf{v}))\,d\mathbf{v},
    & \text{if } \|x\|_\infty \le M,
\\0,& \text{if } \|x\|_\infty > M,
\end{array}
\right., \quad \mathbf{u}\in D,\ x\in \mathcal{H}. \label{PSI_psi}
\end{align}
When $\Psi_{0}$ is of the form \eqref{Uryshorn}, the truncated operator
$\bar{\Psi}_{0}^{(M)}$ satisfies 
$\bar{\Psi}_{0}^{(M)} = \Psi_{\bar{\psi}_{0}^{(M)}}$
for any $M > 0$.
Although $\bar{\psi}_0^{(M)}$ is originally defined on $D \times D \times \mathbb{R}$, it is identically zero whenever $\lvert x \rvert > M$. Therefore, 
we identify $\bar{\psi}_0^{(M)}$ with its restriction to $D \times D \times [-M,M]$.
Accordingly, throughout this part, we use the notation $\Psi_{\bar{\psi}_{0}^{(M)}}$ for $\bar{\Psi}_{0}^{(M)}$. For a bounded set $U \subset \mathbb{R}^{r}$, we define the ball of $\beta$-Hölder functions with radius $R$ by
    \begin{align}
        &\mathcal{C}_{r}^{\beta}(U, R) := \left\{f : U \to \mathbb{R} \mid \sum_{\bm{\alpha}:\left\lvert \bm{\alpha} \right \rvert_{1} < \beta} \left\| \partial^{\bm{\alpha}} f \right\|_{\infty} + \sum_{\bm{\alpha}:\lvert \bm{\alpha} \rvert = \lfloor \beta \rfloor} \sup_{\substack{\bm{x},\bm{y} \in U\\ \bm{x} \neq \bm{y}}}\frac{\lvert \partial^{\bm{\alpha}}f (\bm{x}) - \partial^{\bm{\alpha}}f(\bm{y})}{\lvert \bm{x} - \bm{y}\rvert_{\infty}^{\beta - \lfloor \beta \rfloor} }\leq R \right\}, \nonumber
    \end{align}
    where we use the multi-index notation, i.e.,
$\partial^{\bm{\alpha}} = \partial^{\alpha_1} \cdots \partial^{\alpha_r}$
with $\bm{\alpha} = (\alpha_1,\ldots,\alpha_r) \in \mathbb{Z}_{\ge 0}^r$
and $\lvert \bm{\alpha} \rvert_{1} := \sum_{j=1}^r\alpha_j$.
    For $q \in \mathbb{N}$, $\bm{d} := (d_{0},d_{1},\dots,d_{q}) \in \mathbb{N}^{q + 1}$, $\bm{t} := (t_{0},t_{1},\dots,t_{q}) \in \mathbb{N}^{q + 1}$ with $t_{i} < d_{i}$ for all $i$, $\bm{\beta} := (\beta_{0},\dots,\beta_{q}) \in (0,\infty)^{q + 1}$, and $R>0$,  we define the class of composite functions $\mathcal{C}\left(q,\bm{d},\bm{t},\bm{\beta},R\right)$ by
    \begin{align}
\mathcal{C}\left(q,\bm{d},\bm{t},\bm{\beta},R\right)&:= \left\{
        f = g_{q} \circ \dots \circ g_{1} \,\biggm|\,  
      \begin{matrix}
       g_{i} = (g_{ij})_{j} : [a_{i},b_{i}]^{d_{i}} \to [a_{i + 1},b_{i + 1}]^{d_{i + 1}},\\
         g_{ij} \in C_{t_{i}}^{\beta_{i}}\left([a_{i},b_{i}]^{t_{i}},R\right)
          \mathrm{for}\,\mathrm{some}\,\lvert a_{i} \rvert, \lvert b_{i} \rvert \leq R
        \end{matrix}
        \right\}.\nonumber
    \end{align}
We next describe the architecture of DNNs.
The network architecture $(L,\bm{p})$ consists of a positive integer $L$ called the number of hidden layers
or depth and a width vector $\bm{p} = \left(p_{0},\dots,p_{L+1}\right) \in \mathbb{N}^{L+2}$. 
A DNN with network architecture $(L,\bm{p})$ is then any function of the form 
\begin{align} 
    &f:\mathbb{R}^{p_{0}} \to \mathbb{R}^{p_{L + 1}}: x \mapsto (A_{L + 1}\circ\sigma_{L}\circ A_{L}\circ \sigma_{L-1}\circ \cdots \circ \sigma_{1}\circ A_{1})(x) \label{DNN_form}
\end{align}
where 
$A_{l} : \mathbb{R}^{p_{l-1}} \to \mathbb{R}^{p_{l}}; x \mapsto W_{l}x + \bm{b}_{l}$ is an affine linear map 
for given weight matrix $W_{l}$ and a shift vector $\bm{b}_{l} \in \mathbb{R}^{p_{l}}$,
and $\sigma_{l}:\mathbb{R}^{p_{l}} \to \mathbb{R}^{p_{l}}$ is an element-wise nonlinear activation function map 
defined as $\sigma_{l}(\bm{z}) := \left(\sigma(\bm{z}_{1}),\dots,\sigma(\bm{z}_{p_{l}})\right)^{\top}, \bm{z} \in \mathbb{R}^{p_{l}}$
with ReLU activation function $\sigma(x):= \max\{x,0\}, x \in \mathbb{R}$.
For a neural network of the form \eqref{DNN_form}, 
we define $\theta(f):= \left(\mathrm{vec}(W_{1})^{\top},\bm{b}_{1}^{\top},\dots,\mathrm{vec}(W_{L+1})^{\top},\bm{b}_{L+1}^{\top}\right)^{\top}$
where $\mathrm{vec}(W)$ transforms the matrix $W$ into the corresponding vector by concatenating the column vectors.
We let $\mathcal{F}_{p_{0}, p_{L+1}}$ be the class of DNNs which take $p_{0}$-dimensional input to 
$p_{L+1}$-dimensional output with the ReLU activation function. 

In this section, we focus on estimating the function $\bar{\psi}_{0}^{(M)}:D \times D \times [-M,M] \to \mathbb{R}$. Accordingly, we set $p_{0} = 2d+1$ and $p_{L+1} = 1$.
For positive constants $M, B, F$ and positive integers $L,N$, we set 
\begin{align}
\mathcal{F}_{p_{0}}^{(M)}(L,N,B) &:= \left\{f\bm{1}_{D \times D \times [-M,M]} \mid f \in \mathcal{F}_{p_{0},1}:\mathrm{depth}(f) \leq L,\,\mathrm{width}(f) \leq N,\,\lvert \theta(f) \rvert_{\infty} \leq B\right\}, \nonumber
\end{align}
and 
\begin{align}
\mathcal{F}_{p_{0}}^{(M)}(L,N,B,F) &:= \left\{ f \in \mathcal{F}_{p_{0}}^{(M)}(L,N,B) \mid \left\| f \right\|_{\infty} \leq F\right\}. \nonumber
\end{align}
Moreover, we define a class of sparsity constrained DNNs with sparsity level $S > 0$ by
\begin{align}
\mathcal{F}_{p_{0}}^{(M)}(L,N,B,F,S):= \left\{ f \in \mathcal{F}_{p_{0}}^{(M)}(L,N,B,F) \mid \lvert \theta(f) \rvert_{0} \leq S\right\}. \nonumber
\end{align}

We adopt as 
the hypothesis set the operator family $\bm{G} := \left(\Psi_{\psi}\right)_{\psi \in \mathcal{G}}$ with $\mathcal{G} := \mathcal{F}_{p_{0}}^{(M)}(L,N,B,F,S)$.
% Specifically, the operator learning model $\Psi_{\psi}$ with $\psi \in \mathcal{G}$ takes as input $X \in \mathcal{H}$ with $\|X\|_{\infty} \leq M$ and $u \in D$, and produces, for each $v \in D$, the value $\psi(\bm{u}, \bm{v}, X(\bm{v})) \in \mathbb{R}$ (See Figure \ref{Architecture}). 
% The final output is obtained by averaging over $v \in D$, that is,
% $
% \int_{D}\psi(\bm{u},\bm{v},X(\bm{v}))\,d\bm{v}
% $.
Specifically, for an input pair $(\bm{u},x) \in D \times \mathcal{H}$ 
with $\|x\|_{\infty} \le M$, the model first computes $\psi(\bm{u},\bm{v},x(\bm{v}))$ for each $\bm{v} \in D$ (See Figure \ref{Architecture}), and then outputs
$\Psi_{\psi}(x)(\bm{u}) = \int_{D}\psi(\bm{u},\bm{v},x(\bm{v}))d\bm{v}$.
We estimate the true operator $\Psi_{0}$ via the estimator
\begin{align}
    &\widehat{\Psi}_{T} := \Psi_{\widehat{\psi}_{T}},\quad \widehat{\psi}_{T} := \arg\min_{\psi \in \mathcal{G}}\frac{1}{T}\sum_{t = 1}^{T}\left\| X_{t+1} - \Psi_{\psi}(X_{t}) \right\|_{\mathcal{H}}^{2}. \label{kernelNN}
\end{align}
% The architecture of $\Psi_{\psi} \in \mathcal{G}$, represented via $\psi \in \mathcal{G}$ introduced above, is illustrated as follows:
% \begin{align}
%     \Psi(X)(\bm{u}) : (\bm{u}, X) \mapsto \left[\int_{D}\psi\left(\bm{u},\bm{v},X(\bm{v})\right)d\bm{v}\right]\bm{1}_{\left\{\left\| X \right\|_{\infty} \leq M\right\}}, \quad \bm{u} \in D, \; X \in \mathcal{H}, \nonumber
% \end{align}
% and the architecture of $\psi \in \mathcal{G}$ is illustrated below.
To derive an upper bound on the generalization error of $\widehat{\Psi}_{T}$, we impose an additional assumption on the true function $\psi_{0}$:
\begin{Assumption}\label{m3}
Let $M \geq 1$. There exist $q \in \mathbb{N}$, $\bm{d} \in \mathbb{N}^{q+1}$ with $d_{0} = 2d+1$ and $d_{q} = 1$, $\bm{t} \in \mathbb{N}^{q+1}$, and $\bm{\beta} \in (0,\infty)^{q+1}$ such that $\overline{\psi}_{0}^{(M)} \in \mathcal{C}(q, \bm{d}, \bm{t}, \bm{\beta}, M)$.
\end{Assumption}
An upper bound on the generalization error of $\widehat{\Psi}_{T}$ is given in Theorem \ref{convergence_rate}.

\begin{figure}[tb]
\centering
\includegraphics[scale=0.60]
{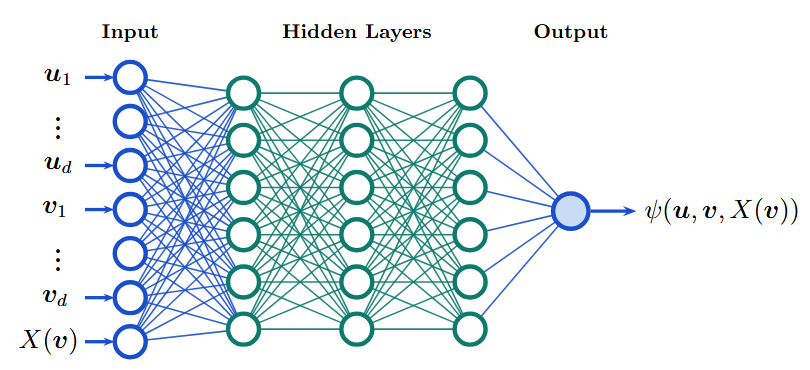}
\caption{The learning architecture of $\mathcal{G}$ for  Urysohn kernels}
\label{Architecture}
\end{figure}

\begin{Theorem}\label{convergence_rate}
    Suppose that \Aref{a1}, \Aref{a3} and \Aref{a4} hold for the model \eqref{N_far}.
Let $M \geq 1$ and $B \geq 1$ be positive constants. Let $L,S$ as well as $N$ be any integers 
    with $L \asymp \log N, S \asymp N\log N$ and 
    $N \asymp \max_{0 \leq i \leq q} T^{\frac{t_{i}}{2\beta_{i} + t_{i}}}$. 
    Suppose further that the true Urysohn kernel $\bar{\psi}_{0}^{(M)}$ satisfies \Aref{m3}.
    Then, the estimator $\widehat{\Psi}_{T}$, defined in \eqref{kernelNN}, satisfies
    \begin{align}
&R(\widehat{\Psi}_{T},\Psi_{\bar{\psi_{0}}^{(M)}})\lesssim (\log T)^{4}\max_{0 \leq i \leq q}T^{-\frac{2\beta_{i}}{2\beta_{i} + t_{i}}}. \nonumber
    \end{align}
\end{Theorem}

\begin{Remark}\label{true_function_class}
We consider the class of functions $\psi_{0}$ for which $\Psi_{0}$ satisfies \Aref{m1} and $\bar{\psi}_{0}^{(M)}$ satisfies \Aref{m3} for a given $M \geq 1$. 
In particular, we focus on the case where $\psi_{0}$ admits the representation
\[
\psi_{0}(\bm{u}, \bm{v}, x) = f(\bm{u}, \bm{v})\,\tau(\bm{v}, x), 
\quad \bm{u}, \bm{v} \in D, \; x \in \mathbb{R},
\]
for some functions $f: D \times D \to \mathbb{R}$ and $\tau: D \times \mathbb{R} \to \mathbb{R}$. 
From Example \ref{example_mixing}, the following class satisfies \Aref{m1} and \Aref{m3}:
\begin{align}
\left\{ f(\bm{u},\bm{v})\tau(\bm{v},x)\bm{1}_{\{\lvert x \rvert \leq M\}} \;\middle|\;
\begin{array}{l}
f \in \mathcal{C}(q, \bm{d}, \bm{t}, \bm{\beta}, 1), \\
\tau \in \mathcal{C}(q^{\prime}, \bm{d}^{\prime}, \bm{t}^{\prime}, \bm{\beta}^{\prime}, M), \\
\text{where $f$ and $\tau$ satisfy } \eqref{f_tau_condition}
\end{array} 
\right\}, \nonumber
\end{align}
where $q, q^{\prime} \in \mathbb{N}$ are arbitrary, $\bm{\beta} \in (0,\infty)^{q+1}$,
$\bm{\beta}^{\prime} \in (0,\infty)^{q^{\prime}+1}$, $\bm{t} \in \mathbb{N}^{q+1}$, $\bm{t}^{\prime} \in \mathbb{N}^{q^{\prime}+1}$,
$\bm{d} \in \mathbb{N}^{q+1}$ with $d_{0} = 2d$ and $d_{q} = 1$, and $\bm{d}^{\prime} \in \mathbb{N}^{q^{\prime}+1}$
with $d_{0}^{\prime} = d + 1$ and $d_{q^{\prime}}^{\prime} = 1$.
\end{Remark}

\section{Numerical experiments}\label{s5}
In this section, we study the case considered in Section \ref{s42} and conduct numerical experiments by estimating the Urysohn kernel using a deep learning approach. The notation in this section follows that of Section \ref{s42}. We begin by introducing additional notation needed to describe the simulation setting. 
For any function $\psi:[0,1]^{2} \times [0,1]^{2} \times \mathbb{R} \to \mathbb{R}$, we define $\tilde{\Psi}_{\psi}:\mathbb{R}^{25 \times 25} \to \mathbb{R}^{25 \times 25}$ by
\begin{align}
    &\tilde{\Psi}_{\psi}(z)\!\left(\frac{i_{1}}{25},\frac{j_{1}}{25}\right):= \frac{1}{25^{2}}\sum_{i_{2},j_{2} = 0}^{24}
    \psi\!\left(\frac{i_{1}}{25},\frac{j_{1}}{25},\frac{i_{2}}{25},\frac{j_{2}}{25},z_{i_{2}j_{2}}\right),\nonumber
\end{align}
for $z = \left\{z_{i_{2}j_{2}}\right\}_{0 \leq i_{2},j_{2} \leq 24} \in \mathbb{R}^{25 \times 25}$ and $0 \leq i_{1},j_{1} \leq 24$.
Consider the following NFAR model:
\begin{align}
    X_{t+1}(u_{1},u_{2}) & = \int_{0}^{1}\int_{0}^{1}5K_{G}(u_{1}-v_{1},u_{2}-v_{2})\tau\left(X_{t}(v_{1},v_{2})\right)dv_{1}dv_{2} + \xi_{t}(u_{1},u_{2}), \label{sim_far}
\end{align}
for $u_{1},u_{2} \in [0,1], t \in \mathbb{N}$, where the function $K_{G}: \mathbb{R}^{2} \to \mathbb{R}$ is defined by $K_{G}(h_{1},h_{2})= \exp\{-5(\lvert h_{1} \rvert^{2} + \lvert h_{2} \rvert^{2})\}, h_{1},h_{2} \in \mathbb{R}$, 
and $\tau : \mathbb{R} \to \mathbb{R}$ is given by $\tau(x) = 1.5 + 2.5 \cos(x) + 2.0 \sin(2x), x \in \mathbb{R}$. 
% The function $\tau$ is introduced in Example \ref{example_mixing} of Section~3, here specified explicitly for this section.
The function $\tau$ is the one introduced in Example~\ref{example_mixing} of Section~3, specified here explicitly in this section.
The sequence $\{\xi_t\}_{t\in\mathbb{N}}$ consists of i.i.d.\ Gaussian processes with mean function $\bm{0}$ and covariance function $K_{G}$(the same kernel $K_{G}$ as above), where
$\mathbb{E}\!\left[\xi_{1}(u_{1},u_{2})\,\xi_{1}(v_{1},v_{2})\right] = 
K_{G}(u_{1}-v_{1},u_{2}-v_{2}), u_{1},u_{2},v_{1},v_{2} \in [0,1]$.
In this setting, the true Urysohn kernel $\psi_{0}$ is
\begin{align}
\psi_{0}(u_{1},u_{2},v_{1},v_{2},x) &:= 5K_{G}(u_{1}-v_{1},u_{2}-v_{2})\tau(x), \quad u_{1},u_{2},v_{1},v_{2} \in [0,1], x \in \mathbb{R}, \nonumber
\end{align}
and the true operator $\Psi_{0}$ is given by
\begin{align}
\Psi_{0}(x)(u_{1},u_{2}) &:= \int_{0}^{1}\int_{0}^{1}\psi_{0}(u_{1},u_{2},v_{1},v_{2},x(v_{1},v_{2}))dv_{1}dv_{2},\quad u_{1},u_{2} \in [0,1], x \in \mathcal{H}.  \nonumber
\end{align}
Note that the NFAR model \eqref{sim_far} satisfies \Aref{a1} and \Aref{a2}--\Aref{m2}.
Therefore, by Theorem \ref{exponential_ergodic} and Corollary \ref{beta_mixing_for_far_chapter3}, 
the time series $\{X_t\}_{t\in\mathbb{N}}$
is exponentially ergodic and exponentially \(\beta\)-mixing.

We next describe how the generalization error is evaluated numerically.
Based on data $\boldsymbol{X}_{T} := \{X_{t}\}_{t = 1}^{T}$ generated according to \eqref{sim_far}, 
we train an estimator of $\Psi_{0}$ using the method proposed in Section \ref{s42} and evaluate its generalization error in \eqref{generalization_gap_def}.
% From Remark~\ref{stationary_remark} and 
% Theorem~\ref{exponential_ergodic}, 
% it can be expressed as
% \begin{align}
%  \mathbb{R}\left[\left\| \widehat{\Psi}_{T}(\tilde{X}) - \bar{\Psi}_{0}^{(M)}(\tilde{X})\right\|_{\mathcal{H}}^{2}\right], \label{sim_generalization_gap_2}
% \end{align}
% where $\tilde{X}$ is an independent copy with the stationary distribution of $\{X_{t}\}_{t = 1}^{T}$. 
As shown in Section~\ref{s4}, the generalization error can be expressed as
\begin{align}
 \mathbb{E}\!\left[\left\| \widehat{\Psi}_{T}(\tilde{X}) - \bar{\Psi}_{0}^{(M)}(\tilde{X})\right\|_{\mathcal{H}}^{2}\right], \nonumber
\end{align}
where $\tilde{X}$ has the same distribution as $X_{1}$ and is independent of $\boldsymbol{X}_{T}$.
In this numerical experiment, however, we use the true operator $\Psi_0$
instead of the truncated operator $\bar{\Psi}_0^{(M)}$ for simplicity.
Accordingly, we consider
\begin{align}
 \mathbb{E}\!\left[\left\| \widehat{\Psi}_{T}(\tilde{X}) - \Psi_{0}(\tilde{X})\right\|_{\mathcal{H}}^{2}\right]. \label{sim_generalization_gap_2}
\end{align}
To approximate the above expectation, we generate $B$ independent sample paths $\{\boldsymbol{X}_{T}^{(b)}\}_{b=1}^{B}$ of length $T$, where
$\boldsymbol{X}_{T}^{(b)}:=\{X_t^{(b)}\}_{t=1}^{T}$ is defined by \eqref{sim_far}.
For each $b$, we then generate a random element $\tilde{X}^{(b)}$ 
independently of $\boldsymbol{X}_{T}^{(b)}$ such that $\tilde{X}^{(b)}$ has the same distribution
as $X_1^{(b)}$. Therefore, the simulation data consist of $B$ independent pairs $\{(\boldsymbol{X}_{T}^{(b)},\tilde{X}^{(b)})\}_{b=1}^{B}$.
By the law of large numbers, for sufficiently large $B$,
\begin{align}
 \mathbb{E}\!\left[\left\| 
 \widehat{\Psi}_{T}(\tilde{X}) - \Psi_{0}(\tilde{X})
  \right\|_{\mathcal{H}}^{2}
  \right]
  &\approx 
  \frac{1}{B}\sum_{b = 1}^{B}
  \left\|
  \widehat{\Psi}_{T}^{(b)}(\tilde{X}^{(b)}) 
  - \Psi_{0}(\tilde{X}^{(b)})
  \right\|_{\mathcal{H}}^{2}, \label{sim_generalization_gap_3}
\end{align}
where the estimator $\widehat{\Psi}_{T}^{(b)}:\mathcal{H} \to \mathcal{H}$ is defined by 
\begin{align}
    \widehat{\Psi}_{T}^{(b)}&:= \Psi_{\widehat{\psi}_{T}},\quad  \widehat{\psi}_{T}^{(b)} := \arg\min_{\psi \in \mathcal{F}_{p_{0}, p_{L+1}}}\frac{1}{T-1}\sum_{t = 1}^{T-1}\left\| X_{t+1}^{(b)} - \Psi_{\psi}(X_{t}^{(b)})\right\|_{\mathcal{H}}^{2},\label{DNN_opt_b}
\end{align}
where $\mathcal{F}_{p_{0},p_{L+1}}$ denotes the DNN class defined in Section \ref{s42} and 
we adopt it with $p_{0} = 5, p_{L + 1} = 1$.
% Accordingly, we compute the right-hand side of \eqref{sim_generalization_gap_3} using the simulated pairs $\{(\boldsymbol{X}_{T}^{(b)},\tilde{X}^{(b)})\}_{b=1}^{B}$.

In what follows, we describe how to generate the time series defined by  \eqref{sim_far}.
Since generating fully continuous functional data is generally difficult, 
we instead simulate time series obtained by discretizing \eqref{sim_far} 
in the spatial domain. Specifically, we generate data from the following 
$100 \times 100$-dimensional nonlinear autoregressive model:
\begin{align}
        Z_{t+1}\left(\frac{i_{1}}{100},\frac{j_{1}}{100}\right)
        &= \frac{1}{100^{2}} \sum_{i_{2},j_{2} = 0}^{99} 
        5 K_{G}\left(\frac{i_{1} - i_{2}}{100},\frac{j_{1} - j_{2}}{100}\right)
        \tau\left(Z_{t}\left(\frac{i_{2}}{100},\frac{j_{2}}{100}\right)\right) 
        + \xi_{t}\left(\frac{i_{1}}{100}, \frac{j_{1}}{100}\right), 
        \label{approx_sim_far}
\end{align}
where $i_{1},j_{1} = 0,\dots,99$.
% For each $b$, we denote by $Z^{(b)}$ and $\tilde{Z}^{(b)}$ 
% the spatial discretizations of $X^{(b)}$ and $\tilde{X}^{(b)}$, respectively. 
% Here, $Z^{(b)} := \{Z_{t}^{(b)}\}_{t = 1}^{T}$ is a high-dimensional time series 
% with sample size $T$, where each element satisfies 
% $Z_{t}^{(b)} \in \mathbb{R}^{100 \times 100}$, and 
% $\tilde{Z}^{(b)} \in \mathbb{R}^{100 \times 100}$ is an independent copy of $\{Z_{t}^{(b)}\}_{t = 1}^{T}$ 
% and has the same stationary distribution. 
For each $b=1,\dots,B$ and $t\in\mathbb{N}$, define
$
Z_t^{(b)} := \left\{ Z_t^{(b)}(i/100,j/100) \right\}_{0\le i,j\le 99}
$ and let $
\boldsymbol{Z}_T^{(b)} := \left\{ Z_t^{(b)} \right\}_{t=1}^{T}
$. Let $\tilde{Z}^{(b)} \in \mathbb{R}^{100\times 100}$ be a random variable that is independent of
$\boldsymbol{Z}_{T}^{(b)}$ and has the same distribution as $Z_1^{(b)}$.
Similarly, we write
$
\tilde{Z}^{(b)} = \left\{ \tilde{Z}^{(b)}(i/100,j/100) \right\}_{0\le i,j\le 99}
$.
% Here, $Z_t^{(b)}$ is a spatial discretization of $X_t^{(b)}$.
% Likewise, $\tilde{Z}^{(b)}$ is a spatial discretization of $\tilde{X}^{(b)}$.
Further details of the data generation procedure are provided in Section~3 in the supplementary material.
Figure \ref{demodata_surface} illustrates the structure of the generated data. Each plot visualizes the surface of $\left\{Z_{t}\left(i/100,j/100\right)\right\}_{i,j = 0}^{99}$ at time points $t = 1,\ 500,$ and $1000$.
\begin{figure}[tb]
\centering
\includegraphics[width=\linewidth]
{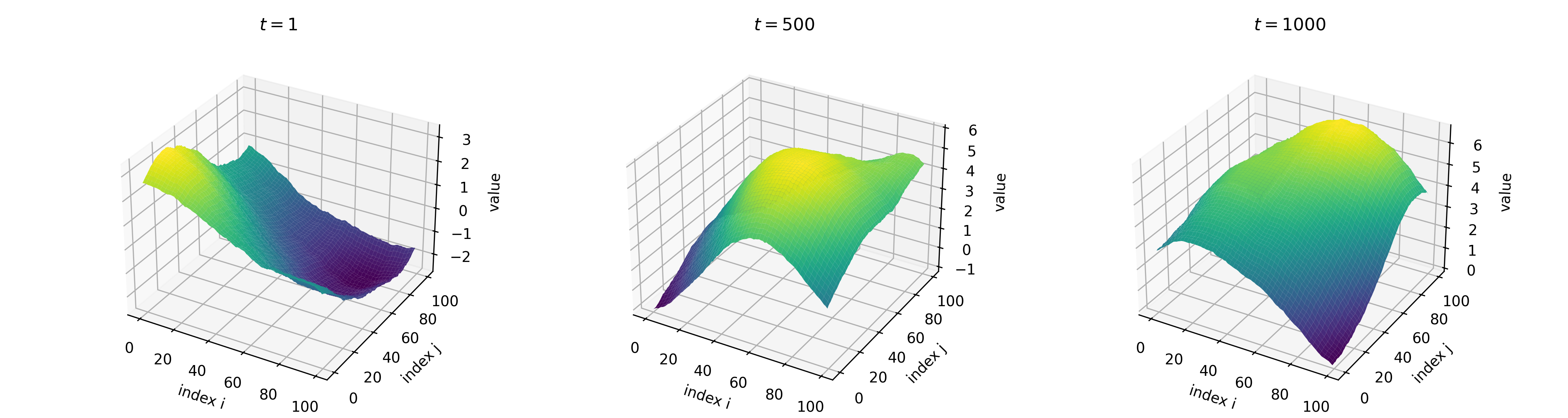}
\caption{Surface plots of the data $Z_{t}$ at times $t = 1, 500, 1000$, indexed by spatial coordinates $i$ and $j$.}
\label{demodata_surface}
\end{figure}

We then downsample the pairs $\left\{(\boldsymbol{Z}_{T}^{(b)}, \tilde{Z}^{(b)})\right\}_{b = 1}^{B}$. 
Specifically, for each $b = 1,\dots,B$ and $t = 1,\dots,T$, 
we extract 
$\left\{Z_{t}^{(b)}\left(i/25,j/25\right)\right\}_{i,j = 0}^{24}$ 
from 
$Z_{t}^{(b)} = \left\{Z_{t}^{(b)}\left(i/100,j/100\right)\right\}_{i,j = 0}^{99}$. 
The same procedure is applied to $\tilde{Z}^{(b)}$. 
For simplicity, we continue to denote the resulting downsampled data by 
$\{(\boldsymbol{Z}_{T}^{(b)}, \tilde{Z}^{(b)})\}_{b = 1}^{B}$. 

To compute $\widehat{\psi}_{T}^{(b)}$ numerically from the observations $\boldsymbol{Z}_{T}^{(b)}$ for each $b$, 
we replace the integrals in \eqref{DNN_opt_b} and \eqref{PSI_psi} with finite sums.
This yields the estimator $\tilde{\psi}_{T}^{(b)}$ defined by
\begin{align}
    \tilde{\psi}_{T}^{(b)} &:= \arg\min_{\psi \in \mathcal{F}_{p_{0},p_{L+1}}}\frac{1}{25^{2}(T-1)}\sum_{t = 1}^{T-1}
    \sum_{i,j = 0}^{24}\bigg\{ Z_{t+1}^{(b)}\left(\frac{i}{25},\frac{j}{25}\right)
    - \tilde{\Psi}_{\psi}(Z_{t}^{(b)})\left(\frac{i}{25},\frac{j}{25}\right)\bigg\}^{2}. \label{DNN_opt}
\end{align}
We then define $\tilde{\Psi}_{T}^{(b)}$ by $\tilde{\Psi}_{T}^{(b)}:= \tilde{\Psi}_{\tilde{\psi}_{T}^{(b)}}$.
Detailed specifications of the DNN training procedure, including the hyperparameter settings and the optimization algorithm, 
are provided in Section \ref{s3} in the supplementary material.
Using $\tilde{\psi}_{T}^{(b)}$ and $\tilde{Z}^{(b)}$, 
we calculate 
\begin{align}
    g_{b,T}&:= \frac{1}{25^{2}}\sum_{i,j = 0}^{24}
    \bigg[\tilde{\Psi}_{\tilde{\psi}_{T}^{(b)}}(\tilde{Z}^{(b)})\left(\frac{i}{25},\frac{j}{25}\right)-\tilde{\Psi}_{\psi_{0}}(\tilde{Z}^{(b)})\left(\frac{i}{25},\frac{j}{25}\right) 
    \bigg]^{2}, \nonumber
\end{align}
which is the discretized counterpart of 
$\left\| \widehat{\Psi}_{T}^{(b)}(\tilde{X}^{(b)}) - \Psi_{0}(\tilde{X}^{(b)})\right\|_{\mathcal{H}}^{2}$, for each $b=1,\dots,B$.

Then, we approximate the generalization error by
$
G_{B,T}:=\sum_{b=1}^{B} g_{b,T}/B.
$
The overall procedure for numerically evaluating the generalization error is summarized as follows:
\begin{enumerate}
    \item For each $b$, compute $\tilde{\psi}_{T}^{(b)}$, defined in \eqref{DNN_opt}, using $\boldsymbol{Z}_{T}^{(b)}$.
    \item For each $b$, compute the least squares distance $g_{b,T}$ using $\tilde{\psi}_{T}^{(b)}$.
    \item Report the generalization error $G_{B,T} = \sum_{b = 1}^{B}g_{b,T}/B$.
\end{enumerate}
We examined $G_{B,T}$ for $B = 161$ 
while varying $T$ across $250, 500, 750, 1000, 1250, 1500, 1750,$ $2000,3000,4000$ and $5000$. 
% Further details of the learning procedure are provided in Section 3 in the supplementary material.
Multidimensional summations arise in this numerical experiment.
Details of their computation are provided in Section~3 of the supplementary material.

The numerical results under the above settings are presented below. Figure \ref{fig:overall}\subref{convergence_rate_plot} depicts how $G_{B,T}$ varies with the sample size.
% \begin{figure}[tb]
% \centering
% \includegraphics[scale = 0.70]{FAR_DNN_Bernoulli/convergence_rate_1202.png}
% \caption{MSE versus sample size}
% \label{convergence_rate_plot}
% \end{figure}
% \begin{figure}[tb]
% \centering
% \includegraphics[width=0.6\linewidth]{FAR_DNN_Bernoulli/convergence_rate_1202.png}
% \caption{MSE versus sample size}
% \label{convergence_rate_plot}
% \end{figure}
\begin{figure}[t]
  \centering
  \begin{subfigure}[t]{0.38\textwidth}
    \centering
    \includegraphics[width=\textwidth]{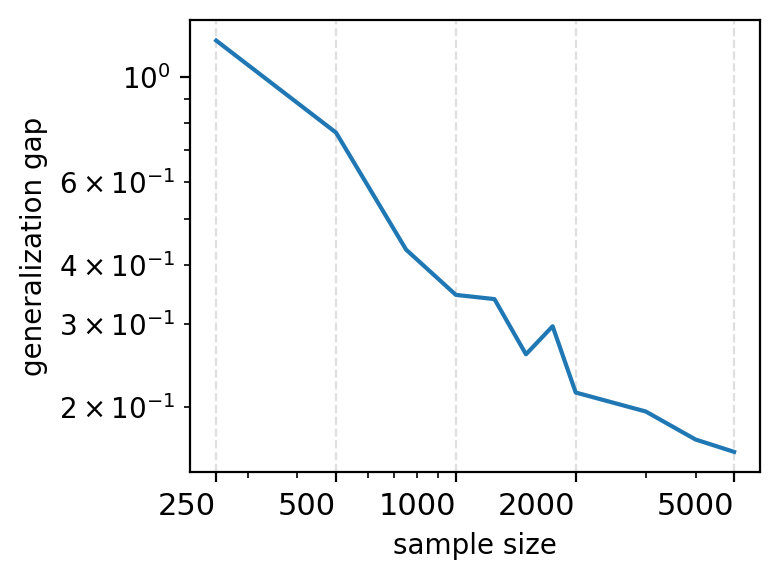}
    \caption{}
    \label{convergence_rate_plot}
  \end{subfigure}
  \hfill
  \begin{subfigure}[t]{0.59\textwidth}
    \centering
    \includegraphics[width=\textwidth]{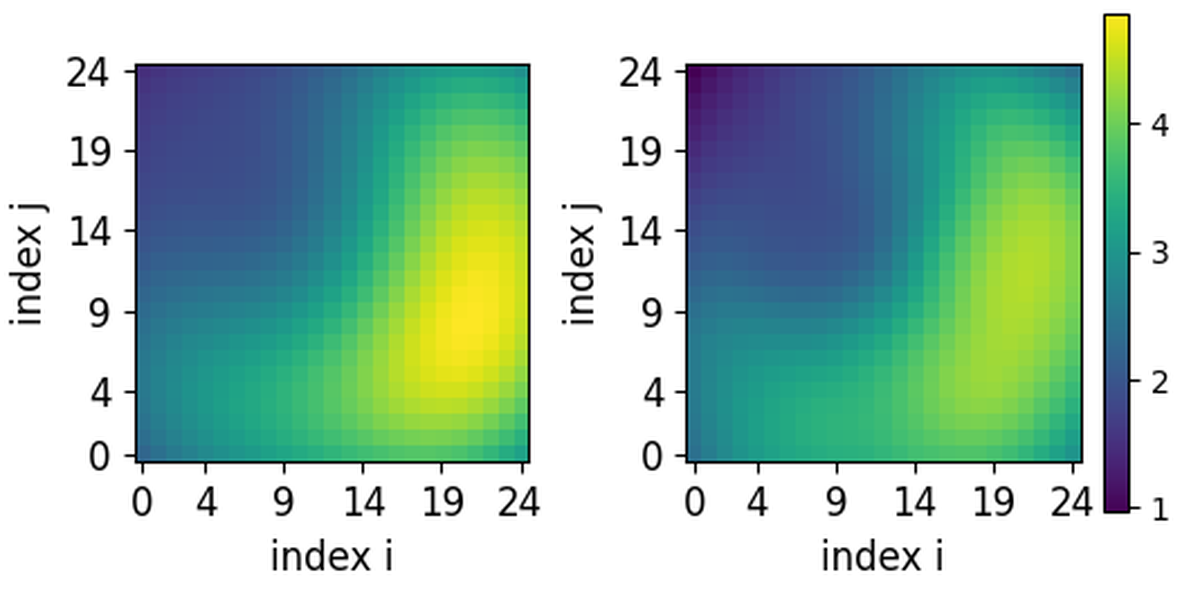}
    \caption{}
    \label{surface_prediction}
  \end{subfigure}
    \caption{Results of the numerical experiments on the generalization error and visualization of predictions. (a) Decrease in $G_{B,T}$ as the sample size increases.
    (b) Output functions obtained by applying the true operator $\tilde{\Psi}_{\psi_{0}}$ (left)
    and the learned operator $\tilde{\Psi}_{\tilde{\psi}_{T}^{(b)}}$ (right) to test data.
    The indices $i$ and $j$ denote spatial coordinates.}
  \label{fig:overall}
\end{figure}
The horizontal axis in Figure \ref{fig:overall}\subref{convergence_rate_plot} represents the sample size $T$, and the vertical axis shows the corresponding generalization error. Both axes are displayed on a logarithmic scale. 
Figure \ref{fig:overall}\subref{convergence_rate_plot} shows that the generalization error approaches $0$ as the sample size increases.

Furthermore, 
% the comparison between the true transformation
% $\left\{\tilde{\Psi}_{\psi_{0}}(\tilde{Z})\left(\frac{i}{25},\frac{j}{25}\right)\right\}_{0 \leq i,j \leq 24}$
% and the predicted transformation
% $\left\{\tilde{\Psi}_{\tilde{\psi}_{T}^{(b)}}(\tilde{Z})\left(\frac{i}{25},\frac{j}{25}\right)\right\}_{0 \leq i,j \leq 24}$
% is shown in Figure \ref{surface_prediction}.
Figure \ref{fig:overall}\subref{surface_prediction} compares the true transformation $\left\{\tilde{\Psi}_{\psi_{0}}(\tilde{Z}^{(b)})\left(\frac{i}{25},\frac{j}{25}\right)\right\}_{0 \leq i,j \leq 24}$
and the predicted transformation
$\left\{\tilde{\Psi}_{\tilde{\psi}_{T}^{(b)}}(\tilde{Z}^{(b)})\left(\frac{i}{25},\frac{j}{25}\right)\right\}_{0 \leq i,j \leq 24}$ for a representative replication $b$. 
The visual comparison indicates that the proposed method in Section \ref{s4} accurately captures the underlying transformation.

\section{Conclusion}\label{s6}
In this paper, we established sufficient conditions for NFAR models to be exponentially $\beta$-mixing 
by applying the Harris-type framework of \citet{hairer2011harris}.
As an application of the exponential $\beta$-mixing properties of NFAR models, 
we derived the oracle inequality in Theorem \ref{thm1} 
for a nonparametric estimator of the true operator $\Psi_{0}$ in the NFAR model \eqref{N_far}.
These results can be regarded as an infinite dimensional extension of \citet{kurisu2024adaptive}.
Furthermore, when the true operator is a Urysohn operator, 
we estimated its Urysohn kernel using DNNs and obtained an upper bound on the convergence rate.

A limitation of our results is that Corollary \ref{beta_mixing_for_far_chapter3} relies essentially on the assumption that the noise sequence $\{\xi_{t}\}_{t \in \mathbb{N}}$ follows a Gaussian process (see \Aref{a1}). 
Extending this assumption to broader classes of noise distributions, 
such as sub-Gaussian distributions, appears to be difficult and possible only in limited settings; 
see Corollary 3 in \citet{rao1966inference2} for related discussion. 
In this sense, the assumption is stronger than the sufficient conditions required for $L^{p}$-$m$ approximation.

We conclude by noting several directions for future work.
We should derive bounds on the error induced by discrete observations of functional time series. 
In practice, data are available only on a grid, and estimators must be constructed from discrete observations. 
It would also be of interest to characterize how the convergence rate depends on the mesh size (sampling frequency) 
from the perspective of statistical inference for stochastic processes. 
It is also important to identify alternative sufficient conditions for NFAR models to be exponentially mixing.
When the nonlinear operator admits an explicit form of how it transforms the input function, 
as in the case of a bilinear operator, 
it may be possible to establish exponential $\beta$-mixing properties 
under weaker assumptions by applying an appropriate Lyapunov drift condition. 
Furthermore, establishing the minimax optimality of the estimator proposed in Section~\ref{s4} remains an important open problem, since both the input and output spaces are infinite dimensional.
\begin{appendix}
\section{Proofs of the results in Section \ref{s31}}\label{s71}
In this section, we provide proofs of lemmas and theorems in Section \ref{s3}.
We begin by showing that $\{X_{t}\}_{t \in \mathbb{N}}$
defined by \eqref{N_far} is a time-homogeneous Markov chain. Indeed, since $\xi_{t}$ is independent of $\mathcal{F}_{t}$ in \eqref{N_far} (see Section \ref{s2}), we have that for any $t \in \mathbb{N}$,
\begin{align}
    \mathbb{P}\left(X_{t+1} \in A \mid X_{t} = x\right) &= \mathbb{P}\left(\Psi_{0}(x) + \xi_{t} \in A\right) = \mathcal{N}\left(\Psi_{0}(x),Q\right)(A),\quad A \in \mathcal{B}\left(\mathcal{H}\right), \nonumber
\end{align}
which shows that the chain $\{X_{t}\}_{t \in \mathbb{N}}$ is time-homogeneous. 
A useful tool for establishing exponential $\beta$-mixing is Lemma 2.7 in the supplementary material.
Based on Lemma 2.7, it suffices to establish the following:
\begin{enumerate}[label=(C\arabic*)]
\item \label{c1} The Markov chain $\{X_{t}\}_{t \in \mathbb{N}}$ admits an invariant distribution $\pi$.
\item \label{c2} There exists a measurable function $C:\mathcal{H} \to \mathbb{R}$ and a constant $\gamma \in (0,1)$ such that, for all $t \in \mathbb{N}$ and $\phi \in \mathcal{C}_{b}(\mathcal{H})$,
    \begin{align}
        &\left\lvert (\mathcal{P}^{t}\phi)(x) - \int_{\mathcal{H}}\phi(z)\pi(dz) \right\rvert \leq \gamma^{t} C(x), \quad x \in \mathcal{H}. \nonumber
    \end{align}
\item \label{c3} With the function $C$ as in \ref{c2} and $\pi$ as in \ref{c1},
    \begin{align}
        &\int_{\mathcal{H}} C(x)\,\pi(dx) < \infty. \nonumber
    \end{align}
\end{enumerate}
Conditions \ref{c1} and \ref{c2} are ensured by Theorem \ref{exponential_ergodic}.
We now prove Theorem \ref{exponential_ergodic}.
It follows by applying Lemma 2.4 in the supplementary material to the time series $\{X_t\}_{t \in \mathbb{N}}$ defined by \eqref{N_far}. 
Lemma 2.4 is a restatement of Theorem 1.2 of \citet{hairer2011harris}, which gives a version of Harris' ergodic theorem.
\begin{proof}[Proof of Theorem \ref{exponential_ergodic}]
    Set $V(x) := \left\| x \right\|_{\mathcal{H}}$, $x \in \mathcal{H}$.
    By Lemma 2.4 in the supplementary material, it suffices to verify the following conditions:
    \begin{itemize}
    \item[(A)] There exist constants $\rho \in (0,1)$ and $\kappa\ge 0$ such that
    \[
        (\mathcal{P}V)(x)\le \rho\,V(x)+\kappa,\quad x\in\mathcal{H}. 
    \]
    \item[(B)] For some $R>2\kappa/(1-\rho)$ there exists a probability measure $\nu_{R}$ on $\big(\mathcal{\mathcal{H}},\mathcal{B}(\mathcal{H})\big)$ and a constant $\varepsilon_{R}\in(0,1]$ such that, for all $x\in V_{R}:=\{x\in\mathcal{H}:V(x)\le R\}$,
     \[
        \mathcal{P}(x,\cdot)\ \ge\ \varepsilon_{R}\,\nu_{R}(\cdot). 
    \]
    \end{itemize}
    
    We first verify the condition (A). 
    From \Aref{m1} and \Aref{m2}, we have that there exist 
    positive numbers $c_{1},c_{2} > 0$ such that for any $x \in \mathcal{H}$, 
    \begin{align}
        \left(\mathcal{P}V\right)(x) 
        = \mathbb{E}\left[ \left\| \Psi_{0}\left(x\right) + \xi_{1} \right\|_{\mathcal{H}}\right] 
        \leq  c_{1} \left\| Q \right\|_{\mathcal{L}} \left\| x \right\|_{\mathcal{H}} + c_{2} \left\| Q \right\|_{\mathcal{L}} + \sqrt{\left\| Q\right\|_{\mathbb{T}}}. \label{a11}
    \end{align}
    Set $\rho := c_{1}\left\| Q  \right\|_{\mathcal{L}}$ 
    and 
    $\kappa := c_{2}\left\| Q \right\|_{\mathcal{L}} + \sqrt{\left\| Q \right\|_{\mathbb{T}}}.$
    Then, we find that $\rho \in (0,1)$ from \Aref{m2}, which yields (A) combined with \eqref{a11}.
    
We now prove (B). 
% Let $P$ denote the law of a centered Gaussian random element $\xi$ on $\mathcal{H}$.  
Let $\mu_{Q}$ be a probability measure on $(\mathcal{H},\mathcal{B}(\mathcal{H}))$ defined by $\mu_{Q}(\cdot) := \mathcal{N}(0,Q)(\cdot)$.
% Let $P$ denote the law of $\xi_{1}$ on $\mathcal{H}$. 
% For $h \in \mathcal{H}$, let $P_h$ denote the law of $\xi_{1} + h$. 
 Let $\mathcal{H}_{k}$ be the reproducing kernel Hilbert space 
associated with the covariance function $k$. 
Let $I^*$ be the canonical embedding from $\mathcal{H}$ into
\[
L_2(\mu_{Q}) := \left\{ f:\mathcal{H} \to \mathbb{R} : \int_{\mathcal{H}} \lvert f(x) \rvert^{2} \, d\mu_{Q}(x) < \infty \right\},
\]
and let $I:L_2(\mu_{Q})\to \mathcal{H}$ denote the adjoint operator of $I^*$.
We denote the closure of $I^{*}(\mathcal{H})$ in $L_{2}(\mu_{Q})$ by $\mathcal{H}_{k}^{*}$. 
To verify the condition (B), we apply the following lemmas.
The proof of Lemma \ref{CM_ver2_appendix} is given in Section~1 of the supplementary material.
    
\begin{Lemma}[Cameron-Martin Theorem; see Theorem 5.1 in \citet{lifshits2012gaussian}]\label{CM_appendix}
Let $h \in \mathcal{H}$, and let $\mu_{h,\mathcal{Q}}(\cdot) := \mathcal{N}(h,Q)(\cdot)$.
Then, $\mu_{h,Q} \ll \mu_{Q}$ if and only if $h \in \mathcal{H}_{k}$. If $h \in \mathcal{H}_{k}$, then the density $d\mathcal{N}(h,Q)/d\mathcal{N}(0,Q)$ has a form 
 \begin{align}
    \frac{d\mathcal{N}(h,Q)}{d\mathcal{N}(0,Q)}(x)& = \exp\left\{z(x) - \frac{\| h \|_{\mathcal{H}_{k}}^{2}}{2}\right\}, \nonumber
 \end{align}
 where $z \in \mathcal{H}_{k}^{*}$ is a measurable linear functional such that $Iz = h$. 
\end{Lemma}
\begin{Lemma}\label{CM_ver2_appendix}
% Suppose that $\mathcal{X}$ is a Hilbert space $\mathcal{H}$.  
Let $h\in\mathcal{H}$ be of the form $h=Qg$ for some $g\in\mathcal{H}$. Then,
\begin{align}
    (I^{-1}h)(x) &= \langle g, x\rangle_{\mathcal{H}},\quad x \in \mathcal{H}, \label{skorohod_int}\\
    \|h\|_{\mathcal{H}_{k}}^{2} &= \langle Qg, g\rangle_{\mathcal{H}}. \label{CM_norm}
\end{align}
\end{Lemma}
%We now return to the proof of (B). 
% Set a probability measure $\mu_{Q}$ on $(\mathcal{H},\mathcal{B}(\mathcal{H}))$ by $\mu_{Q}(\cdot) := \mathcal{N}(0,Q)(\cdot)$.
Fix $R > 0$ and $x \in V_{R}$. 
From Lemma~\ref{CM_appendix}, \Aref{a1}, \Aref{a2}, and \Aref{m1}, we find that
for any $A \in \borelH{\mathcal{H}}$,
    \begin{align}
        \mathcal{P}\left(x,A\right)&= \int_{A}\frac{d\mathcal{N}\left(\Psi_{0}\left(x\right),Q\right)}{d \mathcal{N}\left(0,Q\right)}(h)\,\mu_{Q}\left(dh\right), \nonumber
    \end{align}
    where
    \begin{align}
\frac{d\mathcal{N}\left(\Psi_{0}\left(x\right),Q\right)}{d \mathcal{N}\left(0,Q\right)}(h)&:=  \exp\left\{ 
        \left(I^{-1}\Psi_{0}\left(x\right)\right)\left(h\right) - \frac{1}{2}\left\| \Psi_{0}(x) \right\|_{\mathcal{H}_{k}}^{2} \right\},\quad h \in \mathcal{H}. \nonumber    
    \end{align}
    Moreover, from Lemma \ref{CM_ver2_appendix} and \Aref{m1}, 
    we obtain
    \begin{align}
\frac{d\mathcal{N}\left(\Psi_{0}\left(x\right),Q\right)}{d \mathcal{N}\left(0,Q\right)}(h)&
        = \exp\biggl\{
            \langle m_{0}\left(x\right), h \rangle_{\mathcal{H}} - \frac{1}{2}\left\langle Q m_{0}\left(x\right), m_{0}\left(x\right) \right\rangle_{\mathcal{H}}\biggr\}. \nonumber        
    \end{align}
    For any $x \in \mathcal{H}$, set $f_{x}:\mathcal{H} \to \mathbb{R}$
    by $f_{x}(h) := \langle m_{0}\left(x\right), h \rangle_{\mathcal{H}} - \frac{1}{2}\left\langle Q m_{0}\left(x\right), m_{0}\left(x\right) \right\rangle_{\mathcal{H}}, h \in \mathcal{H}$.
    Note that, for any $x,h \in \mathcal{H}$,
    \[
      f_{x}(h) \geq - \left\|m_{0}\left(x\right)\right\|_{\mathcal{H}}\left\| h \right\|_{\mathcal{H}} - 
    \frac{1}{2}\left\| Q \right\|_{\mathcal{L}}\left\| m_{0}(x) \right\|_{\mathcal{H}}^{2}.
    \]
    Hence, from \Aref{m2}, we have that there exist positive constants $c_{1},c_{2} > 0$ for $x,h \in V_{R}$,
    \begin{align*} 
        f_{x}\left(h\right) \geq& - \left(c_{1}\left\| x \right\|_{\mathcal{H}} + c_{2}\right)\left\| h \right\|_{\mathcal{H}} - \frac{1}{2}\left\| Q \right\|_{\mathcal{L}}\left(c_{1} \left\| x \right\|_{\mathcal{H}} + c_{2}\right)^{2} \\
        \geq& - \left(c_{1} R + c_{2}\right) R - \frac{1}{2}\left\| Q \right\|_{\mathcal{L}}\left(c_{1} R + c_{2}\right)^{2}=: - \epsilon,
    \end{align*}
    and therefore
    \begin{align}
         \mathcal{P}\left(x,A\right)
         &= \int_{A}\exp\left\{f_{x}(h)\right\}\mu_{Q}\left(dh\right)
         \geq \int_{A \cap V_{R}}\exp\left\{f_{x}(h)\right\}\mu_{Q}\left(dh\right)\nonumber\\
         &\geq \int_{A \cap V_{R}} \exp(-\epsilon)\,\mu_{Q}\left(dh\right) 
         = \exp(-\epsilon)\,\mu_{Q}\left(V_{R}\right)\frac{\mu_{Q}\left(A \cap V_{R} \right)}{\mu_{Q}\left(V_{R}\right)}
           = \delta \nu_{R}(A),\label{aa2}
    \end{align}
    where $\delta = \exp(-\epsilon)\mu_{Q}\left(V_{R}\right)$ and $\nu_{R}$ is a probability measure on $(\mathcal{H},\mathcal{B}(\mathcal{H}))$ defined by $\nu_{R}(A) := \mu_{Q}\left(A \cap V_{R} \right)/\mu_{Q}\left(V_{R}\right)\;\;(A \in \mathcal{B}(\mathcal{H}))$.
Here we use that $\mu_{Q}(V_{R})>0$ in the definition of $\nu_{R}$, which follows from \Aref{a2}.
By definition, $\delta\in(0,1)$, and hence (B) follows from \eqref{aa2}.
\end{proof}

The following lemma establishes \ref{c3}.
\begin{Lemma}\label{moment_control_appendix}
Suppose that \Aref{a1}, \Aref{a2}--\Aref{m2} hold for \eqref{N_far}.
    Then, any invariant measure $\pi$ of the time series $\{X_{t}\}_{t \in \mathbb{N}}$ defined by \eqref{N_far} satisfies that 
    \begin{align}
        &\int_{\mathcal{H}}\left\| x \right\|_{\mathcal{H}}\pi(dx) < \infty. \nonumber
    \end{align}
\end{Lemma}
%We now prove Lemma \ref{moment_control_appendix}.
\begin{proof}%[Proof of Lemma \ref{moment_control_appendix}]
Set $V:\mathcal{H} \to \mathbb{R}_{\geq 0}$ by $V(x) := \left\| x \right\|_{\mathcal{H}}, x \in \mathcal{H}$.
    Then, we have that there exist constants $\rho \in (0,1)$ and $\kappa > 0$ such that
    \begin{align}
        \left(\mathcal{P}V\right)(x)\leq \rho V(x) + \kappa,\quad x \in \mathcal{H}, \nonumber
    \end{align}
    as in the proof of Theorem \ref{exponential_ergodic}. Setting $\tilde{V} := 1 + V$ and adjusting $\kappa > 0$ if necessary, we obtain
    \begin{align}
        \left(\mathcal{P}\tilde{V}\right)(x)\leq \rho \tilde{V}(x) + \kappa, \quad x \in \mathcal{H}. \nonumber
    \end{align}
    By Lemma 2.5 in the supplementary material, there exist constants $R > 0$, $\rho^{'} > 0$, and $b > 0$ such that
    \begin{align}
        \left(\mathcal{P}\tilde{V}\right)(x) - \tilde{V}(x)\leq -\rho^{'}\tilde{V}(x) + b \bm{1}_{\tilde{V}_{R}}(x), \quad x \in \mathcal{H}, \nonumber
    \end{align}
    where $\tilde{V}_{R} := \{x \in \mathcal{H} \mid \tilde{V}(x) \leq R\}$.
    Applying Lemma 2.6 in the supplementary material, the claim follows.
\end{proof}

We are now ready to prove Corollary \ref{beta_mixing_for_far_chapter3}.

\begin{proof}[Proof of Corollary \ref{beta_mixing_for_far_chapter3}]
Since $\{X_t\}_{t \in \mathbb{N}}$ is stationary, 
the Riesz representation theorem and Lemma~2.7 in the supplementary material imply that for any $j \in \mathbb{N}$,
\begin{align}
\beta(j)
&= \int_{\mathcal{H}} 
\sup_{\substack{\phi \in \mathcal{C}_{b}(\mathcal{H})\\ 
\|\phi\|_{\mathcal{C}_{b}(\mathcal{H})} \leq 1}}
\left| (\mathcal{P}^{j}\phi)(x) - \int_{\mathcal{H}}\phi(z)\pi(dz) \right| \, \pi(dx). \label{aaa1}
\end{align}
Moreover, by Lemma~\ref{moment_control_appendix}, we have
\begin{align}
\int_{\mathcal{H}} \|x\|_{\mathcal{H}} \, \pi(dx) < \infty. \label{aaa2}
\end{align}
Combining \eqref{aaa1} and \eqref{aaa2} and applying Theorem~\ref{exponential_ergodic}, we obtain
\begin{align}
\beta(j)
&\lesssim \gamma^{j} \int_{\mathcal{H}} \bigl(1 + \|x\|_{\mathcal{H}}\bigr) \, \pi(dx)
\lesssim \gamma^{j},\nonumber
\end{align}
which completes the proof.
\end{proof}
\section{Proofs of the results in Section \ref{s42}}\label{s72}
In this section, we provide proofs of theorems and lemmas in Section \ref{s42}.
To apply Theorem \ref{thm1} to $\bm{G}$, we first verify that $\bm{G}$ satisfies \Aref{a6} and \Aref{a7}.
\begin{Lemma}\label{check_a6_and_a7}
Let $M > 0$.
   Fix $L \in \mathbb{N}$, $B \geq 1$, $F \geq 1$, $S > 0$, and $N \in \mathbb{N}$.
   Set $\mathcal{G} := \mathcal{F}_{p_{0}}^{(M)}(L,N,B,F,S)$ and $\bm{G} := (\Psi_{\psi})_{\psi \in \mathcal{G}}$.
   Then, $\bm{G}$ satisfies \Aref{a6} and \Aref{a7}.
\end{Lemma}
\begin{proof}[Proof of Lemma \ref{check_a6_and_a7}]
Let $\psi \in \mathcal{G}$. From the definition of $\Psi_{\psi}$, 
we have that the domain of $\Psi_{\psi}$ is included in the bounded set $\{x \in \mathcal{H} : \|x\|_{\infty} \leq M\}$, which verifies \Aref{a6}.
We now prove \Aref{a7}. Define $I:\mathcal{G}\to\bm{G}$ by $I(\psi):=\Psi_{\psi}, \psi \in \mathcal{G}$.
Then $I$ is continuous from the normed space $\big(\mathcal{G}, \| \cdot \|_{\infty}\big)$ to the normed space $\big(\bm{G}, \|\cdot\|_{\infty}\big)$.
Since $\mathcal{G}$ is compact in $\mathbb{R}^{P}$ with $P = 2(d+1)N + (L-1)(N^2 + N) + (N+1)$, it follows from the continuity of $I$ that $\mathrm{Im}(I)$ is compact. 
Finally, because $\mathrm{Im}(I)=\bm{G}$, we conclude that $\bm{G}$ is compact, which is precisely \Aref{a7}.
\end{proof}

% ここまでの議論から作用素のクラス $\bm{G}$ に対して定理 \ref{thm1} を適用することができる．
% したがって汎化誤差 $R(\Psi_{\widehat{\psi}_{T}},\Psi_{\psi_{0}^{*}})$ の導出を行うためには以下の2つの補題を証明できればよい:
% \begin{Lemma}\label{covering_for_dnn}
%      任意に $L > 0, B \geq 1, F \geq 1, S > 0, N \in \mathbb{N}$ を固定する．
%      上で定めた $\bm{G}$ に対して
%     \begin{align}
%       \mathcal{N}(\bm{G},\left\| \cdot \right\|_{\infty,\infty},\delta)
%       &\leq 2S(L + 1)\log\left(\frac{(L + 1)(N + 1)B}{\delta}\right), \nonumber
%     \end{align}
%     が成り立つ．
% \end{Lemma}
% \begin{Lemma}\label{bias_for_dnn}
%     任意に $L > 0, B \geq 1, F \geq 1, S > 0, N \in \mathbb{N}$ を固定する．
%     真の関数 $\psi_{0}^{*}$ について $\psi_{0}^{*} \in \mathcal{G}\left(q,\bm{d},\bm{t},\bm{\beta},K\right)$ を仮定する．
%     この時以下が成り立つ：
%     \begin{align}
%      \inf_{\Psi \in \bm{G}}R(\Psi,\Psi_{0}^{*})&\lesssim \max_{i = 0,\dots,q}N^{-\frac{2\beta_{i}^{*}}{t_{i}}}.\nonumber
%     \end{align}
% \end{Lemma}

Building on the above, we can apply Theorem \ref{thm1} to the hypothesis set $\bm{G}$.
Consequently, to derive the generalization error $R\left(\widehat{\Psi}_{T},\bar{\Psi}_{0}^{(M)}\right)$ it suffices to prove the following two lemmas.

\begin{Lemma}\label{covering_for_dnn}
Let $M > 0$. Fix $L \in \mathbb{N}$, $B \geq 1$, $F \geq 1$, $S > 0$, and $N \in \mathbb{N}$.
   Set $\mathcal{G} := \mathcal{F}_{p_{0}}^{(M)}(L,N,B,F,S)$ and $\bm{G} := (\Psi_{\psi})_{\psi \in \mathcal{G}}$.
Then, we have that for any $\delta > 0$
\begin{align}  
\log N\!\left(\delta, \bm{G},\|\cdot\|_{\infty}\right)
  \leq 2S(L + 1)\log\!\left(\frac{(L + 1)(N + 1)B}{\delta}\right). \nonumber
\end{align}
\end{Lemma}
\begin{proof}%[Proof of Lemma \ref{covering_for_dnn}]
Let $\psi_{1},\psi_{2} \in \mathcal{G}$. 
We have that
for any $x \in \mathcal{H}$ with $\|x\|_{\infty}\le M$,
\begin{align}
\big\|\Psi_{\psi_{1}}(x)-\Psi_{\psi_{2}}(x)\big\|_{\mathcal{H}}
&=\Bigg\{\int_{D}\!\Big(\int_{D}\!\big(\psi_{1}(\bm{u},\bm{v},x(\bm{v}))-\psi_{2}(\bm{u},\bm{v},x(\bm{v}))\big)\,d\bm{v}\Big)^{2}du\Bigg\}^{1/2}\nonumber\\
&\le \Bigg(\int_{D}\!\int_{D}\!\big\{\psi_{1}(\bm{u},\bm{v},x(\bm{v}))-\psi_{2}(\bm{u},\bm{v},x(\bm{v}))\big\}^{2}\,d\bm{u}\,d\bm{v}\Bigg)^{1/2}\nonumber\\
&\le \left\|\psi_{1} - \psi_{2}\right\|_{\infty}. \nonumber
\end{align}
Hence, we have
\begin{align}
\|\Psi_{\psi_{1}} - \Psi_{\psi_{2}}\|_{\infty}
\leq \left\| \psi_{1} - \psi_{2} \right\|_{\infty}. \label{ineq1}
\end{align}
It follows that $N\!\left(\delta, \bm{G},\|\cdot\|_{\infty}\right)\le N\!\left(\delta,\mathcal{G},\|\cdot\|_{\infty}\right)$. By Lemma 2.9 in the supplementary material,
\[
\log N\!\left(\delta,\bm{G},\|\cdot\|_{\infty}\right)\le 2S(L+1)\log\!\left(\frac{(L+1)(N+1)B}{\delta}\right),
\]
which proves the claim.
\end{proof}

\begin{Lemma}\label{bias_for_dnn}
Let $M \geq 1$. Fix $L \in \mathbb{N}$, $B \geq 1$, $F \geq 1$, $S > 0$, and $N \in \mathbb{N}$.
   Set $\mathcal{G} := \mathcal{F}_{p_{0}}^{(M)}(L,N,B,F,S)$ and $\bm{G} := (\Psi_{\psi})_{\psi \in \mathcal{G}}$.
Suppose the true Urysohn kernel $\psi_{0}$ satisfies \Aref{m3}.
Then
\begin{align}
 \inf_{\Psi \in \bm{G}} R\left(\Psi,\bar{\Psi}_{0}^{(M)}\right) \lesssim \max_{i=0,\dots,q} N^{-\frac{2\beta_{i}}{t_{i}}}. \nonumber
\end{align}
\end{Lemma}
%We now prove Lemmas \ref{covering_for_dnn} and \ref{bias_for_dnn}.
\begin{proof}%[Proof of Lemma \ref{bias_for_dnn}]
For any $\psi \in \mathcal{G}$, we have that
\begin{align}
R(\Psi_{\psi},\bar{\Psi}_{0}^{(M)})
&=\mathbb{E}\!\left[\int_{D}\!\Big(\int_{D}\psi(\bm{u},\bm{v},X(\bm{v}))\,d\bm{v}-\int_{D}\bar{\psi}_{0}^{(M)}(\bm{u},\bm{v},X(\bm{v}))\,d\bm{v}\Big)^{2}\,d\bm{u}\right]\nonumber\\
&\leq \left\|\psi -\bar{\psi}_{0}^{(M)}\right\|_{\infty}^{2}. \nonumber
\end{align}
Consequently,
\[
\inf_{\Psi\in\bm{G}} R\left(\Psi,\bar{\Psi}_{0}^{(M)}\right)
\leq \inf_{\psi\in\mathcal{G}}
\left\|\psi  - \bar{\psi}_{0}^{(M)}\right\|_{\infty}^{2}.
\]
From \Aref{m3}, we can apply the equation~(26) of \citet{schmidt2019nonparametric}, which yields the desired result.
\end{proof}

From these lemmas, Theorem \ref{convergence_rate} follows.

\begin{proof}[Proof of Theorem \ref{convergence_rate}]
By Lemma~\ref{check_a6_and_a7}, the hypothesis set $\bm{G}$ satisfies
\Aref{a6} and \Aref{a7}, and for any $\delta>0$,
\begin{align}
    \log N\!\left(\delta, \bm{G},\|\cdot\|_{\infty}\right)
    \leq 2S(L+1)\log\!\left(\frac{(L+1)(N+1)B}{\delta}\right). \nonumber
\end{align}
Applying Theorem \ref{thm1}, we find that for any $\varepsilon \in (0,1)$ there exists a constant $C_{\varepsilon}>0$ depending only on $(\varepsilon,C_{1,\beta},C_{2,\beta},K)$ with $K = 
 \max\{\sqrt{\left\| Q \right\|_{\mathcal{L}}}, \sqrt{\left\| Q \right\|_{\mathbb{T}}}, K^{'}, 1\}$ such that for all $T \in \mathbb{N}$ with $\log T \geq 2$ and $C_{2,\beta} \geq \frac{2\log T}{T}$,
\begin{align}
R\left(\widehat{\Psi}_{T},\bar{\Psi}_{0}^{(M)}\right) \leq
    \frac{1+\varepsilon}{1-\varepsilon}\!\left(\Delta(\widehat{\Psi}_{T}) 
    + \inf_{\Psi \in \bm{G}} R\left(\Psi,\bar{\Psi}_{0}^{(M)}\right)\right)
    + C_{\varepsilon}F^{2}\frac{\log N\!\left(\delta, \bm{G},\|\cdot\|_{\infty}\right)\big(\log T\big)}{T}. \label{temp1}
\end{align}
Let $\mathcal{N}_{T} := N\!\left(\frac{1}{T}, \bm{G}, \|\cdot\|_{\infty}\right) \lor \lfloor \exp(10) + 1 \rfloor$.
Then there exists a constant $c_{1}>0$ such that
\begin{align}
    \log \mathcal{N}_{T}
    \leq c_{1}\, S\,(L+1)\log\!\big((L+1)(N+1)BT\big). \nonumber
\end{align}
Moreover, from Lemma \ref{bias_for_dnn} and the assumption that $N \asymp \max_{0 \leq i \leq q}T^{\frac{t_{i}}{2\beta_{i} + t_{i}}}$,
we obtain that
\begin{align}
    \inf_{\Psi \in \bm{G}} R\left(\Psi,\bar{\Psi}_{0}^{(M)}\right)
    \lesssim \max_{0 \leq i \leq q}N^{-\frac{2\beta_{i}}{t_{i}}} \lesssim \max_{0 \leq i \leq q}T^{-\frac{2\beta_{i}}{2\beta_{i}+ t_{i}}}, \label{bias_rate}
\end{align}
Furthermore, from the assumption that $L \asymp \log N, S \asymp N\log N$ and $N \asymp \max_{0 \leq i \leq q}T^{\frac{t_{i}}{2\beta_{i} + t_{i}}}$, 
\begin{align}
  \frac{\log \mathcal{N}_{T}\,(\log T)}{T}
  \lesssim \frac{(\log T)(N \log N) (\log 2NT)}{T} \lesssim (\log T)^{4}\max_{0 \leq i \leq q}T^{-\frac{2\beta_{i}}{2\beta_{i} + t_{i}}}. \label{var_rate}
\end{align}
Combining \eqref{temp1}, \eqref{bias_rate} and \eqref{var_rate}, for any $\varepsilon \in (0,1)$,
\begin{align}
R\left(\widehat{\Psi}_{T},\bar{\Psi}_{0}^{(M)}\right)
     \lesssim (\log T)^{4}\max_{0 \leq i \leq q}T^{-\frac{2\beta_{i}}{2\beta_{i} + t_{i}}}. \nonumber
\end{align}
\end{proof}

\section{Proofs for theorems and lemmas}
In this supplementary material, all notation is consistent with that used in the main paper.
\subsection{Proofs of Theorem~4.1}
In this part, we provide proofs of the theorems and lemmas in Section~4.1.
In this paper, we prove Theorem 4.1 in Section~4.1 following the argument of \citet{kurisu2024adaptive}.
Our contribution is essentially an extension of their result to the infinite-dimensional setting.
For completeness, we include the full proof in the appendix so that the paper is self-contained.
% We introduce the following notation. For any $x \in \mathbb{R}$, let $\lfloor x \rfloor$ denote the integer satisfying $\lfloor x \rfloor \le x < \lfloor x \rfloor + 1$.
We begin by proving the following two lemmas, which are used to prove Theorem 4.1.

\begin{Lemma}\label{lemmac1}
Suppose that Assumption~4.2 holds.
Let $\delta > 0$ and suppose that $N(\delta, \mathcal{F}, \| \cdot \|_{\infty}) < \infty$. Let $\mathcal{N}_{\delta}$ be an integer such that $\mathcal{N}_{\delta} \geq N(\delta, \mathcal{F}, \| \cdot \|_{\infty}) \lor \exp(10)$. 
Also, let $a_T$ be a positive number such that $\mu_T := \left\lfloor \frac{T}{2a_T} \right\rfloor > 0$.
In addition, suppose that there is a number $F \geq 1$ such that 
$\| \Psi \|_{\infty} \leq F$ for all $\Psi \in \mathcal{F} \cup \left\{ \bar{\Psi}_{0}^{(M)} \right\}.$
Then, for all $\varepsilon \in (0,1]$,
\begin{align}
R\left(\widehat{\Psi}_{T}, \bar{\Psi}_{0}^{(M)}\right)
\leq &(1 + \varepsilon) \widehat{R}_{T}\left(\widehat{\Psi}_{T}, \bar{\Psi}_{0}^{(M)}\right)
+ \frac{21(1 + \varepsilon)^2}{\varepsilon} \cdot \frac{F^2 \log \mathcal{N}_{\delta}}{\mu_T}
+ \frac{4F^2}{\mu_T} \nonumber\\
&+ 4(2 + \varepsilon) F^2 \beta(a_T)
+ 4(2 + \varepsilon) F \delta. \nonumber
\end{align}
\end{Lemma}

% 以下補題を述べるために必要な記号を追加する．\eqref{far} に登場する $\{\xi\}_{t \in \mathbb{N}}$ と同分布な確率要素として $\xi^{*}$ を任意に与える．
% $\xi^{*}$ はガウス過程であるから, \textcite{Giorgobiani2019SubGaussian} の定理 1 より
% ある定数 $K^{'} > 0$ が存在して 
% \begin{align}
%     &\mathbb{E}\left[\exp\{\left\|\xi_{t} \right\|_{\mathcal{H}}^{2}/(K^{'})^{2}\}\right] \leq 2, \nonumber
% \end{align}
% がなりたつ．

\begin{Lemma}\label{lemmac2}
 Suppose that Assumption 2.1, Assumption 4.1 hold and that $\mathcal{F}$ satisfies Assumption 4.3 -- 4.4.  
 Let $T$ satisfy $\log T \geq 2$. 
 Let $\delta > 0$ and
 assume that $N(\delta,\mathcal{F}, \left\| \cdot \right\|_{\infty}) < \infty$. Let $\mathcal{N}_{\delta}$ be an integer such that $\mathcal{N}_{\delta} \geq N(\delta,\mathcal{F}, \left\| \cdot \right\|_{\infty}) \lor \exp(10)$. 
 Moreover, suppose that there is a number $F \geq 1$ such that $\left\| \Psi \right\|_{\infty} \leq F$ for all $\Psi \in \mathcal{F} \cup \left\{\bar{\Psi}_{0}^{(M)}\right\}$. 
 Then, for all $\varepsilon \in (0,1)$ there exists a constant $C_{\varepsilon}$ 
 depending only on $(\varepsilon,K,C_{1,\beta},C_{2,\beta})$ with $K := 
 \max\{\sqrt{\left\| Q \right\|_{\mathcal{L}}}, \sqrt{\left\| Q\right\|_{\mathbb{T}}}, K^{'}, 1\}$ such that
\begin{align}
 &\widehat{R}\left(\widehat{\Psi}_{T},\bar{\Psi}_{0}^{(M)}\right) \leq \frac{1}{1 - \varepsilon}\Delta(\widehat{\Psi}_{T}) + \frac{1}{1 - \varepsilon}\inf_{\Psi \in \mathcal{F}}R\left(\Psi,\bar{\Psi}_{0}^{(M)}\right) + C_{\varepsilon}F^{2}\gamma_{\delta,T}, \nonumber
\end{align}
where $\gamma_{\delta,T} := \delta + \frac{\left(\log T\right)\left(\log \mathcal{N}_{\delta}\right)}{T} + \frac{1}{T}$.
\end{Lemma}

We prove Theorem~4.1 by combining Lemma \ref{lemmac1} and Lemma \ref{lemmac2}.
\begin{proof}[Proof of Theorem 4.1]
% Let $\delta$ be any positive number.
% For simplicity, for any $\delta>0$ we denote
% $N_{\delta} := N\!\left(\delta, \mathcal{F}, \|\cdot\|_{\infty}\right)$.
  Let $a_{T}$ be any integer such that $\mu_{T} := \lfloor \frac{T}{2a_{T}} \rfloor > 0$. 
  Then, from Lemma \ref{lemmac1}, we have that for any $\varepsilon \in (0,1)$ and $\delta > 0$,
\begin{align}
R\left(\widehat{\Psi}_{T},\bar{\Psi}_{0}^{(M)}\right)\leq& 
  (1 + \varepsilon) \widehat{R}\left(\widehat{\Psi}_{T},\bar{\Psi}_{0}^{(M)}\right)
  + 4\left(2 + \varepsilon\right) F \delta + 21F^{2} \frac{(1 + \varepsilon)^{2}}{\varepsilon} \frac{\log \mathcal{N}_{\delta}}{\mu_{T}} \nonumber\\
  &+ 4\left(2 + \varepsilon\right)F^{2}\beta(a_{T}) + \frac{4F^{2}}{\mu_{T}}.\nonumber
\end{align}
Further, by Lemma \ref{lemmac2}, we have that for any $\varepsilon \in (0,1)$ there exists a positive constant $C_{\varepsilon}$ depending only on $(\varepsilon, K, C_{1,\beta}, C_{2,\beta})$ 
 such that
\begin{align}
R\left(\widehat{\Psi}_{T},\bar{\Psi}_{0}^{(M)}\right) \leq&
  \frac{1+ \varepsilon}{1 - \varepsilon}\left(\Delta(\widehat{\Psi}_{T}) + \inf_{\Psi \in \mathcal{F}}R\left(\Psi,\bar{\Psi}_{0}^{(M)}\right)\right) + \left(1 + \varepsilon\right)C_{\varepsilon}F^{2}\left(\delta + \frac{1}{T} + \frac{\left(\log T\right)\left(\log \mathcal{N}_{\delta}\right)}{T}\right) \nonumber\\
  &+ \frac{21(1 + \varepsilon)^2}{\varepsilon} \cdot \frac{F^2 \log \mathcal{N}_{\delta}}{\mu_T}
+ \frac{4F^2}{\mu_T}
+ 4(2 + \varepsilon) F^2 \beta(a_T)
+ 4(2 + \varepsilon) F \delta. \nonumber
\end{align}
Letting $a_{T}:= \lfloor C_{2,\beta}^{-1}\log T \rfloor$, which implies $\beta(a_{T}) \lesssim \frac{1}{T}$,
we find that $\mu_{T} > 0$ under the assumption $C_{2,\beta} \geq \frac{2\log T}{T}$. 
We also have that $\frac{1}{\mu_{T}} \lesssim \frac{\log T}{T}$ from the definition of $\mu_{T}$.
Based on the above estimates, we obtain that
\begin{align}
  R\left(\widehat{\Psi}_{T}, \bar{\Psi}_{0}^{(M)}\right) 
  \leq&\frac{1+ \varepsilon}{1 - \varepsilon}\left(\Delta(\widehat{\Psi}_{T}) + \inf_{\Psi \in \mathcal{F}}R\left(\Psi,\bar{\Psi}_{0}^{(M)}\right)\right) 
  + \left(1 + \varepsilon\right)C_{\varepsilon}F^{2}
  \frac{\left(\log T\right)\left(\log \mathcal{N}_{\delta}\right)}{T}\nonumber\\
  &+ \left(8\left(2 + \varepsilon\right) + 2C_{\varepsilon}(1 + \varepsilon)\right)\frac{F^{2}}{T} + \left(21 \frac{(1 + \varepsilon)^{2}}{\varepsilon} \log \mathcal{N}_{\delta} + 4\right)F^{2}\frac{\log T}{T} \nonumber\\
  \lesssim &\frac{1+ \varepsilon}{1 - \varepsilon}\left(\Delta(\widehat{\Psi}_{T}) + \inf_{\Psi \in \mathcal{F}}R\left(\Psi,\bar{\Psi}_{0}^{(M)}\right)\right) + C_{\varepsilon}^{'}F^{2}\frac{\left(\log \mathcal{N}_{\delta}\right)(\log T)}{T}, \nonumber
\end{align}
where $C_{\varepsilon}^{'}$ is a constant depending only on $(\varepsilon, K, C_{1,\beta}, C_{2,\beta})$,
which is the desired bound and hence proves Theorem~4.1.
\end{proof}
\subsection{Proofs of lemma \ref{lemmac1} and \ref{lemmac2}}
 We will provide the proof for Lemma \ref{lemmac1} and Lemma \ref{lemmac2}. We begin by proving Lemma \ref{lemmac1}. 
\begin{proof}[Proof of Lemma \ref{lemmac1}]
Let $\delta$ be any positive number and
$\{\Psi_{1},\dots,\Psi_{\mathcal{N}_{\delta}}\}$ be a $\delta$ covering of $\mathcal{F}$ with respect to $\left\| \cdot \right\|_{\infty}$.
Let $J$ be a random variable taking values in $\{1,\dots,\mathcal{N}_{\delta}\}$ satisfying $\left\| \Psi_{J} - \widehat{\Psi}_{T} \right\|_{\infty} \leq \delta$. 
\setcounter{proofstep}{0}  % Reset for each proof
\proofstep{Reduction to independence}
We rely on the coupling technique for $\beta$-mixing sequences to construct independent blocks. See \citet{rio2013inequalities}. For $\ell = 0, \dots, \mu_T - 1$, let $I_{1,\ell} = \{2\ell a_T + 1, \dots, (2\ell + 1)a_T\}, \, I_{2,\ell} = \{(2\ell + 1)a_T + 1, \dots, 2(\ell + 1)a_T\}$. Define
\begin{align}
\tilde{g}_{\ell} &:= 
(\tilde{g}_{1,\ell}, \dots, \tilde{g}_{\mathcal{N}_{\delta},\ell})^{\top} = 
\left( \sum_{t \in I_{1,\ell}} \left\|\Psi_{1}(X_t) - \bar{\Psi}_{0}^{(M)}(X_t)\right\|_{\mathcal{H}}^2
, \dots, 
\sum_{t \in I_{1,\ell}} \left\|
\Psi_{\mathcal{N}_{\delta}}(X_t) - \bar{\Psi}_{0}^{(M)}(X_t)\right\|_{\mathcal{H}}^2 \right)^{\top},\nonumber\\
\tilde{g}_{\ell}^{\prime} &:= 
(\tilde{g}^\prime_{1,\ell}, \dots, \tilde{g}^\prime_{\mathcal{N}_{\delta},\ell})^{\top} = \left( \sum_{t \in I_{1,\ell}} 
\left\|\Psi_{1}(X_t^\prime) - \bar{\Psi}_{0}^{(M)}(X_t^\prime)\right\|_{\mathcal{H}}^{2},
 \dots, \sum_{t \in I_{1,\ell}} \left\|\Psi_{\mathcal{N}_{\delta}}(X_t^\prime) - \bar{\Psi}_{0}^{(M)}(X_t^\prime)\right\|_{\mathcal{H}}^{2} \right)^{\top}. \nonumber
\end{align}

In the following, we extend the probability space if necessary and assume that there is a sequence $\{U_\ell\}_{\ell = 0}^{\infty}$ of i.i.d. uniform random variables over $[0,1]$ independent of $\{X_{t}\}_{t = 1}^{T}$.

We will show that there exist two sequences of independent $\mathbb{R}^{\mathcal{N}_{\delta}}$-valued random variables $\{g_\ell\}_{\ell = 0}^{\mu_T - 1}$ and $\{g^\prime_\ell\}_{\ell = 0}^{\mu_T - 1}$ such that for all $\ell = 0, \dots, \mu_T - 1$,
\begin{align}
    |\mathbb{E}[g_{j,\ell}] - \mathbb{E}[\tilde{g}_{j,\ell}]| &\leq 4F^2 a_T \beta(a_T), \label{diff_ori_and_cou} \\
    |\mathbb{E}[g^\prime_{j,\ell}] - \mathbb{E}[\tilde{g}^\prime_{j,\ell}]| &\leq 4F^2 a_T \beta(a_T). \label{diff_ori_and_cou2}
\end{align}

where $g_{j,\ell}$ and $g^\prime_{j,\ell}$ are the $j$-th component of $g_\ell$ and $g^\prime_\ell$, respectively. We only prove \eqref{diff_ori_and_cou} since the proof of \eqref{diff_ori_and_cou2} is similar.

First we will show that there exists a sequence $\{g_\ell\}_{\ell = 0}^{\mu_T - 1}$ of independent random vectors in $\mathbb{R}^{\mathcal{N}_{\delta}}$ such that
\[
g_\ell \overset{d}{=} \tilde{g}_\ell, \quad \mathbb{P}(g_\ell \neq \tilde{g}_\ell) \leq \beta(a_T) \quad \text{for } 0 \leq \ell \leq \mu_T - 1.
\]

For all $\ell_1, \ell_2$, define the $\sigma$-field $\mathcal{A}(\ell_1, \ell_2)$ generated by $\{X_t\}_{t \in I(\ell_1, \ell_2)}$ where $I(\ell_1, \ell_2) := \bigcup_{\ell = \ell_1}^{\ell_2} I_{1,\ell}$. From the definition of $\tilde{g}_\ell$, we find that $\sigma(\tilde{g}_\ell) \subset \mathcal{A}(\ell, \ell)$ for all $\ell$. From Assumption 4.2 and Lemma \ref{rio} in the supplementary material, there exists a random vector $g_\ell$ such that $g_\ell \overset{d}{=} \tilde{g}_\ell$, independent of $\mathcal{A}(0, \ell - 1)$, and $\mathbb{P}(g_\ell \neq \tilde{g}_\ell) \leq \beta(a_T)$. Moreover, $g_\ell$ is measurable with respect to the $\sigma$-field generated by $\mathcal{A}(0,\ell) \cup \sigma(U_\ell)$. Therefore, for any $\ell$, $g_\ell$ is independent of $\{g_{\ell'}\}_{\ell' = 0}^{\ell - 1}$ since for $\ell_1 < \ell_2$, $g_{\ell_2}$ is independent of the $\sigma$-field generated by $\mathcal{A}(0, \ell_1)$ and $U_{\ell_1}$. This implies that $\{g_\ell\}_{\ell = 0}^{\mu_T - 1}$ is a sequence of independent random variables.

Next we will show \eqref{diff_ori_and_cou}. 
By definition we have $0 \leq \tilde{g}_{j,\ell} \leq 4F^2 a_T$ for all $j$. Since $g_\ell$ has the same law as $\tilde{g}_\ell$, we also have $0 \leq g_{j,\ell} \leq 4F^2 a_T$ a.s.\ for all $j$. Consequently, we have
\begin{align}
\left\lvert \mathbb{E}[g_{j,\ell}] - \mathbb{E}[\tilde{g}_{j,\ell}]\right\rvert &\leq \mathbb{E}[\left\lvert g_{j,\ell} - \tilde{g}_{j,\ell} \right\rvert \bm{1}_{\{g_{\ell} \neq \tilde{g}_{\ell}\}}] \leq 4F^2 a_T \mathbb{P}(g_{\ell} \neq \tilde{g}_{\ell}) \leq 4F^2 a_T \beta(a_T).\nonumber
\end{align}
\proofstep{Bounding the difference of the sum of independent blocks}
In this step, we will show that for any $\varepsilon \in (0,1]$, 
\begin{align}
  \mathbb{E}\left[\sum_{\ell = 0}^{\mu_{T}-1}g_{J,\ell}^{\prime}\right]&\leq \left( 1 + \varepsilon \right) \mathbb{E}\left[\sum_{\ell = 0}^{\mu_{T}-1}g_{J,\ell}\right] + \frac{21\left(1 + \varepsilon\right)^{2}}{\varepsilon}F^{2}a_{T}\log \mathcal{N}_{\delta}, \label{ber 1}\\
  \mathbb{E}\left[\sum_{\ell = 0}^{\mu_{T}-1}\tilde{g}_{J,\ell}^{\prime}\right]&\leq \left( 1 + \varepsilon \right) \mathbb{E}\left[\sum_{\ell = 0}^{\mu_{T}-1}\tilde{g}_{J,\ell}\right] + \frac{21\left(1 + \varepsilon\right)^{2}}{\varepsilon}F^{2}a_{T}\log \mathcal{N}_{\delta}. \label{ber 2}
\end{align}
% Since \eqref{ber 2} can be shown similarly as in the proof of \eqref{ber 1}, we only show \eqref{ber 1}.
For $j = 1,\dots,\mathcal{N}_{\delta}$, define $r_{j}$ and $B$ as 
\begin{align}
  r_{j}&:= \left( \frac{4F^{2} a_{T} \log \mathcal{N}_{\delta}}{\mu_{T}} \lor \frac{1}{\mu_{T}}\sum_{\ell = 0}^{\mu_{T}-1}\mathbb{E}\left[ g_{j,\ell} \right] \right)^{1/2},\quad B := \max_{1 \leq j \leq \mathcal{N}_{\delta}}\left\lvert \frac{ \sum_{\ell = 0}^{\mu_{T}-1}\left(g_{j,\ell} - g_{j,\ell}^{\prime}\right) }{2 F r_{j}} \right\rvert. \nonumber
\end{align}
Then, using the Cauchy-Schwarz inequality, we obtain that
\begin{align}
 \left\lvert \mathbb{E}\left[ g_{J,\ell} - g^{\prime}_{J,\ell} \right] \right\rvert& \leq 2 F \mathbb{E}\left[r_{J}B\right]\nonumber\\
 & \leq 2  \mathbb{E}\left[ \left(\frac{1}{\mu_{T}}\sum_{\ell = 0}^{\mu_{T}-1}\mathbb{E}\left[g_{J,\ell}\right] \right)^{1/2} \times B^{1/2} \right] + 2 \sqrt{\frac{4F^{2} a_{T} \log \mathcal{N}_{\delta}}{\mu_{T}}} \mathbb{E}\left[ B \right] \nonumber\\
 & \leq 2 \sqrt{\frac{\mathbb{E}[B^{2}]}{\mu_{T}}} \times \mathbb{E}\left[\sum_{\ell = 0}^{\mu_{T}-1}g_{J,\ell}\right]^{1/2} + 2 \sqrt{\frac{4F^{2}a_{T}\log \mathcal{N}_{\delta}}{\mu_{T}}} \mathbb{E}\left[B\right] . \label{oracle ineq 1}
\end{align}
Combining this with (C.4) in \citet{schmidt2019nonparametric}, we find that for any $\varepsilon \in (0,1]$,
\begin{align}
 \mathbb{E}\left[\sum_{\ell = 0}^{\mu_{T}-1}g_{J,\ell}^{\prime}\right]&\leq (1 + \varepsilon) \mathbb{E}\left[\sum_{\ell = 0}^{\mu_{T}-1}g_{J,\ell}\right] + 4F^{2}(1 + \varepsilon)\sqrt{\frac{a_{T}\log \mathcal{N}_{\delta}}{\mu_{T}}}\mathbb{E}\left[B\right] + \frac{(1 + \varepsilon)^{2}}{\varepsilon}\frac{\mathbb{E}[B^{2}]}{\mu_{T}}. \nonumber
\end{align}
Now, we show that
\begin{align}
  &\mathbb{E}\left[ B \right] \leq 3\sqrt{a_{T}\mu_{T}\log \mathcal{N}_{\delta}},\quad \mathbb{E}\left[ B^{2} \right] \leq 9 a_{T}\mu_{T} \log \mathcal{N}_{\delta}, \label{bernstein}
\end{align}
We find that for $j = 1, \dots, \mathcal{N}_{\delta}$
\begin{align}
  \left\lvert \frac{g_{j,\ell} - g^{\prime}_{j,\ell}}{2r_{j}F} \right\rvert &\leq \sqrt{\frac{a_{T}\mu_{T}}{\log \mathcal{N}_{\delta}}}, \nonumber\\
  \sum_{\ell = 0}^{\mu_{T}-1}\mathbb{E}\left[\left\lvert \frac{g_{j,\ell} - g_{j,\ell}^{\prime}}{2Fr_{j}}\right\rvert^{2}\right] &\leq 2\sum_{\ell = 0}^{\mu_{T}-1}\mathbb{E}\left[\frac{g_{j,\ell}^{2}}{(2Fr_{j})^{2}}\right] \leq 2\frac{4F^{2}a_{T}}{4F^{2}}\sum_{\ell = 0}^{\mu_{T}-1}\frac{\mathbb{E}\left[ g_{j,\ell} \right]}{\frac{1}{\mu_{T}}\sum_{\ell = 0}^{\mu_{T}-1}\mathbb{E}\left[g_{j,\ell}\right]} = 2a_{T}\mu_{T}. \nonumber
\end{align}
Thus, applying Bernstein's inequality (see Lemma 2.2.9 in \citet{vdvaart1996weak}), we have that for $j = 1,\dots,\mathcal{N}_{\delta}$,
\begin{align}
  &\mathbb{P}\left(B \geq u\right) \leq 2\exp\left\{-\frac{\frac{1}{2}u^{2}}{2a_{T}\mu_{T} + \frac{1}{3}\sqrt{\frac{\log \mathcal{N}_{\delta}}{a_{T}\mu_{T}}}u}\right\},\quad u \in \left(0,\infty\right). \nonumber
\end{align}
Here, let $\gamma > 0$ be a positive number satisfying $3\gamma^{2} - 12\gamma - 2 = 0$ and define $\alpha = \gamma \sqrt{a_{T}\mu_{T}\log \mathcal{N}_{\delta}}$. 
Then, we find that for $u \geq \alpha$
\begin{align}
 &\exp\left\{-\frac{\frac{1}{2}u^{2}}{2a_{T}\mu_{T} + \frac{1}{3}\sqrt{\frac{\log \mathcal{N}_{\delta}}{a_{T}\mu_{T}}}u}\right\} \leq \exp\left\{-\frac{\log \mathcal{N}_{\delta}}{\alpha}u \right\}. \nonumber
\end{align}
Combining these results, we obtain that
\begin{align}
  \mathbb{E}\left[B\right] =& \int_{0}^{\infty}\mathbb{P}\left( B > u \right)du \leq \alpha + \int_{\alpha}\mathbb{P}\left(B \geq u\right)du \nonumber \\
\leq& \alpha + 2\int_{\alpha}^{\infty}\mathcal{N}_{\delta}\exp\left\{-\frac{\log \mathcal{N}_{\delta}}{\alpha}x\right\}dx \nonumber\\
=& \alpha + 2\mathcal{N}_{\delta}  \frac{\alpha}{\log \mathcal{N}_{\delta}} \exp\left\{-\frac{\log \mathcal{N}_{\delta}}{\alpha} \alpha\right\} \nonumber\\
=& \left(1 + \frac{2}{\log \mathcal{N}_{\delta}} \right) \gamma \sqrt{a_{T}\mu_{T}\log \mathcal{N}_{\delta}} \nonumber\\
\leq& 3 \sqrt{a_{T}\mu_{T}\log \mathcal{N}_{\delta}}, \nonumber
\end{align}
and
\begin{align}
 \mathbb{E}\left[B^{2}\right] =& 2\int_{0}^{\infty}u\mathbb{P}\left(B \geq u\right)du \leq 2\int_{0}^{\alpha}u du + 2 \int_{\alpha}^{\infty} u \mathbb{P}\left(B \geq u \right)du\nonumber\\
 \leq & \alpha^{2} + 4\int_{\alpha}^{\infty}u \exp\left\{-\frac{\log \mathcal{N}_{\delta}}{\alpha}u\right\}du \nonumber\\
 \leq & \alpha^{2} + 4\mathcal{N}_{\delta}\left(\frac{\alpha^{2}}{\log \mathcal{N}_{\delta}} + \frac{\alpha^{2}}{\log^{2}\mathcal{N}_{\delta}}\right)\exp\left\{-\frac{\log \mathcal{N}_{\delta}}{\alpha} \alpha\right\} \nonumber\\
 \leq & 1.44 \gamma^{2} a_{T}\mu_{T}\log \mathcal{N}_{\delta} \leq 9 a_{T}\mu_{T}\log \mathcal{N}_{\delta}.\nonumber
\end{align}
Therefore, combining \eqref{oracle ineq 1} and \eqref{bernstein}, we find that
\begin{align}
 \mathbb{E}\left[\sum_{\ell = 0}^{\mu_{T}-1}g_{J,\ell}^{\prime}\right]&\leq (1 + \varepsilon) \mathbb{E}\left[\sum_{\ell = 0}^{\mu_{T}-1}g_{J,\ell}\right] + 4F^{2}(1 + \varepsilon)a_{T}\log \mathcal{N}_{\delta} + 9F^{2}\frac{(1 + \varepsilon)^{2}}{\varepsilon}a_{T}\log \mathcal{N}_{\delta} \nonumber\\
 &\leq \left( 1 + \varepsilon \right)\mathbb{E}\left[\sum_{\ell = 0}^{\mu_{T}-1}g_{J,\ell}\right] + \frac{21\left(1 + \varepsilon\right)^{2}}{\varepsilon}F^{2}a_{T}\log \mathcal{N}_{\delta}, \nonumber
\end{align}
where the last line follows from $1 \leq (1 + \varepsilon)/\varepsilon$.
%Thus, the inequality \eqref{ber 1} is shown.
\proofstep{Conclusion}
From \eqref{diff_ori_and_cou2} and \eqref{ber 1}, we have
\begin{align}
&\mathbb{E} \left[ \sum_{\ell=0}^{\mu_T - 1} \widetilde{g}_{J,\ell}^\prime \right] \leq \mathbb{E} \left[ \sum_{\ell=0}^{\mu_T - 1} g_{J,\ell}^\prime \right] + 4F^2 a_T \mu_T \beta(a_T) \nonumber\\
&\leq (1+\varepsilon) \mathbb{E} \left[ \sum_{\ell=0}^{\mu_T - 1} g_{J,\ell} \right]
+ \frac{21(1+\varepsilon)^2}{\varepsilon} F^2 a_T \log \mathcal{N}_{\delta} + 4F^2 a_T \mu_T \beta(a_T) \nonumber\\
&\leq (1+\varepsilon) \mathbb{E} \left[ \sum_{\ell=0}^{\mu_T - 1} \widetilde{g}_{J,\ell} \right]
+ \frac{21(1+\varepsilon)^2}{\varepsilon} F^2 a_T \log \mathcal{N}_{\delta} + 4(2+\varepsilon) F^2 a_T \mu_T \beta(a_T). \label{temp1}
\end{align}

Additionally, define
\begin{align}
\widetilde{h}_{J,\ell} &:= \sum_{t \in I_{2,\ell}} 
\left\| \Psi_{J}(X_{t}) - \bar{\Psi}_{0}^{(M)}(X_{t})\right\|_{\mathcal{H}}^{2}, \quad \widetilde{h}_{J,\ell}^{\prime} := \sum_{t \in I_{2,\ell}} 
\left\|\Psi_{J}(X_{t}^{\prime}) - \bar{\Psi}_{0}^{(M)}(X_{t}^{\prime})\right\|_{\mathcal{H}}^{2}. \nonumber
\end{align}

Then a similar argument to derive \eqref{temp1} yields
\begin{align}
\mathbb{E} \left[ \sum_{\ell=0}^{\mu_T - 1} \widetilde{h}_{J,\ell}^\prime \right]
\leq& (1+\varepsilon) \mathbb{E} \left[ \sum_{\ell=0}^{\mu_T - 1} \widetilde{h}_{J,\ell} \right]
+ \frac{21(1+\varepsilon)^2}{\varepsilon} F^2 a_T \log \mathcal{N}_{\delta} \nonumber\\
&+ 4(2+\varepsilon) F^2 a_T \mu_T \beta(a_T). \label{temp2}
\end{align}

Now, note that
\begin{align*}
R\left(\Psi_{J}, \bar{\Psi}_{0}^{(M)}\right) &
= \frac{1}{T} \mathbb{E} \left[
\sum_{\ell = 0}^{\mu_{T} - 1} 
\widetilde{g}_{J,\ell}^{\prime} + \sum_{\ell = 0}^{\mu_{T} - 1}
 \widetilde{h}_{J,\ell}^{\prime}
+ \sum_{t = 2a_T \mu_T + 1}^{T} 
\left\|\Psi_{J}(X_{t}^{\prime}) - \bar{\Psi}_{0}^{(M)}(X_t^\prime)\right\|_{\mathcal{H}}^{2}
\right], \\
\widehat{R}\left(\Psi_{J},\bar{\Psi}_{0}^{(M)}\right) &= \frac{1}{T} \mathbb{E} \left[
\sum_{\ell=0}^{\mu_T - 1} \widetilde{g}_{J,\ell} 
+ \sum_{\ell=0}^{\mu_T - 1} \widetilde{h}_{J,\ell}
+ \sum_{t=2a_T \mu_T + 1}^{T} 
\left\|\Psi_{J}(X_t) - \bar{\Psi}_{0}^{(M)}(X_{t})\right\|_{\mathcal{H}}^{2}. \nonumber
\right].
\end{align*}

Together with \eqref{temp1} and \eqref{temp2}, we have
\begin{align*}
R\left(\Psi_J, \bar{\Psi}_{0}^{(M)}\right) &\leq \frac{1}{T} \left\{ (1+\varepsilon) \mathbb{E} \left[
\sum_{\ell = 0}^{\mu_{T} - 1} 
\widetilde{g}_{J,\ell} + \sum_{\ell = 0}^{\mu_T - 1} \widetilde{h}_{J,\ell}
\right]
+ \mathbb{E} \left[ \sum_{t=2a_T \mu_{T} + 1}^{T} 
\left\|\Psi_{J}(X_t^\prime) - \bar{\Psi}_{0}^{(M)}(X_{t}^{\prime})\right\|_{\mathcal{H}}^{2} \right] \right\} \\
&\quad + \frac{2}{T} \left(
\frac{21(1+\varepsilon)^2}{\varepsilon} F^2 a_T \log \mathcal{N}_{\delta} + 4(2+\varepsilon) F^2 a_T \mu_T \beta(a_T)
\right) \\
&\leq (1+\varepsilon) \widehat{R}\left(\Psi_{J}, \bar{\Psi}_{0}^{(M)}\right) + \frac{21(1+\varepsilon)^2}{\varepsilon} \frac{F^2 \log \mathcal{N}_{\delta}}{\mu_T}
+ 4(2+\varepsilon) F^2 \beta(a_T) \\
&\quad + \left\lvert \frac{1}{T} \mathbb{E} \left[
\sum_{t=2a_T \mu_T + 1}^{T} \left\{
\left\|\Psi_{J}(X_{t}^{\prime}) - \bar{\Psi}_{0}^{(M)}(X_{t}^{\prime})\right\|_{\mathcal{H}}^{2} 
- 
\left\|\Psi_{J}(X_{t}) - \bar{\Psi}_{0}^{(M)}(X_t)\right\|_{\mathcal{H}}^{2}
\right\} \right] \right\rvert.\nonumber
\end{align*}
Since
\begin{align}
&\left\lvert \left\|\Psi_{J}(X_{t}) - \bar{\Psi}_{0}^{(M)}(X_{t})
\right\|_{\mathcal{H}}^2 - \left\|\Psi_{J}(X_{t}^{\prime}) - \bar{\Psi}_{0}^{(M)}(X_{t}^{\prime})\right\|_{\mathcal{H}}^2\right\rvert \leq 4F^2, \nonumber\\
&\left\lvert \left\| \widehat{\Psi}_{T}(x) - \bar{\Psi}_{0}^{(M)}(x)\right\|_{\mathcal{H}}^{2} 
- 
\left\|\Psi_{J}(x) - \bar{\Psi}_{0}^{(M)}(x)\right\|_{\mathcal{H}}^{2}\right\rvert \leq 4F \delta,\quad x \in \mathcal{H}, \label{lemmac1_1}
\end{align}
we obtain
\begin{align}
R\left(\widehat{\Psi}_{T}, \bar{\Psi}_{0}^{(M)}\right) &\leq 4F \delta + (1+\varepsilon) \left\{
\widehat{R}\left(\widehat{\Psi}_T, \bar{\Psi}_{0}^{(M)}\right) + 4F \delta
\right\}
+ \frac{21(1+\varepsilon)^2}{\varepsilon} \frac{F^2 \log \mathcal{N}_{\delta}}{\mu_T} \nonumber\\
&\quad + 4(2+\varepsilon) F^2 \beta(a_T)
+ \frac{4F^2 \cdot 2a_T}{T} \notag \\
&\leq (1+\varepsilon) \widehat{R}\left(\widehat{\Psi}_{T}, \bar{\Psi}_{0}^{(M)}\right) + \frac{21(1+\varepsilon)^2}{\varepsilon} \frac{F^2 \log \mathcal{N}_{\delta}}{\mu_T}
+ \frac{4F^2}{\mu_T} \nonumber\\
&\quad + 4(2+\varepsilon) F^2 \beta(a_T) + 4(2+\varepsilon) F \delta. \nonumber
\end{align}
%Here, we obtain the desired result.
\end{proof}

We now prove Lemma \ref{lemmac2}
\begin{proof}[Proof of Lemma \ref{lemmac2}]
Let $\delta$ be any positive number and
$\{\Psi_{1},\dots,\Psi_{\mathcal{N}_{\delta}}\}$ be a $\delta$ covering of $\mathcal{F}$ with respect to $\left\| \cdot \right\|_{\infty}$.

\newcounter{step}
\stepcounter{step} % Reset for each proof
\noindent\textbf{Step \thestep.}
First, we show that for any $\Psi \in \mathcal{F}$:
\begin{align}
\widehat{R}\left(\widehat{\Psi}_{T},\bar{\Psi}_{0}^{(M)}\right) &= \Delta(\widehat{\Psi}_{T},\Psi) + R\left(\Psi,\bar{\Psi}_{0}^{(M)}\right) + 2 \mathbb{E}\left[\frac{1}{T}\sum_{t = 1}^{T}\left\langle \widehat{\Psi}_{T}(X_{t}) - \bar{\Psi}_{0}^{(M)}(X_{t}), \xi_{t}\right\rangle_{\mathcal{H}}\right]. \label{c2_0}
\end{align}
Indeed, we obtain that
\begin{align}
    \left\| X_{t+1} \right\|_{\mathcal{H}}^{2} - 
    \left\| \bar{\Psi}_{0}^{(M)}(X_{t})\right\|_{\mathcal{H}}^{2}
    =& \left\| X_{t+1} - \widehat{\Psi}_{T}(X_{t})\right\|_{\mathcal{H}}^{2} 
    - \left\|\widehat{\Psi}_{T}(X_{t}) - \bar{\Psi}_{0}^{(M)}(X_{t}) \right\|_{\mathcal{H}}^{2} \nonumber\\
    &+ 2 \left\langle \widehat{\Psi}_{T}(X_{t}) , \Psi_{0}(X_{t}) - \bar{\Psi}_{0}^{(M)}(X_{t}) + \xi_{t} \right\rangle_{\mathcal{H}}, \nonumber\\
    =& \left\| X_{t+1} - \widehat{\Psi}_{T}(X_{t})\right\|_{\mathcal{H}}^{2} 
    - \left\|\widehat{\Psi}_{T}(X_{t}) - \bar{\Psi}_{0}^{(M)}(X_{t}) \right\|_{\mathcal{H}}^{2} \nonumber\\
    &+ 2 \left\langle \widehat{\Psi}_{T}(X_{t}) , \xi_{t} \right\rangle_{\mathcal{H}}, \nonumber
\end{align}
where the second equality follows since
\begin{align}
    &\left\langle \widehat{\Psi}_{T}(X_{t}) , 
    \Psi_{0}(X_{t}) - \bar{\Psi}_{0}^{(M)}(X_{t}) + \xi_{t} 
    \right\rangle_{\mathcal{H}} \nonumber\\
    &= \int_{D} \left(\widehat{\Psi}_{T}(X_{t})\right)(u)
     \left(\Psi_{0}(X_{t})\bm{1}_{\{\left\| X_{t} \right\|_{\infty} > M\}}\right)(u)\,du
     + \left\langle \widehat{\Psi}_{T}(X_{t}) , \xi_{t} \right\rangle_{\mathcal{H}}, \nonumber
\end{align}
and the right-hand side equals zero, the latter being a consequence of Assumption 4.3.
We also obtain that for any $\Psi \in \mathcal{F}$
\begin{align}
    \left\| X_{t+1} \right\|_{\mathcal{H}}^{2} - \left\| \bar{\Psi}_{0}^{(M)}(X_{t})\right\|_{\mathcal{H}}^{2}
    =& \left\| X_{t+1} - \Psi(X_{t}) + \Psi(X_{t})\right\|_{\mathcal{H}}^{2} \nonumber\\
    &- \left\| \Psi(X_{t}) - \Psi(X_{t}) 
    + \bar{\Psi}_{0}^{(M)}(X_{t})\right\|_{\mathcal{H}}^{2} \nonumber\\ 
    =& \left\| X_{t+1} - \Psi(X_{t})\right\|_{\mathcal{H}}^{2} 
    - \left\|\Psi(X_{t}) - \bar{\Psi}_{0}^{(M)}(X_{t}) \right\|_{\mathcal{H}}^{2} \nonumber\\
    &+ 2\left\langle \Psi(X_{t}),\xi_{t} \right\rangle_{\mathcal{H}}. \nonumber
\end{align}
Thus, we obtain that
\begin{align}
 &\mathbb{E}\left[
    \frac{1}{T}\sum_{t = 1}^{T}\left\| X_{t+1} - \widehat{\Psi}_{T}(X_{t})
    \right\|_{\mathcal{H}}^{2} \right] - 
    \mathbb{E}\left[\frac{1}{T}\sum_{t = 1}^{T}
    \left\|\widehat{\Psi}_{T}(X_{t}) - \bar{\Psi}_{0}^{(M)}(X_{t}) \right\|_{\mathcal{H}}^{2} 
    \right] -2 \mathbb{E}\left[\frac{1}{T}\sum_{t = 1}^{T}\left\langle 
    \widehat{\Psi}_{T}(X_{t}), \xi_{t} \right\rangle_{\mathcal{H}}\right] \nonumber\\
 &= \mathbb{E}\left[\frac{1}{T}\sum_{t = 1}^{T}
 \left\| X_{t+1} - \Psi(X_{t})\right\|_{\mathcal{H}}^{2} \right]
  -  \mathbb{E}\left[\frac{1}{T}\sum_{t = 1}^{T}
  \left\|\Psi(X_{t}) - \bar{\Psi}_{0}^{(M)}(X_{t}) \right\|_{\mathcal{H}}^{2}  \right], \nonumber
\end{align}
which implies that
\begin{align}
&\widehat{R}\left(\widehat{\Psi}_{T},\bar{\Psi}_{0}^{(M)}\right) = \Delta (\widehat{\Psi}_{T},\Psi) + R\left(\Psi,\bar{\Psi}_{0}^{(M)}\right) + 2 \mathbb{E}\left[\frac{1}{T}\sum_{t = 1}^{T}\left\langle \widehat{\Psi}_{T}(X_{t}) - \bar{\Psi}_{0}^{(M)}(X_{t}), \xi_{t}\right\rangle_{\mathcal{H}}\right]. \nonumber
\end{align}

\stepcounter{step}
\noindent\textbf{Step \thestep.}
In this step, we will show
\begin{align}
&\mathbb{E}\left[\frac{1}{T}\sum_{t=1}^{T}\left \langle \Psi_{J}(X_{t})-\bar{\Psi}_{0}^{(M)}(X_t),
\xi_{t}\right \rangle_{\mathcal{H}}\right] \nonumber\\
&\leq 4K\left(\frac{(\log T)(\log\mathcal{N}_{\delta}+\tfrac12\log2)}{T}\right)^{1/2}
\left(\widehat R\left(\widehat{\Psi}_{T},\bar{\Psi}_{0}^{(M)}\right)+F\delta+\frac{4F^2}{T}\right)^{1/2}.
\label{c2_1}
\end{align}
For each $j=1,\dots,\mathcal{N}_{\delta}$, define
\begin{align}
A_{j}&:=\sum_{t=1}^{T}\left \langle \Psi_{j}(X_{t})- \bar{\Psi}_{0}^{(M)}(X_t),\xi_t\right\rangle_{\mathcal{H}},\nonumber\\
B_{j}&:=\left(\sum_{t=1}^{T}\left\{\left\langle \Psi_{j}(X_t) - \bar{\Psi}_{0}^{(M)}(X_t),\xi_{t} \right\rangle_{\mathcal{H}}^{2} + \left\langle Q\left(\Psi_{j}(X_{t}) - \bar{\Psi}_{0}^{(M)}(X_{t})\right),\Psi_{j}(X_{t}) - \bar{\Psi}_{0}^{(M)}(X_{t})\right\rangle_{\mathcal{H}}\right\}\right)^{1/2}, \nonumber
\end{align}
and
\begin{align}
\xi_j:=\frac{A_j}{2\sqrt{B_j^2+\mathbb{E}\bigl[B_J^2\bigr]}}.\nonumber
\end{align}
If  the denominator is zero, then $\xi_{j}:=0$. By the Cauchy - Schwarz inequality,
\begin{align}
\mathbb{E}\left[\frac{1}{T}\sum_{t=1}^{T}\left\langle \Psi_{J}(X_t) - \bar{\Psi}_{0}^{(M)}(X_t),\xi_{t} \right\rangle_{\mathcal{H}}\right]
&\leq \frac{2}{T} \mathbb{E}\left[\lvert\xi_{J}\rvert 
\left(B_{J}^{2} + \mathbb{E}[B_{J}^{2}]\right)^{1/2}\right] \nonumber\\
&\leq \frac{2}{T} \mathbb{E}\left[\xi_{J}^2\right]^{1/2}(2 \mathbb{E}[B_{J}^{2}])^{1/2}.\label{c2_2}
\end{align}
If $\mathbb{E}[B_J^2]>0$, then, by the fact $X_{t+1}$ is $\mathcal{F}_{t}$-measurable for $t = 1,\dots,T$ and from Lemma~\ref{lemmae1} in the supplementary material, and
Lemma~\ref{lemmae2} therein applied with $y=\sqrt{\mathbb{E}[B_J^2]}$,
\begin{align}
\mathbb{E}\left[\frac{\sqrt{\mathbb{E}[B_J^2]}}{\sqrt{B_{j}^{2}+\mathbb{E}[B_{J}^{2}]}}
\exp\left(2\xi_{j}^{2}\right)\right] \leq 1. \nonumber
\end{align}
Hence
\begin{align}
\mathbb{E}\left[\frac{\sqrt{\mathbb{E}[B_{J}^{2}]}}{\sqrt{B_{J}^{2}+\mathbb{E}[B_{J}^{2}]}}
\exp \left(2\xi_{J}^{2}\right)\right]
\le\mathbb{E}\left[\max_{1\le j\le\mathcal{N}_{\delta}}
\frac{\sqrt{\mathbb{E}[B_j^2]}}{\sqrt{B_j^2+\mathbb{E}[B_j^2]}}
\exp\bigl(2\xi_j^2\bigr)\right]\leq \mathcal{N}_{\delta}. \nonumber
\end{align}
By the Cauchy - Schwarz inequality,
\begin{align}
&\mathbb{E}\left[\exp(\xi_{J}^{2})\right]
\leq
\sqrt{
  \mathbb{E}\left[\frac{\sqrt{\mathbb{E}[B_{J}^{2}]}}{
    \sqrt{B_{J}^{2}+\mathbb{E}[B_{J}^2]}}
  \exp(2\xi_{J}^{2})\right] 
  \mathbb{E}\left[\frac{\sqrt{B_{J}^{2}+\mathbb{E}[B_{J}^{2}]}}{\sqrt{\mathbb{E}[B_{J}^{2}]}}\right]
}
\leq
\sqrt{\mathcal{N}_{\delta}
\mathbb{E}\left[\frac{\sqrt{B_{J}^{2}+\mathbb{E}[B_{J}^{2}]}}{\sqrt{\mathbb{E}[B_{J}^{2}]}}\right]
}.\nonumber
\end{align}
Since
\begin{align}
\mathbb{E}\left[\frac{\sqrt{B_{J}^{2}+\mathbb{E}[B_{J}^{2}]}}{\sqrt{\mathbb{E}[B_{J}^{2}]}}\right]
\leq
\sqrt{\mathbb{E}\left[\frac{B_{J}^{2}+\mathbb{E}[B_{J}^{2}]}{\mathbb{E}[B_{J}^{2}]}\right]}
\leq \sqrt{2},\nonumber
\end{align}
we conclude
\begin{align}
&\mathbb{E}\left[\exp(\xi_J^2)\right]\leq 2^{1/4}\sqrt{\mathcal{N}_{\delta}}. \nonumber
\end{align}
This bound also holds if $\mathbb{E}[B_j^2]=0$ (then $\xi_j=0$). Finally, by Jensen's inequality,
\begin{equation}
\mathbb{E}\left[\xi_{J}^{2}\right]
\leq \log\mathbb{E}\left[\exp(\xi_{J}^{2})\right]
\leq \frac{1}{2}\log\mathcal{N}_{\delta}+\frac{1}{4}\log2.
\label{c2_3}
\end{equation}
%----------- Previous part -----------%
Using Assumption 2.1 and \eqref{lemmac1_1}, we have

%----------- (C.13) -----------%
\begin{align}
\mathbb{E}[B_{J}^{2}]
\leq& \mathbb{E}\left[\sum_{t=1}^{T}
  \left\lvert\left \langle \Psi_{J}(X_{t})
  -
  \bar{\Psi}_{0}^{(M)}(X_{t}),\xi_{t} \right \rangle_{\mathcal{H}}\right\rvert^{2} \right] \nonumber\\
  &+\mathbb{E}\left[
    \sum_{t = 1}^{T}\left\lvert\left \langle \Psi_{J}(X_{t})
  -
  \bar{\Psi}_{0}^{(M)}(X_{t}),Q\left(\Psi_{J}(X_{t})
  -
  \bar{\Psi}_{0}^{(M)}(X_{t})\right) \right \rangle_{\mathcal{H}}\right\rvert^{2} \right]\nonumber\\
\leq&  \mathbb{E}\left[\sum_{t=1}^{T}
  \left\lvert\left \langle \widehat{\Psi}_{T}(X_{t})
  -
  \bar{\Psi}_{0}^{(M)}(X_{t}),\xi_{t} \right \rangle_{\mathcal{H}}\right\rvert^{2}\right]\nonumber\\
  &+
    \mathbb{E}\left[\sum_{t = 1}^{T}\left\lvert\left \langle \widehat{\Psi}_{T}(X_{t})
  -
  \bar{\Psi}_{0}^{(M)}(X_{t}),Q\left(  \widehat{\Psi}_{T}(X_{t})
  -
  \bar{\Psi}_{0}^{(M)}(X_{t})\right) \right \rangle_{\mathcal{H}}\right\rvert^{2} \right]\nonumber\\
&+ 8F^{2}\delta T \left\| Q\right\|_{\mathcal{L}}\nonumber\\
\leq&\mathbb{E}\left[
  \sum_{t=1}^{T}\left\| \widehat{\Psi}_{T}(X_t) - \bar{\Psi}_{0}^{(M)}(X_t)\right\|_{\mathcal{H}}^{2}\left\| \xi_{t} \right\|_{\mathcal{H}}^{2}\right]
  + \left\| Q \right\|_{\mathcal{L}} \widehat R\left(\widehat{\Psi}_{T},\bar{\Psi}_{0}^{(M)}\right)
  + 8F^{2}\delta T \left\| Q\right\|_{\mathcal{L}}.\label{c2_4}
\end{align}

%----------- (C.14) -----------%
Decompose the first term as
\begin{align}
&\mathbb{E}\left[\sum_{t=1}^{T}\left\|\widehat{\Psi}_{T}(X_t) - \bar{\Psi}_{0}^{(M)}(X_{t})\right\|_{\mathcal{H}}^{2} \left\| \xi_{t} \right\|_{\mathcal{H}}^{2}\right]\nonumber\\
&\leq \
\mathbb{E}
\left[
  \sum_{t=1}^{T}
  \left\|\widehat{\Psi}_{T}(X_t) - \bar{\Psi}_{0}^{(M)}(X_t)\right\|_{\mathcal{H}}^{2}
   \left\| \xi_{t} \right\|_{\mathcal{H}}^{2}\bm{1}_{\left\{\left\| \xi_{t} \right\|_{\mathcal{H}} \leq \sqrt{2 (K^{'})^{2} \log T} \right\}}\right]
   \nonumber\\
&\hspace{11pt}+ \mathbb{E}\left[\sum_{t=1}^{T}
  \left\|\widehat{\Psi}_{T}(X_t) - \bar{\Psi}_{0}^{(M)}(X_t)\right\|_{\mathcal{H}}^{2}
   \left\| \xi_{t} \right\|_{\mathcal{H}}^{2}\bm{1}_{\left\{\left\| \xi_{t} \right\|_{\mathcal{H}} > \sqrt{2 (K^{'})^{2} \log T} \right\}}\right]\nonumber\\
&\leq 2 (K^{'})^{2} T \log T \widehat{R}\left(\widehat{\Psi}_{T},\bar{\Psi}_{0}^{(M)}\right)
   +4F^{2}\left(K^{'}\right)^{2}T
    \mathbb{E}\left[\frac{\left\| \xi_{t} \right\|^{2}}{2(K^{'})^{2}}
    \mathbf{1}_{\{\left\|\xi_{t}\right\| > \sqrt{2 (K^{'})^{2} \log T}\}}\right].\label{c2_6}
\end{align}

%----------- (C.15) -----------%
Since $K'>0$ satisfies (4.2) in the main paper,
\begin{align}
\mathbb{E}\left[\frac{\left\| \xi_{t} \right\|_{\mathcal{H}}^{2}}{2(K^{'})^{2}}\mathbf{1}_{\left\{\left\|\xi_{t}\right\|_{\mathcal{H}} > \sqrt{2 (K^{'})^{2}\log T} \right\}}\right]
\leq& 
\mathbb{E}\left[\exp\left(\frac{\left\| \xi_{t} \right\|_{\mathcal{H}}^{2}}{2(K^{'})^{2}}\right)
\bm{1}_{\left\{\left\|\xi_{t}\right\|_{\mathcal{H}} > \sqrt{2 (K^{'})^{2} \log T}\right\}}\right] \nonumber\\
\leq& 
\mathbb{E}\left[
  \exp\left\{\frac{\left\| \xi_{t}\right\|_{\mathcal{H}}^{2}}
  { (K^{'})^{2}}\right\}
  \exp\left\{-\frac{\left\|\xi_{t}\right\|_{\mathcal{H}}^{2}}
  {2 (K^{'})^{2}}\right\}
  \bm{1}_{\left\{\left\|\xi_{t}\right\|_{\mathcal{H}} > \sqrt{2 (K^{'})^{2} \log T}\right\}}
  \right] \nonumber\\
\leq& 
\mathbb{E}\left[  \exp\left\{\frac{\left\| \xi_{t}\right\|_{\mathcal{H}}^{2}}
  { (K^{'})^{2}}\right\}
  \exp\{-\log T\}\right]
\leq \frac{2}{T}.
\label{c2_5}
\end{align}

%----------- (C.16) -----------%
Combining inequalities \eqref{c2_4}--\eqref{c2_5} with the assumption $\log T\ge 2$ and the definition of $K$, we obtain
\begin{align}
\mathbb{E}[B_{J}^{2}]
&\leq 2(K^{'})^{2}T\log T \widehat{R}\left(\widehat{\Psi}_{T},\bar{\Psi}_{0}^{(M)}\right) + 8F^{2}(K^{'})^{2} + \left\| Q \right\|_{\mathcal{L}}\widehat{R}\left(\widehat{\Psi}_{T},\bar{\Psi}_{0}^{(M)}\right) + 8F^{2}\delta T \left\| Q \right\|_{\mathcal{L}} \nonumber\\
&\leq 4K^{2}T\log T \widehat{R}\left(\widehat{\Psi}_{T},\bar{\Psi}_{0}^{(M)}\right) + 16F^{2}K^{2} + 8F\delta T K^{2} \nonumber\\
&\leq 4K^{2}T\log T \left(\widehat{R}\left(\widehat{\Psi}_{T},\bar{\Psi}_{0}^{(M)}\right) + \frac{4F^{2}}{T} + F\delta \right). 
\label{c2_6}
\end{align}
Combining \eqref{c2_2}, \eqref{c2_3}, and \eqref{c2_6}, we obtain the inequality \eqref{c2_1}.

%----------- Step 3 -----------%
\stepcounter{step}
\noindent\textbf{Step \thestep.}
In this step, we complete the proof. By \eqref{c2_1} and the AM - GM inequality,
\begin{align}
\mathbb{E}\left[
  \frac{1}{T}\sum_{t=1}^{T}\left\langle \Psi_{J}(X_{t}) - \bar{\Psi}_{0}^{(M)}(X_{t}), \xi_{t} \right\rangle_{\mathcal{H}}\right]
\le \frac{\varepsilon}{2} \widehat{R}\left(\widehat{\Psi}_{T} , \bar{\Psi}_{0}^{(M)}\right)
  +\gamma_\varepsilon,\nonumber
\end{align}
where
\[
\gamma_\varepsilon
=\frac{16 K^{2} (\log T)\left(2\log\mathcal{N}_{\delta}+\log2\right)}{\varepsilon\,T}
+\frac{F\delta\,\varepsilon}{2}
+\frac{2F^2\varepsilon}{T}.
\]
Combining this with \eqref{c2_0}, we have for any $\Psi \in\mathcal F$,
\begin{align}
\widehat{R}\left(\widehat{\Psi}_{T}, \bar{\Psi}_{0}^{(M)}\right)
&= \Delta(\widehat{\Psi}_{T},\Psi) + \widehat{R}\left(\Psi, \bar{\Psi}_{0}^{(M)}\right)
  +2\,\mathbb E\left[\frac{1}{T}\sum_{t = 1}^{T}\left\langle \widehat{\Psi}_{T}(X_t)
  - \Psi_{J}(X_t),\xi_{t}\right\rangle_{\mathcal{H}}\right] \nonumber\\
&\quad +2\mathbb{E}\left[\frac{1}{T}\sum_{t=1}^{T}
\left\langle \Psi_{J}(X_{t}) - \bar{\Psi}_{0}^{(M)}(X_{t}),\xi_{t}\right\rangle_{\mathcal{H}}\right] \nonumber\\
&\leq \Delta\left(\widehat{\Psi}_{T}\right)+\widehat R\left(\Psi, \bar{\Psi}_{0}^{(M)}\right)
  + 2 \delta \sqrt{\left\| Q \right\|_{\mathbb{T}}}
  +\varepsilon\,\widehat{R}\left(\widehat{\Psi}_{T},\bar{\Psi}_{0}^{(M)}\right)
  +2\gamma_\varepsilon \nonumber\\
&\leq \Delta(\widehat{\Psi}_{T}) + \widehat{R}\left(\Psi, \bar{\Psi}_{0}^{(M)}\right)
  + 2 \delta K
  + \varepsilon\,\widehat R\left(\widehat{\Psi}_{T},\bar{\Psi}_{0}^{(M)}\right)
  + 2\gamma_\varepsilon. \nonumber
\end{align}
Since $\widehat R\left(\Psi,\bar{\Psi}_{0}^{(M)}\right) = R\left(\Psi, \bar{\Psi}_{0}^{(M)}\right)$, it follows that
\begin{align}
(1-\varepsilon) \widehat{R}\left(\widehat{\Psi}_{T},\bar{\Psi}_{0}^{(M)}\right)
&\leq \Delta(\widehat{\Psi}_{T}) + R\left(\Psi,\bar{\Psi}_{0}^{(M)}\right)+\gamma_\varepsilon^{'},\nonumber
\end{align}
where $\gamma_{\varepsilon}^{'} := 2\delta K + 2\gamma_{\varepsilon}$.
Taking the infimum over $\Psi \in \mathcal{F}$, we conclude that 
\begin{align}
 \widehat{R}\left(\widehat{\Psi}_{T},\bar{\Psi}_{0}^{(M)}\right)&\leq \frac{1}{1 - \varepsilon}\Delta(\widehat{\Psi}_{T}) + \frac{1}{1 - \varepsilon}\inf_{\Psi \in \mathcal{F}}R\left(\Psi,\bar{\Psi}_{0}^{(M)}\right) + \frac{1}{1 - \varepsilon}\gamma_{\varepsilon}^{'}. \nonumber
\end{align}
Noting that $F \geq 1$, we obtain the desired result.
\end{proof}

\subsection{Proof of Lemma A.2}
In this section, we provide the proof of Lemma A.2. 
% Let $\mathcal{H}$ be a separable Hilbert space. For any $x,y \in \mathcal{H}$, we define the operator $x \otimes y : \mathcal{H} \to \mathcal{H}$ by
% $(x \otimes y)(f) = x \langle y,f \rangle_{\mathcal{H}}, f \in \mathcal{H}$.
We add some notation needed to prove Lemma A.2.
For any $f, g \in \mathcal{H}$, define the operator $f \otimes g : \mathcal{H} \to \mathcal{H}$ by 
$
(f \otimes g)(h) = f \langle g,h \rangle_{\mathcal{H}},\,h \in \mathcal{H}
$.
\begin{proof}[Proof of Lemma A.2]
Since $\mathcal{H}$ is a separable Hilbert space, there exist eigenvalues $\{\lambda_i\}_{i=1}^{\infty}$ and eigenfunctions $\{e_i\}_{i=1}^{\infty}$ of $Q$ such that
\[
Q = \sum_{i=1}^{\infty}\lambda_i \,e_i\otimes e_i.
\]
By the discussion in Example 4.2 of Section 2, Chapter 4 of \citet{lifshits2012gaussian}, for any $h\in \mathcal{H}_{k}$,
\begin{align}
    I^{-1}h = \sum_{i=1}^{\infty}\frac{\langle e_i,h\rangle_{\mathcal{H}}}{\lambda_i}\,e_i. \label{form_of_I_inv}
\end{align}
Hence, for $x \in \mathcal{H}$,
\begin{align*}
    (I^{-1}h)(x)
    &= \sum_{i=1}^{\infty}\frac{\langle e_i,Qg\rangle_{\mathcal{H}}}{\lambda_i}\langle e_i,x\rangle_{\mathcal{H}}
     = \sum_{i=1}^{\infty}\langle e_i,g\rangle_{\mathcal{H}}\langle e_i,x\rangle_{\mathcal{H}}
     = \langle g,x\rangle_{\mathcal{H}},
\end{align*}
which proves (A.2). Moreover, from \eqref{form_of_I_inv},
\begin{align*}
    \|h\|_{\mathcal{H}_{k}}^{2}
    &= \sum_{i=1}^{\infty}\frac{\lvert \langle h,e_i\rangle_{\mathcal{H}}\rvert^{2}}{\lambda_i}
     = \sum_{i=1}^{\infty}\frac{|\langle Qg,e_i\rangle_{\mathcal{H}}|^{2}}{\lambda_i}
     = \sum_{i=1}^{\infty}\frac{|\langle g,Qe_i\rangle_{\mathcal{H}}|^{2}}{\lambda_i}
     = \sum_{i=1}^{\infty}\lambda_i \,|\langle g,e_i\rangle_{\mathcal{H}}|^{2}
     = \langle Qg,g\rangle_{\mathcal{H}},
\end{align*}
establishing (A.3).
\end{proof}

\section{Technical lemmas}
In this section, we collect several technical lemmas used in the proofs.
Let $\mathcal{A}$ and $\mathcal{B}$ be two $\sigma$-fields of a probability space $(\Omega,\mathcal{T},\mathbb{P})$. 
The $\beta$-mixing coefficient between $\mathcal{A}$ and $\mathcal{B}$ is defined by
\[
\beta(\mathcal{A},\mathcal{B})
= \frac{1}{2} \sup \left\{ \sum_{i \in I}\sum_{j \in J} 
\left| \mathbb{P}(A_i \cap B_j) - \mathbb{P}(A_i)\,\mathbb{P}(B_j) \right| \right\},
\]
where the maximum is taken over all finite partitions 
$\{A_i\}_{i \in I} \subset \mathcal{A}$ and $\{B_j\}_{j \in J} \subset \mathcal{B}$ of $\Omega$. 
Moreover, for two probability measures $\mu$ and $\nu$ on a measurable space
$(\Omega,\mathcal{I})$, the total variation distance between them is defined by
\[
\|\nu-\mu\|_{\mathrm{TV}}
:=
\sup_{A\in\mathcal{I}} |\nu(A)-\mu(A)|.
\]
\begin{Lemma}[Lemma 5.1 in \citet{rio2013inequalities}]\label{rio}
Let $\mathcal{A}$ be a $\sigma$-field in a probability space $(\Omega, \mathcal{T}, \mathbb{P})$ and $X$ be a random variable with values in some Polish space. Let $U$ be a random variable with uniform distribution over $[0,1]$, independent of the $\sigma$-field generated by $X$ and $\mathcal{A}$. Then there exists a random variable $X^{\prime}$, with the same law as $X$, independent of $X$, such that 
\[
\mathbb{P}(X \neq X^{\prime}) = \beta(\mathcal{A}, \sigma(X)),
\]
where $\sigma(X)$ denotes the $\sigma$-field generated by $X$. Furthermore, $X^{\prime}$ is measurable with respect to the $\sigma$-field generated by $\mathcal{A}$ and $(X,U)$.
\end{Lemma}

\begin{Lemma}[Lemma E.3 in \citet{kurisu2024adaptive}]\label{lemmae1}
Let $\{d_i\}$ be a martingale difference sequence with respect to a filtration $\{\mathcal{F}_i\}$. Assume $\mathbb{E}[d_i^2] < \infty$ for all $i$. Then
\[
\mathbb{E}\left[\exp\left(\lambda \sum_{i=1}^n d_i - \frac{\lambda^2}{2}\left(\sum_{i=1}^n d_i^2 + \sum_{i=1}^n \mathbb{E}[d_i^2 \mid \mathcal{F}_{i-1}]\right)\right)\right] \leq 1
\]
for all $n \geq 1$ and $\lambda \in \mathbb{R}$.
\end{Lemma}

\begin{Lemma}[Theorem 1.2 in \citet{delapena2004self}]\label{lemmae2}
Let $B \geq 0$ and $A$ be two random variables satisfying
\[\mathbb{E}\left[\exp\left(\lambda A - \frac{\lambda^2}{2} B^2\right)\right] \leq 1 \quad \text{for all } \lambda \in \mathbb{R}.
\]
Then for all $y > 0$,
\[
\mathbb{E}\left[\frac{y}{\sqrt{B^2+y^2}} \exp\left(\frac{A^2}{2(B^2+y^2)}\right)\right] \leq 1.
\]
\end{Lemma}

% \begin{Lemma}[\textcite{hairer2011harris}]\label{harrier_theorem}
%     状態空間 $\left(\mathcal{X},\mathcal{B}(\mathcal{X})\right)$ に値をとるマルコフ連鎖 $\{X_{t}\}_{t \in \mathbb{Z}}$ が以下の仮定を満たすとする：
% \begin{itemize}
%     \item 可測写像 $V: \mathcal{X} \to \mathbb{R}^{+}$, $0 < \rho < 1$, および $\kappa \geq 0$ となる定数が存在して，
% \begin{align}
% \left(\mathcal{P}V\right)(x) &\leq \rho V(x) + \kappa,\quad x \in \mathcal{X}. \tag{A} \label{ca}
% \end{align}
% \item $R > 2\kappa/(1 - \rho)$ を満たすようなある定数 $R > 0$ に対して，
% $\left(\mathcal{X},\mathcal{B}(\mathcal{X})\right)$ 上の確率測度 $\nu_{R}$ と定数 $(0 <) \varepsilon_{R} (\leq 1)$ が存在して
% 任意の $x \in V_{R} := \left\{x \in \mathcal{X} \mid V(x) \leq R \right\}$ 上で以下が成り立つ：
% \begin{align}
%     &\mathcal{P}\left(x, \cdot \right) \geq \varepsilon \nu(\cdot). \tag{B} \label{cb}
% \end{align}
% \end{itemize}
% このとき $\{X_{t}\}_{t \in \mathbb{Z}}$ には一意な定常分布が存在し，
% さらに定数 $\lambda \in (0,1)$ と $C > 0$ が存在して 
% $
% \left\| f \right\|_{V} = \sup_{x \in \mathcal{X}}
% \frac{\lvert f(x) \rvert}{1 + V(x)}
% $ をみたす任意の可測写像
% $f:\mathcal{X} \to \mathbb{R}$ に対して
% \begin{align}
%     \left\lvert \mathcal{P}_{t}\psi - \pi\psi \right\rvert &\leq C\lambda^{t}\left(1 + V(x)\right)\left\| f \right\|_{V}, \nonumber
% \end{align}
% が任意の $t \in \mathbb{N}$ と任意の $x \in \mathcal{X}$ に対して成り立つ．
% \end{Lemma}

\begin{Lemma}[Theorem~1.2 in \citet{hairer2011harris}]\label{harrier_theorem}
Let $\{X_{t}\}_{t \in \mathbb{N}}$ be a Markov chain on $\big(\mathcal{H},\mathcal{B}(\mathcal{H})\big)$.
Suppose that there exist a measurable function $V:\mathcal{H}\to\mathbb{R}_{+}$ and constants $0<\rho<1$, $\kappa\ge 0$ such that
    \begin{align}
        (\mathcal{P}V)(x)\le \rho\,V(x)+\kappa,\quad x\in\mathcal{H}. \tag{A}\label{ca}
    \end{align}
Moreover, for some $R>2\kappa/(1-\rho)$ there exist a probability measure $\nu_{R}$ on $\big(\mathcal{H},\mathcal{B}(\mathcal{H})\big)$ and a constant $\varepsilon_{R}\in(0,1]$ such that, for all $x\in V_{R}:=\{x\in\mathcal{H}:V(x)\le R\}$,
    \begin{align}
        \mathcal{P}(x,\cdot)\ \ge\ \varepsilon_{R}\,\nu_{R}(\cdot). \tag{B}\label{cb}
    \end{align}
Then, the chain admits a unique invariant probability measure $\pi$, and there exist constants $\lambda\in(0,1)$ and $C>0$ such that, for every measurable $f:\mathcal{H}\to\mathbb{R}$ with
\[
\|f\|_{V}:=\sup_{x\in\mathcal{H}}\frac{\lvert f(x) \rvert}{1+V(x)}<\infty,
\]
we have for all $t\in\mathbb{N}$ and $x\in\mathcal{H}$,
\begin{align}
    \big|(\mathcal{P}^{t}f)(x)-\pi f\big|\ \le\ C\,\lambda^{t}\,(1+V(x))\,\|f\|_{V}. \nonumber
\end{align}
\end{Lemma}
%  We recall the geometric drift condition \emph{(V4)} in \citet{meyn_tweedie_2009}.
% Let $\mathsf{X}$ be a state space.
% There exist an extended-real valued function
% $V : \mathsf{X} \to [1,\infty)$,
% a measurable set $C \subset \mathsf{X}$,
% and constants $\beta > 0$ and $b < \infty$
% such that
% \begin{align}
% \mathcal{P}V(x) - V(x) 
% \leq -\beta V(x) + b \mathbf{1}_C(x),
% \qquad x \in \mathsf{X}. \label{V4}
% \end{align}
% All notation and terminology are as in 
% \citet{meyn_tweedie_2009}.
Lemmas \ref{drift_condition_dash} and \ref{drift_control} are restated from \citet{meyn_tweedie_2009}. We use the notation and terminology of \citet{meyn_tweedie_2009}.
Let $\boldsymbol{X} := \{X_t\}_{t \in \mathbb{N}}$ be a Markov chain on $\left(\mathcal{H},\mathcal{B}(\mathcal{H})\right)$. 
We begin by recalling the geometric drift condition 
$(\mathrm{V4})$ of \citet{meyn_tweedie_2009}. 
There exist an extended-real-valued function $V:\mathcal{H}\to[1,\infty]$, a measurable set $C\subset\mathcal{H}$, and constants $\beta>0$ and $b<\infty$ such that
\begin{align}
\mathcal{P}V(x)-V(x)
\leq -\beta V(x)+b\mathbf{1}_C(x),
\qquad x\in\mathcal{H}. \label{V4}
\end{align}
For any measurable set $C\subset\mathcal{H}$, let
$
\tau_C := \inf\{t\in\mathbb{N}: t \geq 2,  X_t\in C\}$. 
% For a measurable function $f:\mathsf{X}\to[1,\infty)$ and a probability measure $\nu$ on $(\mathsf{X},\mathcal{B}(\mathsf{X}))$, define
% \begin{align}
% \|\nu\|_f
% :=
% \sup_{\substack{g:\mathsf{X}\to\mathbb{R}\\ \lvert g(x) \rvert \leq f(x),\ x\in\mathsf{X}}}
% \left|\int_{\mathsf{X}} g(x)\,\nu(dx)\right|.
% \end{align}
% All notation and terminology are as in 
% \citet{meyn_tweedie_2009}.
\begin{Lemma}[Lemma 15.2.8 in \citet{meyn_tweedie_2009}]\label{drift_condition_dash}
The drift condition \eqref{V4} holds with a petite set $C$ if and only if
$V$ is unbounded off petite sets and
\begin{equation}\label{eq:linear_drift}
(\mathcal{P}V)(x) \le \lambda V(x) + L, \quad x \in \mathcal{H},
\end{equation}
for some $\lambda < 1$ and $L < \infty$.
\end{Lemma}

We shall use the following part of Theorem 14.1 of \citet{meyn_tweedie_2009}.

\begin{Lemma}\label{drift_control}
Suppose that $\boldsymbol{X}$ is $\psi$-irreducible and aperiodic, and let $f:\mathcal{H}\to[1,\infty)$ be a measurable function. Then the following conditions are equivalent:

\begin{enumerate}
    \item The chain is positive recurrent with invariant probability measure $\pi$ and
    \[
        \int_{\mathcal{H}} f(x)\pi(dx) < \infty.
    \]

    \item There exists some petite set $C \in \mathcal{B}(\mathcal{H})$ such that
    \[
        \sup_{x \in C} \mathbb{E} \left[ \sum_{t=1}^{\tau_{C}-1} f(X_t)\middle| X_{1} = x \right] < \infty.
    \]
    
    \item There exist a petite set $C$, an extended-real-valued nonnegative function 
$V$ satisfying $V(x_0)<\infty$ for some $x_0\in\mathcal{H}$, and a constant 
$b<\infty$ such that
    \[
        \mathcal{P}V(x) - V(x) \leq - f(x) + b \mathbf{1}_C(x), \quad x \in \mathcal{H}.
    \]
\end{enumerate}
% Any of these three conditions imply that the set 
% \[
%     S_V := \{ x : V(x) < \infty \}
% \]
% is absorbing and full, where $V$ is any solution to (iii) satisfying its conditions, and any level set of $V$ satisfies (ii). Moreover, for any $x \in S_V$,
% \[
%     \| \mathcal{P}^{t}(x, \cdot) - \pi \|_f \to 0 \quad \text{as } t \to \infty.
% \]

% Furthermore, if $\pi(V) < \infty$, then there exists a finite constant $B_f$ such that for all $x \in S_V$,
% \[
%     \sum_{t=1}^\infty \| \mathcal{P}^{t}(x, \cdot) - \pi \|_f \leq B_f \big( V(x) + 1 \big).
% \]
\end{Lemma}

\begin{Lemma}[Proposition 1 in \citet{davydov1973mixing}]\label{ergodic_to_mixing}
    Suppose that $\{X_{t}\}_{t \in \mathbb{N}}$ is stationary and admits the unique invariant measure. Then, the following statement holds:
    \begin{align}
        &\beta(j) = \int_{\mathcal{H}} \left\| \mathbb{P}(X_{j} \in \cdot) - \mathbb{P}(X_{j} \in \cdot \mid x)\right\|_{\mathrm{TV}}\pi(dx). \nonumber
    \end{align}
\end{Lemma}

\begin{Lemma}[Theorem 1 in \citet{Giorgobiani2019SubGaussian}]
Let $\xi$ be a separably valued Gaussian random element in $\mathcal{H}$.
Then there exists $\lambda > 0$ such that
\[
\mathbb{E}\exp\!\left(\lambda \|\xi\|_{\mathcal{H}}^2\right) < +\infty .
\]
\end{Lemma}

\begin{Lemma}[Proposition 8 in \citet{OhnKim2022}]
Let $M > 0$. Let $L \in \mathbb{N}$, $N \in \mathbb{N}$, $B \ge 1$, $F > 0$, and $S > 0$.
Then for any $\delta \in (0,1)$,
\[
\log N\!\left(
\delta,\,
\mathcal{F}_{p_{0}}^{(M)}(L,N,B,F,S),\,
\|\cdot\|_{\infty}
\right)
\le
2S(L+1)\log\!\left((L+1)(N+1)B\delta^{-1}\right).
\]
\end{Lemma}

\section{Details for numerical experiments in Section 5}
In this section, we describe the data-generating procedure and the learning method.
Let $\{\mathfrak{u}_{i}\}_{i=0}^{S_g-1}$ and $\{\mathfrak{v}_{i}\}_{i = 0}^{S_{g}-1}$ be the one-dimensional grid points defined by $\mathfrak{u}_{i} = \mathfrak{v}_{i} = i/S_{g}$. 
As discussed in Section~5, we construct a high dimensional time series 
$\boldsymbol{Z}_{T} = \{Z_{t}\}_{t = 1}^{T}$, where each observation is given by 
$Z_{t} = \left\{Z_{t}(\mathfrak{u}_{i}, \mathfrak{u}_{j})\right\}_{0 \leq i,j \leq S_{g}-1}$. 
The process $\boldsymbol{Z}_{T}$ is defined by
\begin{align}
    Z_{t+1}(\mathfrak{u}_{i_{1}},\mathfrak{u}_{j_{1}})
    &= \frac{1}{S_{g}^{2}}
       \sum_{i_{2},j_{2} = 0}^{S_{g}-1}
       5K_{G}(\mathfrak{u}_{i_{1}} - \mathfrak{v}_{i_{2}},\, \mathfrak{u}_{j_{1}} - \mathfrak{v}_{j_{2}})\,
       \tau\!\left(Z_{t}(\mathfrak{v}_{i_{2}}, \mathfrak{v}_{j_{2}})\right)
       + \xi_{t}(\mathfrak{u}_{i_{1}}, \mathfrak{u}_{j_{1}}), \label{discrete}
\end{align}
for each $t \in \mathbb{N}$ and $0 \leq i_{1},j_{1} \leq S_{g}-1$. Here, the kernel function is defined as 
$K_{G}(h_{1}, h_{2}) := \exp\{-5(h_{1}^{2} + h_{2}^{2})\}$ for $h_{1},h_{2} \in \mathbb{R}$, and  
the nonlinear transformation $\tau:\mathbb{R} \to \mathbb{R}$ is given by 
$\tau(x) := 1.5 + 2.5 \cos(x) + 2.0 \sin(2x)$ for $x \in \mathbb{R}$. The noise term $\xi_{t}$ is a zero-mean Gaussian process whose covariance function is given by the same kernel $K_{G}$, that is, 
$\mathbb{E}\left[\xi_{1}(u_{1},u_{2})\xi_{1}(v_{1},v_{2})\right] = K_{G}(u_{1}-v_{1},u_{2}-v_{2})$ for $u_{1},u_{2},v_{1},v_{2} \in [0,1]$. We now describe how we approximate stationarity of $\boldsymbol{Z}_{T}$ and generate the Gaussian noise process $\xi_{t} := \{\xi_{t}(\mathfrak{u}_{i}, \mathfrak{u}_{j})\}_{0 \leq i,j \leq S_{g}-1}$.
To approximate stationarity, we first generate a functional time series 
$\{Z_{t}\}_{t = 1}^{T + 500}$ according to \eqref{discrete},  
discard the first 500 observations as burn-in, 
and retain $\{Z_t\}_{t=501}^{T+500}$ as the simulated sample.
To generate the Gaussian process, we adopt the method introduced by \citet{DietrichNewsam1997}, and we restate the relevant steps here for completeness. 
We begin by introducing the necessary notation. For any integer $\ell_{\mathrm{per}} \ge 2$, define the function class
\begin{align}
&L_{\mathrm{per}}^{2}\left([0,\ell_{\mathrm{per}})^{2}\right)  := \left\{
f:\mathbb R^2\to\mathbb R
\,\middle|\,
\begin{aligned}
& f(x_1+\ell_{\mathrm{per}},x_2)=f(x_1,x_2),
    \quad (x_1,x_2)\in\mathbb R^2,\\
& f(x_1,x_2+\ell_{\mathrm{per}})=f(x_1,x_2),
    \quad (x_1,x_2)\in\mathbb R^2,\\
& \int_{[0,\ell_{\mathrm{per}})^2}
|f(x_1,x_2)|^2\,dx_1dx_2<\infty
\end{aligned}
\right\}.
\end{align}
For any $f \in L_{\mathrm{per}}^{2}\left([0,\ell_{\mathrm{per}})^{2}\right)$ and $z_{1}, z_{2} \in \mathbb{Z}$, define its Fourier coefficients by
\begin{align}
    \widehat{f}(z_{1}, z_{2})
    &= \frac{1}{\ell_{per}^{2}}\int_{[0,\ell_{\mathrm{per}})^{2}}
       f(x_{1},x_{2})
       \exp\!\left( -2\pi \sqrt{-1} \frac{z_{1}x_{1} + z_{2}x_{2}}{\ell_{\mathrm{per}}} \right)
       \, dx_{1}\, dx_{2},
       \nonumber
\end{align}
where $\sqrt{-1}$ denotes the imaginary unit.
We further introduce the periodic extension 
$K_{\mathrm{per}}:(-\ell_{\mathrm{per}},\ell_{\mathrm{per}})^{2}$ $\to \mathbb{R}_{\geq 0}$ defined by
\begin{align}
    K_{\mathrm{per}}(h_{1},\, h_{2})
    &:= \sum_{(p,q) \in \mathbb{Z}^{2}}
        K_{G}(h_{1} + \ell_{\mathrm{per}}p,\,
          h_{2} + \ell_{\mathrm{per}}q).
    \nonumber
\end{align}
Since $K_{G}$ is a Gaussian kernel, the above series is well defined. Similarly, define $K_{\mathrm{circ}}: [0,\ell_{\mathrm{per}})^{2} \to \mathbb{R}_{\geq 0}$ by
\begin{align}
    K_{\mathrm{circ}}(u_{1},u_{2})
    &:= 
    K_{G}\left(\min\{u_{1}, \ell_{\mathrm{per}}-u_{1}\}, \min\{ u_{2}, \ell_{\mathrm{per}} - u_{2}\}\right),
    \qquad (u_{1},u_{2}) \in [0,\ell_{\mathrm{per}})^{2}.
    \nonumber
\end{align}
Throughout the following discussion, unless otherwise specified, we take $\ell_{\mathrm{per}} = 2$. We introduce the remaining necessary notation. 
For any $z \in \mathbb{C}$, let 
$\mathrm{Re}(z)$ denote its real part.
Let $A = (a_{ij}) \in \mathbb{R}^{m \times n}$. 
The vectorization operator $\mathrm{vec} : \mathbb{R}^{m \times n} \to \mathbb{R}^{mn}$ 
is defined by stacking the columns of $A$ into a single vector:
$\mathrm{vec}(A)
    :=
    [
        a_{11}, a_{21}, \dots, a_{m1},
        a_{12}, a_{22}, \dots, a_{m2},\dots, a_{1n}, a_{2n}, \dots, a_{mn}
    ]^{\top}$.
    
We now describe the procedure for generating 
$\left\{\xi(\mathfrak{u}_{i},\mathfrak{u}_{j})\right\}_{0 \leq i,j \leq S_{g}-1}$. 
In this study, instead of directly generating $\xi$, we construct a Gaussian process 
$\xi_{\mathrm{per}} := \left(\xi_{\mathrm{per}}(u_{1},u_{2})\right)_{(u_{1},u_{2}) \in [0,2)^{2}}$ 
with periodicity on $[0,2)^{2}$. Specifically, we generate the evaluations 
$\{\xi_{\mathrm{per}}(\mathfrak{u}_{i},\mathfrak{u}_{j})\}_{0 \le i,j \le S_{g}-1}$ 
on the grid points $\{(\mathfrak{u}_{i},\mathfrak{u}_{j})\}_{0 \leq i,j \leq S_{g}-1}$.
 % Since $\xi_{\mathrm{per}}$ is a zero-mean Gaussian process on the torus $[0,2)^2$, it admits the representation
 The periodic Gaussian field $\xi_{\mathrm{per}}$ with covariance function \(K_{\mathrm{per}}\) is represented through the Fourier expansion
\begin{align}
    \xi_{\mathrm{per}}(u_{1},u_{2})
    = \mathrm{Re}\Bigg(
      \sum_{(p,q)\in \mathbb{Z}^{2}}
      \sqrt{\widehat{K}_{per}(p,q)}\,\tilde{z}_{pq}
      \exp\!\left(-2\pi \sqrt{-1} \frac{pu_{1} + qu_{2}}{2} \right)
      \Bigg), 
    \qquad (u_{1},u_{2}) \in [0,2)^{2}, 
    \label{fourier_transform_xi}
\end{align}
where $\{\tilde z_{pq}\}_{(p,q) \in \mathbb{Z}^{2}}$ is an i.i.d.\ sequence of complex-valued random variables such that the real and imaginary parts of $\tilde z_{pq}$ are independent standard normal random variables for each $p,q$. 
In Section~5, we approximate the infinite series in \eqref{fourier_transform_xi} by
\begin{align}
    \tilde{\xi}_{\mathrm{per}}(u_i,u_j)
    :=
    \operatorname{Re}\left(
    \sum_{p,q=0}^{2S_g-1}
       \sqrt{\widehat K_{\mathrm{per}}(p,q)}\,\tilde z_{pq}
       \exp\!\left(-2\pi\sqrt{-1}\frac{p u_i+q u_j}{2}\right)
    \right).\nonumber
\end{align}
The Fourier coefficients \(\widehat K_{\mathrm{per}}(p,q)\) are also approximated by the finite sums
\begin{align}
    &\frac{1}{(2S_{g})^{2}}\sum_{i,j=0}^{2S_g-1}
    K_{\mathrm{per}}\left(\frac{i}{S_g},\frac{j}{S_g}\right)
    \exp\left(-2\pi\sqrt{-1}\frac{ip+jq}{2S_g}\right).
    \label{Kperfinite}
\end{align}
Since evaluating \(K_{\mathrm{per}}\) itself involves an infinite series, 
we replace \(K_{\mathrm{per}}\) in \eqref{Kperfinite} with \(K_{\mathrm{circ}}\) in the actual computation.
Specifically, for \(p,q=0,\ldots,2S_g-1\), we compute
\begin{align}
    \lambda_{pq}
    &:=
    \sum_{i,j=0}^{2S_g-1}
    K_{\mathrm{circ}}(u_i,u_j)
    \exp\left(
        -2\pi\sqrt{-1}\frac{pi+qj}{2S_g}
    \right),
    \nonumber
\end{align}
and use \(\lambda_{pq}/(2S_g)^2\) as the discrete approximation to
\(\widehat K_{\mathrm{per}}(p,q)\).
Consequently, for \(i,j=0,\ldots,S_g-1\), we compute
\begin{align}
    \tilde{\xi}_{\mathrm{circ}}(u_i,u_j)
    =
    \operatorname{Re}\left(
    \sum_{p,q=0}^{2S_g-1}
    \frac{\sqrt{\lambda_{pq}}}{2S_g}
    \tilde z_{pq}
    \exp\left(2\pi\sqrt{-1}\frac{pi+qj}{2S_g}\right)
    \right),
    \nonumber
\end{align}
and use the process $\tilde{\xi}_{\mathrm{circ}}:=
    \left\{
    \tilde{\xi}_{\mathrm{circ}}(u_i,u_j)
    \right\}_{0\leq i,j\leq S_g-1}$ in the numerical experiments. 
    To accelerate the computation of these finite sums, we use the fast Fourier transform (FFT) of \citet{CooleyTukey1965} to compute 
\(\{\lambda_{pq}\}_{p,q=0}^{2S_g-1}\). 
We then compute 
\(\{\tilde{\xi}_{\mathrm{circ}}(u_i,u_j)\}_{i,j=0}^{2S_g-1}\) from
\(\{\lambda_{pq}\}_{p,q=0}^{2S_g-1}\) and
\(\{\tilde z_{pq}\}_{p,q=0}^{2S_g-1}\) using the inverse fast Fourier transform (IFFT) of \citet{CooleyTukey1965}, and retain the sub-array indexed by \(0\leq i,j\leq S_g-1\).

The overall procedure for generating the data is summarized as follows:
\begin{enumerate}
    \item Set the initial value to $Z_{1}(\mathfrak{u}_{i},\mathfrak{u}_{j}) = 0$ for any $0 \leq i,j \leq S_{g}-1$.
    \item Given the value 
    $Z_{t-1} = \left\{Z_{t-1}(\mathfrak{u}_{i},\mathfrak{u}_{j})\right\}_{0 \leq i,j \leq S_{g}-1} \in \mathbb{R}^{S_{g} \times S_{g}}$, 
    update 
    $Z_{t} = \left\{Z_{t}(\mathfrak{u}_{i},\mathfrak{u}_{j})\right\}_{0 \leq i,j \leq S_{g}-1}$ as follows:
    \begin{enumerate}
        \item Compute 
        $\left\{\Psi(Z_{t-1})(\mathfrak{u}_{i},\mathfrak{u}_{j})\right\}_{0 \leq i,j \leq S_{g}-1} \in \mathbb{R}^{S_{g} \times S_{g}}$.
        
        \item Generate the Gaussian process 
        $\tilde{\xi}_{\mathrm{circ}}
        = \left\{\tilde{\xi}_{\mathrm{circ}}(\mathfrak{u}_{i},\mathfrak{u}_{j})\right\}_{0 \leq i,j \leq S_{g}-1}
        \in \mathbb{R}^{S_{g} \times S_{g}}$.
        
        \item Update 
        \[
            Z_{t}(\mathfrak{u}_{i},\mathfrak{u}_{j})
            := \Psi(Z_{t-1})(\mathfrak{u}_{i},\mathfrak{u}_{j})
               + \tilde{\xi}_{\mathrm{circ}}(\mathfrak{u}_{i},\mathfrak{u}_{j}),
            \qquad 0 \leq i,j \leq S_{g}-1.
        \]
    \end{enumerate}
    
    \item For \(t=2,\ldots,T+500\), repeat Step 2 to obtain
\(\{Z_t\}_{t=1}^{T+500}\).
The first 500 observations, \(\{Z_t\}_{t=1}^{500}\), are discarded as burn-in, and the remaining sequence is used as the sample.
    
    \item Repeat Steps~1--3 independently $B$ times.
\end{enumerate}

The process 
$\left\{\tilde{\xi}_{\mathrm{circ}}(\mathfrak{u}_{i},\mathfrak{u}_{j})\right\}_{0 \leq i,j \leq S_{g}-1}$ 
is generated as follows:
\begin{enumerate}
    \item For $p,q = 0,\dots,2S_{g}-1$, compute 
    % $\lambda_{pq} = \widehat{K}_{\mathrm{circ}}(p,q)$ 
    $\lambda_{pq}$ via FFT.  
    \item For $0 \leq i,j \leq 2S_{g}-1$, compute via IFFT:
    \[
        \tilde{\xi}_{\mathrm{circ}}(\mathfrak{u}_{i},\mathfrak{u}_{j})
        = \mathrm{Re}\left(\sum_{p,q=0}^{2S_{g}-1}
          \frac{\sqrt{\lambda_{pq}}}{2S_{g}} \tilde{z}_{pq}
          \exp\!\left( 2\pi \sqrt{-1}\,\frac{p\mathfrak{u}_{i} + q\mathfrak{u}_{j}}{2} \right) \right).
    \]
    
    \item Extract the sub-array 
    $\left\{\tilde{\xi}_{\mathrm{circ}}(\mathfrak{u}_{i},\mathfrak{u}_{j})\right\}_{0 \leq i,j \leq S_{g}-1}$ 
    from the computed  
    $\left\{\tilde{\xi}_{\mathrm{circ}}(\mathfrak{u}_{i},\mathfrak{u}_{j})\right\}_{0 \leq i,j \leq 2S_{g}-1}$ 
    and return this as the final noise field.
\end{enumerate}
Following the above procedure, we generate the simulated data with $S_{g}=100$.

We now describe the DNN training procedure used in the numerical experiments in Section~5.
The notation in this section follows that in Section~5.
As noted there, we compute the DNN estimator $\tilde{\psi}_{T}^{(b)}$ defined by
\begin{align}
    &\tilde{\psi}_{T}^{(b)} := \arg\min_{\psi \in \mathcal{F}_{p_{0},p_{L+1}}}\frac{1}{25^{2}(T-1)}\sum_{t = 1}^{T-1}\sum_{i,j = 0}^{24}
    \bigg\{ Z_{t+1}^{(b)}\left(\frac{i}{25},\frac{j}{25}\right) - \tilde{\Psi}_{\psi}(Z_{t}^{(b)})\left(\frac{i}{25},\frac{j}{25}\right)\bigg\}^{2}. \label{DNN_opt}
\end{align}
We recall that $\{Z_{t}^{(b)}\}_{t = 1}^{T}$ is obtained by downsampling the $100\times100$-dimensional time series generated according to \eqref{discrete} to a $25\times25$ grid.
We set the network architecture $\mathcal{F}_{p_{0},p_{L+1}}$ as $L=5$, $p_{0} = 5, p_{6} = 1$, and $p_1=p_2=p_3=p_{4}= p_{5} = 32$ with the ReLU activation function $\sigma(x)=\max\{x,0\}$. 
The network weights were trained by Adam (\citealp{kingma2015adam}) with learning rate $10^{-3}$ and minibatch size of 128.
To avoid overfitting, we employ the following early stopping scheme.
Let $E\in\mathbb{N}$ be the number of epochs and $P_{\mathrm{ES}}\in\mathbb{Z}_{\ge 0}$ the patience parameter, with $P_{\mathrm{ES}}\le E$.
Let $\tilde{\psi}_{T}^{(b,e)}$ denote the network obtained after training for epoch $e$.
For initialization at $e=0$, we define the bias vectors $\bm{b}_{1},\dots,\bm{b}_{6}$ as zero vectors and, for $i=1,\dots,6$, draw each entry of each weight matrix $W_i$ independently from the uniform distribution on
$\left[-\sqrt{6/(n_{i,\mathrm{in}} + n_{i,\mathrm{out}})},\, \sqrt{6/(n_{i,\mathrm{in}} + n_{i,\mathrm{out}})}\right]$, where $n_{i,\mathrm{in}}$ and $n_{i,\mathrm{out}}$ denote the input and output dimensions of the $i$ th layer, respectively.
% We split the sample $\{Z_{t}^{(b)}\}_{t=1}^{T}$ into
% $\{Z_{t}^{(b)}\}_{t=1}^{T_{\mathrm{train}}}$ and $\{Z_{t}^{(b)}\}_{t=T_{\mathrm{train}}+1}^{T}$, where
% $T_{\mathrm{train}}$ is the integer satisfying $T_{\mathrm{train}}\le \frac{4}{5}T < T_{\mathrm{train}}+1$.
Before training, we split the sample \(\{Z_t^{(b)}\}_{t=1}^{T}\) into a training sample
\(\{Z_t^{(b)}\}_{t=1}^{T_{\mathrm{train}}}\) and a validation sample
\(\{Z_t^{(b)}\}_{t=T_{\mathrm{train}}+1}^{T}\), where
\(T_{\mathrm{train}}:=\lfloor 4T/5\rfloor\).
For each epoch \(e=1,\ldots,E\), we update the network weights using the training sample.
Given $\{Z_{t}^{(b)}\}_{t=1}^{T_{\mathrm{train}}}$ and the previous-epoch estimator $\tilde{\psi}_{T}^{(b,e-1)}$, we update the weights according to Adam (\citet{kingma2015adam}) with the loss function defined by
\begin{align}
    &L(\psi) := \frac{1}{25^{2}(T_{\mathrm{train}}-1)}\sum_{t = 1}^{T_{\mathrm{train}-1}}\sum_{i,j = 0}^{24}\bigg\{ Z_{t+1}^{(b)}\left(\frac{i}{25},\frac{j}{25}\right)
    - \tilde{\Psi}_{\psi}({Z}_{t}^{(b)})\left(\frac{i}{25},\frac{j}{25}\right)\bigg\}^{2},\quad \psi \in \mathcal{F}_{p_{0},p_{L+1}}. \label{DNN_opt_2}
\end{align}
We then evaluate the validation error on $\left\{Z_{t}^{(b)}\right\}_{t=T_{\mathrm{train}}+1}^{T}$ via
\begin{align}
    &\frac{1}{25^{2}(T - T_{\mathrm{train}} - 1)}\sum_{t = T_{\mathrm{train}} + 1}^{T-1}\sum_{i,j = 0}^{24}\bigg\{ Z_{t+1}^{(b)}\left(\frac{i}{25},\frac{j}{25}\right) - \tilde{\Psi}_{\tilde{\psi}_{T}^{(b,e)}}\left(Z_{t}^{(b)}\right)\left(\frac{i}{25},\frac{j}{25}\right)\bigg\}^{2}. \label{DNN_opt_3}
\end{align}
The procedure is iterated for at most $E$ epochs and terminated early if the validation error fails to improve for $P_{\mathrm{ES}}$ consecutive epochs.
Let $E'$ be the resulting number of epochs at termination. In this experiment, we take $\tilde{\psi}_{T}^{(b,E')}$ as the DNN estimator in place of $\tilde{\psi}_{T}^{(b)}$ in \eqref{DNN_opt}. In our numerical experiments, we run the above training procedure with $E=200$ and $P_{\mathrm{ES}}=20$.

Finally, we describe how to compute, for each $t$, the sum
\[
\frac{1}{25^{2}}\sum_{i,j = 0}^{24}\bigg\{ Z_{t+1}^{(b)}\!\left(\frac{i}{25},\frac{j}{25}\right)
    - \tilde{\Psi}_{\psi}({Z}_{t}^{(b)})\!\left(\frac{i}{25},\frac{j}{25}\right)\bigg\}^{2},
\]
as well as how to evaluate $\tilde{\Psi}_{\psi}(Z_{t}^{(b)})$.
A naive implementation of the former via nested loops is computationally expensive; instead, we use matrix operations.
Specifically, given $Z_{t+1}^{(b)}\in\mathbb{R}^{25\times 25}$ and $\tilde{\Psi}_{\psi}(Z_{t}^{(b)})\in\mathbb{R}^{25\times 25}$, define
\[
A:=\left\{\left| Z_{t+1}^{(b)}\!\left(\frac{i}{25},\frac{j}{25}\right)-\tilde{\Psi}_{\psi}(Z_{t}^{(b)})\!\left(\frac{i}{25},\frac{j}{25}\right)\right|^{2}\right\}_{0\le i,j\le 24}.
\]
Then the sum is computed as $\mathrm{vec}(A)^{\top}\bm{w}$, where $\bm{w}:=\frac{1}{25^{2}}\bm{1}_{25^{2}}\in\mathbb{R}^{25^{2}}$ and
$\bm{1}_{25^{2}}:=[\underbrace{1,\dots,1}_{25^{2}}]^{\top}\in\mathbb{R}^{25^{2}}$.
For the evaluation of $\tilde{\Psi}_{\psi}(Z_{t}^{(b)})$, it would likewise be desirable to express the computation in terms of vector--matrix operations. 
However, applying this approach to both quantities is prohibitive in terms of memory. 
We therefore approximate $\tilde{\Psi}_{\psi}(Z_{t}^{(b)})$ via Monte Carlo sampling. 
Specifically, fix an integer $S_{\mathrm{MC}}<25^{2}$. With replacement, we draw $S_{\mathrm{MC}}$ points
$
\{(v_{0,1},v_{0,2}), (v_{1,1},v_{1,2}), \dots, (v_{S_{\mathrm{MC}}-1,1},v_{S_{\mathrm{MC}}-1,2})\}
$
uniformly at random from $\big\{0,\tfrac{1}{25},\dots,\tfrac{24}{25}\big\}^{2}$. 
We then compute $\tilde{\Psi}_{\psi}(Z)\left(\frac{i}{25},\frac{j}{25}\right), 0 \leq i,j \leq 24$ by
\begin{align}
    \frac{1}{S_{\mathrm{MC}}}\sum_{s = 0}^{S_{\mathrm{MC}}-1}\psi\left(\frac{i}{25},\frac{j}{25},v_{s,1},v_{s,2},Z(v_{s,1},v_{s,2})\right). \nonumber
\end{align}
By carrying out the above steps, we numerically compute the estimator defined in \eqref{DNN_opt}. In this study, we set $S_{\mathrm{MC}} = 500$. 
\end{appendix}
%%%%%%%%%%%%%%%%%%%%%%%%%%%%%%%%%%%%%%%%%%%%%%
%% Support information, if any,             %%
%% should be provided in the                %%
%% Acknowledgements section.                %%
%%%%%%%%%%%%%%%%%%%%%%%%%%%%%%%%%%%%%%%%%%%%%%
\begin{acks}[Acknowledgments]
The authors would like to thank Daisuke Kurisu for his helpful comments.
\end{acks}
%%%%%%%%%%%%%%%%%%%%%%%%%%%%%%%%%%%%%%%%%%%%%%
%% Funding information, if any,             %%
%% should be provided in the                %%
%% funding section.                         %%
%%%%%%%%%%%%%%%%%%%%%%%%%%%%%%%%%%%%%%%%%%%%%%
\begin{funding}
The first author was supported by JST BOOST (JPMJBS2502).
The second author was supported in part by JSPS KAKENHI Grants (JP24K14855, JP25K15032, JP25K21806).
\end{funding}
\bibliographystyle{imsart-nameyear} % Style BST file (imsart-number.bst or imsart-nameyear.bst)
\bibliography{bibliography.bib}       % Bibliography file (usually '*.bib')

@book{bosq2000linear,
  author    = {Bosq, D.},
  title     = {Linear processes in function spaces: theory and applications},
  publisher = {Springer},
  year      = {2000},
  volume    = {149}
}

@book{KokoszkaReimherr2017,
  title     = {Introduction to Functional Data Analysis},
  author    = {Kokoszka, Piotr and Reimherr, Matthew},
  year      = {2017},
  publisher = {CRC Press}
}

@article{ramsay1982when,
  author  = {Ramsay, J. O.},
  title   = {When the data are functions},
  journal = {Psychometrika},
  volume  = {47},
  number  = {4},
  pages   = {379--396},
  year    = {1982}
}

@article{chen2000geometric,
  author = {Chen, M. and Chen, G.},
  title = {Geometric ergodicity of nonlinear autoregressive models with changing conditional variances},
  journal = {Canadian Journal of Statistics},
  volume = {28},
  number = {3},
  pages = {605--614},
  year = {2000}
}

@article{davydov1973mixing,
  author = {Davydov, Yu. A.},
  title = {Mixing conditions for {M}arkov chains},
  journal = {Theory of Probability and its Applications},
  volume = {18},
  number = {2},
  pages = {312--328},
  year = {1973}
}

@article{delapena2004self,
  author = {de la Pe{\~n}a, V. H. and Klass, M. J. and Lai, T. L.},
  title = {Self-normalized processes: Exponential inequalities, moment bounds and iterated logarithm laws},
  journal = {Annals of Probability},
  volume = {32},
  number = {3},
  pages = {1902--1933},
  year = {2004}
}

@inproceedings{Giorgobiani2019SubGaussian,
  author = {Giorgobiani, G. and Kvaratskhelia, V. and Tarieladze, V.},
  title = {Notes on sub-{G}aussian random elements},
  booktitle = {International Conference on Applications of Mathematics and Informatics in Natural Sciences and Engineering},
  pages = {197--203},
  publisher = {Springer International Publishing},
  address = {Cham},
  year = {2019}
}

@inproceedings{hairer2011harris,
  author    = {Hairer, M. and Mattingly, J. C.},
  title     = {Yet another look at Harris' ergodic theorem for Markov chains},
  booktitle = {Seminar on Stochastic Analysis, Random Fields and Applications VI},
  pages     = {109--117},
  publisher = {Springer},
  address   = {Basel},
  year      = {2011}
}

@article{Li2021FNO,
  author = {Li, Zongyi and Kovachki, Nikola and Azizzadenesheli, Kamyar and Liu, Burigede and Bhattacharya, Kaushik and Stuart, Andrew and Anandkumar, Anima},
  title = {Fourier neural operator for parametric partial differential equations},
  journal = {Journal of Machine Learning Research},
  volume = {24},
  pages = {1--58},
  year = {2023}
}

@article{pathak2022fourcastnet,
  title     = {FourCastNet: A Global Data-driven High-resolution Weather Model 
               using Adaptive Fourier Neural Operators},
  author    = {Pathak, Jaideep and 
               Subramanian, Shashank and 
               Harrington, Peter and 
               Raja, Sanjeev and 
               Chattopadhyay, Ashesh and 
               Mardani, Morteza and 
               Kurth, Thorsten and 
               Hall, David and 
               Li, Zongyi and 
               Azizzadenesheli, Kamyar and 
               Hassanzadeh, Pedram and 
               Kashinath, Karthik and 
               Anandkumar, Animashree},
  journal   = {arXiv preprint arXiv:2202.11214},
  year      = {2022},
  url       = {https://arxiv.org/abs/2202.11214}
}

@article{choi2024applications,
  author  = {Choi, B. J. and Jin, H. S. and Lkhagvasuren, B.},
  title   = {Applications of the Fourier neural operator in a regional ocean modeling and prediction},
  journal = {Frontiers in Marine Science},
  volume  = {11},
  pages   = {1383997},
  year    = {2024}
}

@article{subedi2025operator,
  author  = {Subedi, U. and Tewari, A.},
  title   = {Operator Learning: A Statistical Perspective},
  journal = {Annual Review of Statistics and Its Application},
  volume  = {13},
  year    = {2025}
}

@article{liu2024deep,
  title   = {Deep nonparametric estimation of operators between infinite dimensional spaces},
  author  = {Liu, Hongjie and Yang, Haizhao and Chen, Minshuo and Zhao, Tuo and Liao, Wenjing},
  journal = {Journal of Machine Learning Research},
  volume  = {25},
  number  = {24},
  pages   = {1--67},
  year    = {2024}
}

@article{Kissas2022Learning,
  author  = {Kissas, G. and Seidman, J. H. and Guilhoto, L. F. and Preciado, V. M. and Pappas, G. J. and Perdikaris, P.},
  title   = {Learning operators with coupled attention},
  journal = {Journal of Machine Learning Research},
  volume  = {23},
  number  = {215},
  pages   = {1--63},
  year    = {2022}
}

@article{lu2021deeponet,
  author  = {Lu, L. and Jin, P. and Pang, G. and Zhang, Z. and Karniadakis, G. E.},
  title   = {Learning nonlinear operators via {DeepONet} based on the universal approximation theorem of operators},
  journal = {Nature Machine Intelligence},
  volume  = {3},
  number  = {3},
  pages   = {218--229},
  year    = {2021}
}

@article{DeVore2021NeuralNetworkApproximation,
  author  = {DeVore, Ronald and Hanin, Boris and Petrova, Guergana},
  title   = {Neural network approximation},
  journal = {Acta Numerica},
  volume  = {30},
  pages   = {327--444},
  year    = {2021}
}

@article{Liu2025DeepNeuralNetworksAdaptive,
  author  = {Liu, H. and Cheng, J. and Liao, W.},
  title   = {Deep Neural Networks are Adaptive to Function Regularity and Data Distribution in Approximation and Estimation},
  journal = {Journal of Machine Learning Research},
  volume  = {26},
  number  = {213},
  pages   = {1--56},
  year    = {2025}
}

@article{rao1966inference2,
  author  = {Rao, M. M.},
  title   = {Inference in stochastic processes. II},
  journal = {Zeitschrift f{\"u}r Wahrscheinlichkeitstheorie und Verwandte Gebiete},
  year    = {1966},
  volume  = {5},
  number  = {4},
  pages   = {317--335}
}

@article{Zappala2024NeuralIntegralOperators,
  author  = {Zappala, E. and Fonseca, A. H. D. O. and Caro, J. O. and Moberly, A. H. and Higley, M. J. and Cardin, J. and Dijk, D. V.},
  title   = {Learning integral operators via neural integral equations},
  journal = {Nature Machine Intelligence},
  year    = {2024},
  volume  = {6},
  number  = {9},
  pages   = {1046--1062}
}

@book{horvath2012inference,
  author    = {Horv{\'a}th, L. and Kokoszka, P.},
  title     = {Inference for functional data with applications},
  publisher = {Springer},
  year      = {2012},
  volume    = {200}
}

@article{kurisu2024adaptive,
  title   = {Adaptive deep learning for nonlinear time series models},
  author  = {Kurisu, D. and Fukami, R. and Koike, Y.},
  journal = {Bernoulli},
  volume  = {31},
  number  = {1},
  pages   = {240--270},
  year    = {2025}
}

@book{lifshits2012gaussian,
  author = {Lifshits, M.},
  title = {Lectures on {G}aussian Processes},
  series = {SpringerBriefs in Mathematics},
  publisher = {Springer},
  address = {Berlin, Heidelberg},
  year = {2012},
  isbn = {978-3-642-24938-9}
}

@book{meyntweedie2009,
  author = {Meyn, S. P. and Tweedie, R. L.},
  title = {Markov Chains and Stochastic Stability},
  edition = {2nd},
  publisher = {Cambridge University Press},
  address = {Cambridge},
  year = {2009}
}

@book{rio2013inequalities,
  author    = {Rio, E.},
  title     = {Inequalities and Limit Theorems for Weakly Dependent Sequences},
  publisher = {Springer},
  year      = {2017}
}

@article{schmidt2019nonparametric,
  author = {Schmidt-Hieber, J.},
  title = {Nonparametric regression using deep neural networks with {ReLU} activation function},
  journal = {Annals of Statistics},
  volume = {48},
  number = {4},
  pages = {1875--1897},
  year = {2020}
}

@article{ZhuPolitis2017,
  author = {Zhu, T. and Politis, D.},
  title = {Kernel estimates of nonparametric functional autoregression models and their bootstrap approximation},
  journal = {Electronic Journal of Statistics},
  volume = {11},
  pages = {2876--2906},
  year = {2017}
}

@incollection{bosq1991modelization,
  author    = {Bosq, D.},
  title     = {Modelization, nonparametric estimation and prediction for continuous time processes},
  booktitle = {Nonparametric Functional Estimation and Related Topics},
  pages     = {509--529},
  publisher = {Springer},
  address   = {Dordrecht},
  year      = {1991}
}

@article{hormann2010weakly,
  title={Weakly dependent functional data},
  author={H{\"o}rmann, Siegfried and Kokoszka, Piotr},
  journal={Annals of Statistics},
  volume={38},
  number={3},
  pages={1845--1884},
  year={2010}
}

@article{kreiss1998regression,
  author  = {Kreiss, Jens-Peter and Neumann, Michael H.},
  title   = {Regression-Type Inference in Nonparametric Autoregression},
  journal = {The Annals of Statistics},
  year    = {1998},
  volume  = {26},
  number  = {4},
  pages   = {1570--1613}
}

@article{tjostheim1994nonparametric,
  author  = {Tj{\o}stheim, Dag and Auestad, Bj{\o}rn H.},
  title   = {Nonparametric Identification of Nonlinear Time Series: Projections},
  journal = {Journal of the American Statistical Association},
  year    = {1994},
  volume  = {89},
  number  = {428},
  pages   = {1398--1409}
}

@article{horvath2010testing,
  title={Testing the stability of the functional autoregressive process},
  author={Horv{\'a}th, Lajos and Hu{\v{s}}kov{\'a}, Marie and Kokoszka, Piotr},
  journal={Journal of Multivariate Analysis},
  volume={101},
  number={2},
  pages={352--367},
  year={2010}
}

@article{kokoszka2013determining,
  title={Determining the order of the functional autoregressive model},
  author={Kokoszka, Piotr and Reimherr, Matthew},
  journal={Journal of Time Series Analysis},
  volume={34},
  number={1},
  pages={116--129},
  year={2013}
}

@article{AntoniadisSapatinas2003,
  author  = {Antoniadis, A. and Sapatinas, T.},
  title   = {Wavelet methods for continuous-time prediction using Hilbert-valued autoregressive processes},
  journal = {Journal of Multivariate Analysis},
  year    = {2003},
  volume  = {87},
  number  = {1},
  pages   = {133--158}
}

@article{KarginOnatski2008,
  author  = {Kargin, Vladislav and Onatski, Alexei},
  title   = {Curve forecasting by functional autoregression},
  journal = {Journal of Multivariate Analysis},
  year    = {2008},
  volume  = {99},
  number  = {10},
  pages   = {2508--2526}
}

@article{Hairer2002ExponentialMixing,
  author    = {Hairer, Martin},
  title     = {Exponential mixing properties of stochastic {PDE}s through asymptotic coupling},
  journal   = {Probability Theory and Related Fields},
  volume    = {124},
  number    = {3},
  pages     = {345--380},
  year      = {2002},
  month     = {11},
  publisher = {Springer}
}

@article{GoldysMaslowski2005ExponentialErgodicity,
  author  = {Goldys, B. and Maslowski, B.},
  title   = {Exponential ergodicity for stochastic Burgers and 2D Navier--Stokes equations},
  journal = {Journal of Functional Analysis},
  volume  = {226},
  number  = {1},
  pages   = {230--255},
  year    = {2005},
  month   = {9}
}

@article{Reinhardt2024Operator,
  author        = {Reinhardt, N. and Wang, S. and Zech, J.},
  title         = {Statistical learning theory for neural operators},
  journal       = {arXiv preprint arXiv:2412.17582},
  year          = {2024},
  eprint        = {2412.17582},
  archivePrefix = {arXiv},
  primaryClass  = {cs.LG}
}

@inproceedings{kingma2015adam,
  author    = {Kingma, Diederik P. and Ba, Jimmy Lei},
  title     = {Adam: A Method for Stochastic Optimization},
  booktitle = {3rd International Conference on Learning Representations (ICLR 2015)},
  editor    = {Bengio, Yoshua and LeCun, Yann},
  address   = {San Diego, CA, USA},
  month     = may,
  year      = {2015},
  note      = {Conference Track Proceedings}
}

@article{CooleyTukey1965,
  author  = {Cooley, J. W. and Tukey, J. W.},
  title   = {An algorithm for the machine calculation of complex {F}ourier series},
  journal = {Mathematics of Computation},
  volume  = {19},
  number  = {90},
  pages   = {297--301},
  year    = {1965}
}

@article{DietrichNewsam1997,
  author  = {Dietrich, C. R. and Newsam, G. N.},
  title   = {Fast and exact simulation of stationary {G}aussian processes through circulant embedding of the covariance matrix},
  journal = {SIAM Journal on Scientific Computing},
  volume  = {18},
  number  = {4},
  pages   = {1088--1107},
  year    = {1997}
}

@book{meyn_tweedie_2009,
  author    = {Meyn, S. P. and Tweedie, R. L.},
  title     = {Markov Chains and Stochastic Stability},
  edition   = {2},
  publisher = {Cambridge University Press},
  address   = {Cambridge},
  year      = {2009}
}

@article{OhnKim2022,
  author  = {Ohn, I. and Kim, Y.},
  title   = {Nonconvex sparse regularization for deep neural networks and its optimality},
  journal = {Neural Computation},
  volume  = {34},
  number  = {2},
  pages   = {476--517},
  year    = {2022}
}

@book{vdvaart1996weak,
  author    = {van der Vaart, A. W. and Wellner, J. A.},
  title     = {Weak Convergence and Empirical Processes with Applications to Statistics},
  publisher = {Springer},
  year      = {1996}
}

%% or include bibliography directly:
% \begin{thebibliography}{}
% \bibitem{b1}
% \end{thebibliography}

\end{document}